\documentclass[11pt]{article}
\usepackage[includeheadfoot,margin=1in]{geometry}
\usepackage{times}
\usepackage{setspace}
\usepackage{amsmath,amsfonts,amssymb,amsthm}
\usepackage{mathtools}
\usepackage{enumerate}
\usepackage{verbatim}
\usepackage{commath}
\usepackage{upgreek}
\usepackage{bm}
    \usepackage[usenames]{color}
    \definecolor{plum}  {rgb}{.4,0,.4}
    \definecolor{BrickRed} {rgb}{0.6,0,0}
	\definecolor{DarkBlue} {rgb}{0,0,0.6}
\usepackage[plainpages=false,pdfpagelabels,colorlinks=true,linkcolor=BrickRed,citecolor=plum,urlcolor=DarkBlue]{hyperref}
\usepackage[mathscr]{euscript}
\usepackage[round]{natbib}
\usepackage{cleveref}

\def\ddefloop#1{\ifx\ddefloop#1\else\ddef{#1}\expandafter\ddefloop\fi}

\def\ddef#1{\expandafter\def\csname b#1\endcsname{\ensuremath{\mathbb{#1}}}}
\ddefloop ABCDEFGHIJKLMNOPQRSTUVWXYZ\ddefloop

\def\ddef#1{\expandafter\def\csname bf#1\endcsname{\ensuremath{\mathbf{#1}}}}
\ddefloop ABCDEFGHIJKLMNOPpQRSTUuVvWXYZz\ddefloop

\def\ddef#1{\expandafter\def\csname c#1\endcsname{\ensuremath{\mathcal{#1}}}}
\ddefloop ABCDEFGHIJKLMNOPQRSTUVWXYZ\ddefloop

\def\ddef#1{\expandafter\def\csname s#1\endcsname{\ensuremath{\mathsf{#1}}}}
\ddefloop ABCDEFGHIJKLMmNOPQRSTUuVWXYZ\ddefloop

\def\ddef#1{\expandafter\def\csname e#1\endcsname{\ensuremath{\mathscr{#1}}}}
\ddefloop ABCDEFGHIJKLMNOPQRSTUVWXYZ\ddefloop

\def\ddef#1{\expandafter\def\csname f#1\endcsname{\ensuremath{\mathfrak{#1}}}}
\ddefloop ABCDEFGgHIJKLMNOPQRSTUVWXYZ\ddefloop

\def\ddef#1{\expandafter\def\csname r#1\endcsname{\ensuremath{\mathrm{#1}}}}
\ddefloop ABCDEFGHIJKLMNOPQRSTUVWXYZ\ddefloop
\def\Reals{{\mathbb R}}

\def\p{{\partial}} % partial derivative
\def\Ex{{\mathbf E}} % Exation
\def\Pr{{\mathbf P}} % probability
\def\eps{\varepsilon}

\DeclareMathOperator*{\diam}{diam}
\DeclareMathOperator*{\Var}{Var}
\DeclareMathOperator*{\Cov}{Cov}

\let\set\relax   % undefine commath's \set
\let\inner\relax
\let\norm\relax
\let\abs\relax

\DeclarePairedDelimiter{\set}{\{}{\}}
\DeclarePairedDelimiter{\norm}{\lVert}{\rVert}
\DeclarePairedDelimiter{\abs}{\lvert}{\rvert}
\DeclarePairedDelimiter{\floor}{\lfloor}{\rfloor}

\DeclarePairedDelimiterX{\inner}[2]{\langle}{\rangle}{#1,\,#2}

\newcommand{\deq}{\coloneqq}
\newcommand{\eqd}{\eqqcolon}

\newtheorem{theorem}{Theorem}
\newtheorem{definition}{Definition}

\newtheorem{proposition}{Proposition}
\newtheorem{lemma}{Lemma}
\newtheorem{corollary}{Corollary}
\newtheorem{remark}{Remark}

\allowdisplaybreaks

\begin{document}
\title{Upper and Lower Bounds on Expected Soft Maxima\\of Gaussian 
Processes}
\author{Yifeng Chu\thanks{University of Illinois, Urbana, IL 61801, USA.}\\
\href{mailto:ychu26@illinois.edu}{ychu26@illinois.edu} \and Maxim Raginsky\thanks{University of Illinois, Urbana, IL 61801, USA.}\\ \href{mailto:maxim@illinois.edu}{maxim@illinois.edu}}
\date{}
\maketitle
\begin{abstract}
We obtain upper and lower bounds for ``smoothed'' versions of the expected supremum of centered 
Gaussian processes with finite or countable index sets. These so-called \textit{soft maxima} are computed in terms of expected values of random Gibbs averages at inverse temperature $\beta > 0$ and reduce to expected suprema in the zero-temperature limit $\beta \to \infty$. Our analysis builds on ideas from statistical physics and information theory, and relies crucially 
on the tensorization technique introduced recently by \citet{liuSimpleSharpGeneralization2025}.
The bounds retain the same multiscale structure as in the expressions for the expected supremum derived using the method of generic chaining, with a truncation term governed
by the inverse temperature $\beta$. In the zero-temperature limit, we recover the majorizing measure theorem. As an illustrative example, we apply our results to the analysis of the quenched free energy 
in the Sherrington--Kirkpatrick model and obtain a Parisi formula in the finite system size setting. 
\end{abstract}

\tableofcontents

\section{Introduction and statement of main results}

Let $X = (X_t)_{t \in T}$ be a centered Gaussian process, where the index set $T$ is finite or countable. The method of generic chaining \citep{talagrand_2014} gives sharp upper and lower bounds on the expected supremum of $X$ in terms of the geometry of the index set $T$ induced by the canonical metric $d(s,t) := (\Ex|X_s - X_t|^2)^{1/2}$. The key idea is that, for finite $T$, $\sqrt{\log |T|}$ is a good proxy for $\Ex[\sup_{t \in T}X_t]$, in the sense that\footnote{We will use the following notation throughout the paper: The letters $c,C$ will denote universal constants. We will write $x \lesssim  y$ 
(resp.\ $x \gtrsim y$) if
$x \le Cy$ (resp.\ $x \ge c y$) for some universal constant $c,C >0$ 
and $x \asymp y$ if $x \gtrsim y$ and $x \lesssim y$.
Finally, $x \wedge y \deq \min(x,y)$, $x \vee y \deq \max(x,y)$, and $(x)_+ \deq x \vee 0$.}
\begin{align}\label{eq:finite_T}
	a \sqrt{\log |T|} \lesssim \Ex\Big[\sup_{t \in T}X_t\Big] \lesssim \Delta\sqrt{\log |T|},
\end{align}
where $a = \min \{d(s,t) : s,t \in T, s \neq t\}$ is the minimum separation between the elements of $T$ and $\Delta = \max \{d(s,t) : s,t \in T\}$ is the diameter of $T$. The upper bound in \eqref{eq:finite_T} is by the well-known maximal inequality for Gaussian random variables, while the lower bound follows from Sudakov minoration. Another important point is that the centered Gaussian random variables $X_s$ and $X_t$ are highly correlated when $d(s,t)$ is very small and nearly independent when $d(s,t)$ is very large.  The method of generic chaining is a multiscale iterative application of these ideas, where at each scale $n=0,1,2,\dots$ one approximates the process $X$ by its ``subsampled'' version on a finite set of indices $T_n \subseteq T$ satisfying the cardinality bound $|T_n| < 2^{2^n}$. 

While the expected supremum is a natural measure of the size of $X$, various ``smoothed'' versions of it are of interest as well, arising in connection with concentration and superconcentration phenomena in Gaussian random fields \citep{Chatterjee_2014} or statistical physics of disordered systems \citep{Talagrand_MF,panchenko_sherrington-kirkpatrick_2013}. Consider again the case when $T$ is finite. Let $\mu$ be a probability measure on $T$ that assigns positive mass to every singleton. Then $\max_{t \in T} X_t$ can be bounded from above and from below, for every $\beta > 0$, as follows:\footnote{Here, and everywhere throughout the paper, $\log$ denotes natural logarithms.}
\begin{align}\label{eq:softmax_bounds}
	\begin{split}
	\max_{t \in T} X_t - \frac{1}{\beta} \log \max_{t \in T}\frac{1}{\mu(t)}\le \frac{1}{\beta}\log \sum_{t \in T}\mu(t) e^{\beta X_t} \le \max_{t \in T} X_t; \\
	\max_{t \in T} X_t - \frac{1}{\beta} \log \max_{t \in T}\frac{1}{\mu(t)} \le \frac{\sum_{t \in T}\mu(t) X_t e^{\beta X_t}}{\sum_{t \in T}\mu(t) e^{\beta X_t}} \le \max_{t \in T} X_t.
	\end{split}
\end{align}
Moreover, as we increase the parameter $\beta > 0$, the two quantities in the middle of the inequalities in \eqref{eq:softmax_bounds} vary from the average value $\sum_{t \in T}\mu(t) X_t$ at $\beta = 0$ to the maximum $\max_{t \in T} X_t$ in the limit as $\beta \to \infty$. In this paper, we obtain sharp upper and lower bounds on the expected values of such ``soft maxima" of Gaussian processes. These bounds can be seen as the natural analogues of the chaining-type bounds for the expected supremum \citep{talagrand_2014} or their earlier incarnations dating back to the seminal works of \citet{Fernique_1975,Fernique_1978} and \citet{Talagrand_1987,Talagrand_1992}. In addition to the probabilistic tools that are commonly used in this literature, our analysis draws heavily on ideas from statistical physics and from information theory. In particular, we make fundamental use of the tensorization technique introduced recently by \citet{liuSimpleSharpGeneralization2025} in his elegant information-theoretic approach to the study of empirical processes. 

%%%%%%%%%%%%%%%%%%%%%%%%%%%%%%%%%%%%

\paragraph{Gaussian processes and random Gibbs measures.} We consider centered Gaussian processes  $(X_t)_{t \in T}$, where the index set $T$ is either finite or countable. We will often treat such a process as a random element $X$ of the space $\Reals^T$ of real-valued functions on $T$, and will denote its probability law by $P_X$. We will make use of the following definition of stationarity \citep{Fernique_1975}:

\begin{definition}\label{def:stationary} We say that $X$ is a {\em stationary process} if there exists a group ${\mathbb G}$ acting transitively on $T$, such that $d(g(s),g(t)) = d(s,t)$ for all $g \in {\mathbb G}$ and all $s,t \in T$, where
\begin{align}\label{eq:canonical_metric}
	d(s,t) \deq \big( \Ex (X_s - X_t)^2\big)^{1/2}
\end{align}
is the canonical metric on $T$ associated with $X$.\end{definition}

\noindent Let $\eP(T)$ denote the space of probability measures on $T$. Given a constant $\beta > 0$, a probability measure $\mu \in \eP(T)$, and a Gaussian process $X$ on $T$, we denote by $\sm_{\beta,\mu}$ the random element of $\eP(T)$ given by\footnote{For any $F \in \Reals^T$, $\langle \mu, F\rangle \deq \sum_{t \in T}\mu(t)F_t$ denotes the expected value of $F$ with respect to $\mu$.}
\begin{align}
	\frac{\dif \sm_{\beta,\mu}}{\dif \mu} = \frac{e^{\beta X}}{\inner{\mu} {e^{\beta X}}}
\end{align}
and by $\uptau_{\beta,\mu}$ the random element of $T$ whose (regular) conditional probability law given $X$ is equal to $\sm_{\beta,\mu}$. The marginal probability law of $\uptau_{\beta,\mu}$ will be denoted by $\overline{\sm}_{\beta,\mu}$. We further define the quantities
\begin{align}
	\Phi_\beta(X;\mu)  &\deq \frac{1}{\beta}\log \inner{\mu}{e^{\beta X}} = \frac{1}{\beta}\log \sum_{t \in T}\mu(t) e^{\beta X_t}
\end{align}
and
\begin{align}
	\langle X \rangle_{\beta,\mu} &\deq \inner{\sm_{\beta,\mu}}{X},
\end{align}
as well as their expected values with respect to $P_X$:
\begin{align}
	\eF_\beta(P_X; \mu) &\deq \Ex \Phi_\beta(X;\mu), \\
	\eG_\beta(P_X; \mu) &\deq \Ex \langle X \rangle_{\beta,\mu}
\end{align}
With a slight abuse of notation, we will also write $\eF_\beta(X; \mu)$ and $\eG_\beta(X; \mu)$ even though these functionals depend only on the probability law of $X$.

From the viewpoint of statistical physics of disordered systems \citep{bovier_statistical_2006}, we interpret the index set $T$ as the state space of a physical system and $-X_t$ as the \textit{random} energy of the state $t \in T$. Then $\sm_{\beta,\mu}$ is the random Gibbs measure, $\Phi_\beta(X;\mu)$ is the equilibrium free energy, and $\langle X \rangle_{\beta,\mu}$ is the Gibbs average, all computed at inverse temperature $\beta$ with respect to the base measure $\mu$. The expected values $\eF_\beta(P_X;\mu)$ and $\eG_\beta(P_X;\mu)$ are referred to as the quenched equilibrium free energy and the quenched Gibbs average, respectively. 

\paragraph{Some notions from information theory.} We summarize here the key definitions and concepts from information theory that will be used throughout the paper. The reader can consult the texts of \citet{cover_thomas} or \citet{polyanskiyInformationTheoryCoding2024} for more details. Given a Borel space $\cY$,
the relative entropy (or Kullback-Leibler divergence) $D(\mu \| \nu)$
between $\mu, \nu \in \eP(\cY) $  is defined as
\begin{align}
	D(\mu \| \nu) \deq \begin{cases}\left\langle\mu,\log \displaystyle\frac{\dif\mu}{\dif \nu}\right\rangle, & \text{if } \mu \ll \nu \\
	+ \infty, & \text{otherwise}
\end{cases},
\end{align}
where $\ll$ denotes absolute continuity of measures. The relative entropy is always nonnegative, and $D(\mu \| \nu) = 0$ iff $\mu = \nu$. For a random element $(Y,Z)$ of a product space $\cY \times \cZ$ with joint law
$P_{YZ}$, the mutual information $I(Y ; Z)$ between $Y$ and $Z$ is defined as
$D(P_{YZ}\| P_Y \otimes P_Z )$. Let a metric $d$ on $T$ be given. The following definition is key:

\begin{definition}\label{def:RDF} For a fixed $\mu \in \eP(T)$, the {\em Shannon rate-distortion function} is defined as
\begin{align}\label{eq:RDF}
	R_\mu(\epsilon)& \deq \inf_{(\tau,\tau'): \, P_\tau = \mu} \left \{ I(\tau;\tau') : \Ex d^2(\tau,\tau') \le \epsilon \right\},
\end{align}
where the infimum is over random elements $(\tau,\tau')$ of $T \times T$ such that $P_\tau = \mu$.
\end{definition}
\noindent An equivalent expression for $R_\mu(\epsilon)$ is
\begin{align*}
	R_\mu(\epsilon) = \inf_{P_{\tau'|\tau}} \left \{ I(\tau;\tau') : \Ex d^2(\tau,\tau') \le \epsilon \right\},
\end{align*}
where the infimum is over all regular conditional probability laws $P_{\tau'|\tau}$ of $\tau'$ given $\tau$ with $P_\tau = \mu$. The rate-distortion function plays a key role in the work of \citet{liuSimpleSharpGeneralization2025}, and it will occupy a prominent position in our development as well.

\paragraph{Main results: the case of finite $T$.} We first consider the case when the index set $T$ is finite. Let $\Delta$ denote the diameter of the metric space $(T,d)$, where $d$ is the natural metric associated with $X$:
\begin{align*}
	\Delta = \diam(T) \deq \max \left\{ d(s,t) : s,t \in T \right\}.
\end{align*}
For any $t \in T$ and $\epsilon \ge 0$, 
\begin{align*}
	B(t,\epsilon) \deq \left\{ s \in T : d(s,t) \le \epsilon \right\}
\end{align*}
is the closed ball of radius $\epsilon$ centered at $t$. The \textit{covering number} $\sN(A,\epsilon)$ of any $A \subseteq T$ is the smallest number of such $\epsilon$-balls that cover $A$. We will write $B(t,\epsilon,d)$ and $\sN(A,\epsilon,d)$ whenever we need to indicate the metric $d$ explicitly (e.g., when we are working with several different metric structures on $T$). 

From \eqref{eq:softmax_bounds}, it is readily seen that the quantities
\begin{align*}
	\sup_{\mu \in \eP(T)} \eF_\beta(P_X; \mu) \quad \text{and} \quad \sup_{\mu \in \eP(T)} \eG_\beta(P_X; \mu)
\end{align*}
are the natural proxies for the expected supremum of $X$ at a given inverse temperature $\beta > 0$. Our first two results deal with the case when $X = (X_t)_{t \in T}$ is stationary (recall Definition~\ref{def:stationary}). Define the following function of $n \in \bN$ and $\beta > 0$:
\begin{align}\label{eq:eta}
	\eta(n,\beta) \deq \beta^2 \wedge \log n.
\end{align}
This definition encodes the difference between low-temperature (large $\beta$) and high-temperature (small $\beta$) regimes, where the breakpoint occurs at $\beta = \sqrt{\log n}$. As an illustration, let $G = (G_i)_{i \le n}$ be a vector of i.i.d.\ standard Gaussian random variables, and let $\su_n$ denote the uniform probability distribution on the set $[n] \deq \{1,\dots,n\}$. We will show in Lemma~\ref{lem:iid} that
\begin{align*}
	\frac{1}{\beta}\Ex \log \Bigg(\frac{1}{n}\sum^n_{i=1}e^{\beta G_i}\Bigg) \equiv \eF_\beta(G; \su_n) \asymp \sqrt{\eta(n,\beta)},
\end{align*}
where the linear (high-temperature) portion of $\sqrt{\eta}$ reflects the law of large numbers (LLN) behavior, i.e., for small $\beta$ and large $n$,
\begin{align*}
  \frac{1}{\beta}\Ex \log \Bigg(\frac{1}{n} \sum_{i=1}^n e^{\beta  G_i}\Bigg) 
  \approx \frac{1}{\beta} \log \Ex [e^{\beta G_1}] = \frac{\beta}{2},
\end{align*}
whereas the constant (low-temperature) portion of $\sqrt{\eta}$ reflects the fact that, as $\beta \to \infty$,  $\eF_\beta(G ; \su_n)$ saturates at 
$\Ex[\max_{1 \le i \le n} G_i]\asymp \sqrt{\log n}$. This intuition can be made precise in the rigorous analysis of the thermodynamic limit formula for the so-called \textit{Random Energy Model} or REM (see e.g. \cite{bovier_statistical_2006}). In particular, the linear (high-temperature) scaling of $\eF_\beta(G; \su_n)$ with $\beta$ when $\beta$ is small and $n$ is sufficiently large arises from the fact that the \emph{quenched free energy} $\eF_\beta(G; \su_n)$ is approximately 
equal to the \textit{annealed free energy} $\frac{1}{\beta} \log \Ex[e^{\beta G_1}]$ due to law of large numbers. (In a more general setting of limit theorems for sums of random exponentials \citep{ben_arous_2005}, the high-temperature regime corresponds to LLN-like behavior, whereas the low-temperature regime is better described by extreme value theory.)

With this motivating discussion out of the way, we can state our first main result:

\begin{theorem}\label{thm:stationary_1}
	Let $X = (X_t)_{t\in T}$ be a centered stationary Gaussian process, $\Delta = 1$.
	For a fixed $r \ge 4$ and any $\beta > 0$,
	\begin{align*}
	\sup_{\mu \in \eP(T)}\eF_\beta(P_X;\mu) \asymp \sup_{\mu \in \eP(T)}\eG_\beta(P_X;\mu) 
	\asymp \sum_{k \ge 1} r^{-k } \sqrt{\eta(n_k, \beta  r^{-k})} 
	\end{align*}
	where $n_k=\sN(B(r^{-k+1}),{r^{-k}})$ and $B(r^{-k})$ is a $r^{-k}$-ball at any center.
\end{theorem}

\noindent Our next result shows that the 
the weighted ``soft'' local metric entropy is equivalent to the weighted global
one up to some multiplicative constant.
Thus, we obtain a soft version of Fernique's theorem for stationary Gaussian processes:

\begin{theorem}\label{thm:stationary_2}
	  With the same assumption as in Theorem~\ref{thm:stationary_1},
	\begin{align*}
	\sup_{\mu \in \eP(T)}\eF_\beta(P_X;\mu) \asymp \sup_{\mu \in \eP(T)}\eG_\beta(P_X;\mu) 
	\asymp 
	\int_0^\Delta 
	\sqrt{\eta(\sN_\epsilon,\beta \epsilon)} \dif \epsilon
	\end{align*}
	where $\sN_\epsilon \deq \sN(T,\epsilon)$.
\end{theorem}
\begin{remark}\label{rmk:fix_pt} Since the functional $\mu \mapsto \eF_\beta(P_X;\mu)$ is concave and $\eP(T)$ is compact, the supremum of $\eF_\beta(P_X;\mu)$ over $\mu \in \eP(T)$ is achieved at some $\mu^*_\beta \in \eP(T)$. As shown in Lemma~\ref{lem:optimal_mu}, such a maximizing $\mu^*_\beta$ satisfies the fixed-point relation
	\begin{align*}
	\overline{\sm}_{\beta,\mu^*_\beta} = \mu^*_\beta,
\end{align*}
and the bounds of Theorems~\ref{thm:stationary_1} and \ref{thm:stationary_2} hold for $\eF_\beta(P_X; \mu^*_\beta)$ and $\eG_\beta(P_X; \mu^*_\beta)$.
\end{remark}

\noindent When we drop the assumption of stationarity, the rate-distortion function $R_\mu(\epsilon^2)$ (recall Definition~\ref{def:RDF}) emerges as the natural counterpart of $\log \sN(T,\epsilon)$, as already observed by \citet{liuSimpleSharpGeneralization2025}. Our main result for the nonstationary, finite $T$ case is as follows:

\begin{theorem}\label{thm:nonstarionary_1} Let $X = (X_t)_{t \in T}$ be a centered Gaussian process on a finite index set $T$. Then
\begin{align*}
\sup_{\mu \in \eP(T)}\eF_\beta(P_X;\mu) \asymp \sup_{\mu \in \eP(T)}\eG_\beta(P_X;\mu) 
\asymp \sup_{\mu\in \eP(T)}\int_0^\Delta \Big(\sqrt{R_{\mu}(\epsilon^2)}\wedge \beta \epsilon\Big) 
\dif \epsilon.
\end{align*}
\end{theorem}
\begin{remark}
  Similar to Remark~\ref{rmk:fix_pt}, the above bounds can be expressed as follows: 
  \begin{align*}
  \eF_{\beta}(P_X;\mu_{\beta}^*) \asymp \eG_{\beta}(P_X;\mu_{\beta}^*) \asymp \int_0^\Delta \Big(\sqrt{R_{\mu_{\beta}^*}(\epsilon^2)}\wedge \beta \epsilon\Big)
  \end{align*}
\end{remark}

\paragraph{Main results: the case of countable $T$.} When the index set $T$ is countably infinite, the supremum of $\mu \mapsto \eF_\beta(P_X;\mu)$ may not be achieved. The following theorem is the counterpart of Theorem~\ref{thm:nonstarionary_1} for the countably infinite case (we assume that $\Delta$, the diameter of $(T,d)$, is finite):

\begin{theorem}\label{thm:countable_softmmt}
Let $X = (X_t)_{t\in T}$ be a centered Gaussian process indexed by a countable set $T$, such that $\Delta < \infty$. Then
\begin{align*}
\sup_{\mu\in \eP(T)}\eG_\beta(X;\mu)
\asymp \sup_{\mu\in \eP(T)} \eF_\beta(X; \mu)
\asymp
\sup_{\mu\in \eP(T)}\int_0^\Delta \Big(\sqrt{R_{\mu}(\epsilon^2)}\wedge \beta \epsilon\Big)  \dif \epsilon.
\end{align*}
\end{theorem}

\paragraph{Recovering the expected supremum in the zero-temperature limit.} As $\beta \to \infty$, all of the quantities displayed above converge to the correct expressions for the expected supremum of $X$. For example, taking $\beta \to \infty$ in the setting of Theorem~\ref{thm:stationary_2}, we recover the well-known Dudley--Fernique characterization of the expected supremum of a stationary Gaussian process:
\begin{align*}
	\Ex\Big[\sup_{t \in T}X_t \Big] \asymp \int^\Delta_0 \sqrt{\log \sN(T,\epsilon)}\dif \epsilon.
\end{align*}
Similarly, taking $\beta \to \infty$ in Theorem~\ref{thm:nonstarionary_1} (for the nonstationary, finite $T$ case) or in Theorem~\ref{thm:countable_softmmt} (for the countably infinite $T$ case) gives the rate-distortion integral characterization due to \citet{liuSimpleSharpGeneralization2025}:
\begin{align*}
	\Ex\Big[\sup_{t \in T}X_t \Big] \asymp \sup_{\mu \in \eP(T)}\int^\Delta_0 \sqrt{R_\mu(\epsilon^2)}\dif \epsilon.
\end{align*}

\paragraph{Organization.} The remainder of the paper is organized as follows. Section~\ref{sec:prelims} is devoted to various preliminary results that are needed for the subsequent development. Section~\ref{sec:finite_T_stationary} presents the proofs of Theorems~\ref{thm:stationary_1} and \ref{thm:stationary_2} for stationary Gaussian processes indexed by a finite set, as well as an alternative derivation of an intermediate result using a symmetrization argument. Section~\ref{sec:finite_T_nonstationary} is devoted to the nonstationary, finite $T$ setting. In particular, we prove Theorem~\ref{thm:nonstarionary_1} via recourse to the tensorization technique of \citet{liuSimpleSharpGeneralization2025}, which we first adapt to the ``soft maximum'' $(\beta < \infty)$ setting. Some additional results for nonstationary processes are also proved in Section~\ref{sec:finite_T_nonstationary}, including a lower bound for a ``soft" Fernique functional of a Gaussian process and related bounds that rely on an information-theoretic concept of code-measure correspondence \citep{kontoyiannis_arbitrary_2002}. Theorem~\ref{thm:countable_softmmt} is proved in Section~\ref{sec:countable_T}, along with some related results on the soft Fernique functional. Finally, Section~\ref{sec:SK} contains an application of some of our results to a nonasymptotic analysis of the Sherrington--Kirkpatrick (SK) model \citep{panchenko_sherrington-kirkpatrick_2013}. In particular, we show that the so-called \textit{Parisi formula}, which describes the limiting behavior of the scaled quenched free energy in the SK model as the system size tends to infinity, holds up to multiplicative constants for each finite system size. Proofs of two auxiliary results from Section~\ref{sec:finite_T_nonstationary} are relegated to  Appendix~\ref{app:proofs}, and Appendix~\ref{app:misc} contains miscellaneous technical results that are invoked in a few places in the main text.

\section{Preliminaries}
\label{sec:prelims}

\subsection{Subgaussian random variables and processes}

A random variable $Z$ with $\Ex Z < \infty$ is \emph{$\sigma$-subgaussian} if the following inequality holds for any $\lambda\in \bR$:
\begin{align*}
  \Ex e^{\lambda(Z-\Ex Z)} \le e^{\lambda^2 \sigma^2 / 2}
\end{align*}
A process $X = (X_t)_{t \in T}$ indexed by the elements of a metric space $(T,d)$ is subgaussian if it satisfies the following increment condition:
\begin{align*}
	\Pr\Big[|X_s - X_t| \ge u d(s,t)\Big] \le 2e^{-u^2}, \qquad \text{for all } s,t \in T, u > 0.
\end{align*}
We state a simple but useful maximal inequality for subgaussian variables:
\begin{proposition}[subgaussian maximal inequality]
If $X_1,\ldots,X_n$ are $\sigma$-subgaussian random variables defined on a common probability space, then 
\begin{align*}
  \Ex [\max_{i\in [n]} X_i] \le \max_{i\in [n]} \Ex X_i + \sigma\sqrt{2 \log n}
\end{align*}
\end{proposition}
\noindent The following result, see e.g. \citet{panchenko_sherrington-kirkpatrick_2013}, will be used repeatedly:
\begin{lemma}
  \label{lem:concentration} Let $X = (X_t)_{t \in T}$ be a centered Gaussian process with a discrete index set $T$, such that
  $$
 \sup_{t \in T} \Var(X_t) \le \sigma^2 < \infty
 $$
 for some $\sigma > 0$. Let a finite measure $m$ on $T$ be given. Then the random variable
  \begin{align*}
  Z \deq \log \inner{m}{e^X}
  \end{align*}
  is $(\sqrt{2}\sigma)$-subgaussian, i.e., 
  \begin{align*}
  \Ex e^{\lambda(Z-\Ex Z)} \le e^{\lambda^2 \sigma^2 }, \qquad \text{for all } \lambda \in \Reals.
  \end{align*}
\end{lemma}
\noindent We will repeatedly use the concentration property of $\Psi_\beta$ and $\Phi_\beta$, derived from the above lemma.

\subsection{Gibbs measures}

Let $X = (X_t)_{t \in T}$ be a centered Gaussian process with a finite index set $T$. Let a probability measure $\mu \in \eP(T)$ and an inverse temperature $\beta > 0$ be given. Let $\uptau_{\beta,\mu}$ be a random element of $T$ whose conditional probability law given $X$ is equal to the random Gibbs measure $\sm_{\beta,\mu}$, and let $\overline{\sm}_{\beta,\mu}$ denote the marginal probability law of $\uptau_{\beta,\mu}$. (In fact, the definitions of $\Phi_\beta(X;\mu)$, $\langle X \rangle_{\beta,\mu}$, etc.\ extend straightforwardly to the case when $X$ is a random element of $\Reals^T$, such that every $X_t$ is centered and subgaussian.)

\begin{lemma}\label{lm:identities}
	The following identities hold:
	\begin{align}
	  \begin{split}\label{eq:identities}
	  D(\sm_{\beta,\mu}\|\mu) &= \beta(\langle X \rangle_{\beta,\mu} - \Phi_\beta(X ; \mu)), \\
	  \Ex D(\sm_{\beta,\mu}\|\mu)
	  &= I(X ; \uptau_{\beta,\mu}) + D(\overline{\sm}_{\beta, \mu} \| \mu).
	  \end{split}
	\end{align}
\end{lemma}
\begin{proof} We will use the shorthand notation $\uptau$, $\sm$, $\overline{\sm}$, etc. The formula for $D(\sm\|\mu)$ follows from definitions:
	\begin{align*}
		D(\sm\|\mu) &= \inner{\sm}{ \log \frac{\dif\sm}{\dif \mu}} \\
		&= \inner{\sm}{\beta X - \beta \Phi_{\beta}(X;\mu)} \\
		&= \beta  \left(\langle X \rangle_{\beta,\mu} -  \Phi_{\beta}(X;\mu)\right).
	\end{align*}
	For $\Ex D(\sm\|\mu)$, we compute the mutual information between $X$ and $\uptau$:
	\begin{align*}
		I(X; \uptau) &= \Ex \left[ \frac{\dif P_{X,\uptau}}{\dif\,(P_X \otimes P_{\uptau})}\right] \\
		&= \Ex \left[\log \frac{\dif P_{\uptau|X}}{\dif P_{\uptau}}\right] \\
		&= \int_{\Reals^T} \dif P_X \int_T \dif \sm \log \frac{\dif \sm}{\dif \overline{\sm}} \\
		&= \int_{\Reals^T} \dif P_X \left(\int_T \dif \sm \log \frac{\dif \sm}{\dif \mu} - \int_T \dif \sm \log \frac{\dif \overline{\sm}}{\dif \mu}\right) \\
		&= \Ex D(\sm \| \mu) - D(\overline{\sm} \| \mu),
	\end{align*}
	where the last step follows from the fact that $\overline{\sm}$ is the marginal law of $\uptau$ so that the identity
	\begin{align*}
		\Ex\left[\inner{\sm}{h}\right] = \inner{\overline{\sm}}{h}
	\end{align*}
	holds for any measurable function $h : T \to \Reals$.
\end{proof}

\begin{lemma}[Gibbs variational principle]\label{lm:Gibbs_VP} Let $\mu \in \eP(T)$ be given. Then, for any $\beta > 0$,
	\begin{align}\label{eq:Gibbs_VP}
		\Phi_\beta(X; \mu) = \langle \sm_{\beta,\mu}, X \rangle - \frac{1}{\beta}D(\sm_{\beta,\mu}\|\mu) = \sup_{\nu \in \eP(T)} \left[\langle \nu, X \rangle - \frac{1}{\beta}D(\nu \| \mu)\right].
	\end{align}
\end{lemma}
\begin{proof} This is a standard result in statistical physics; we give the proof for completeness. The first equality in \eqref{eq:Gibbs_VP} is verified by a straightforward computation. For the second one, take any $\nu \in \eP(T)$. We can assume without loss of generality that $\nu \ll \mu$; otherwise, $D(\nu \| \mu) = +\infty$, so we can restrict the maximization only to those $\nu$ that are absolutely continuous with respect to $\mu$. We have
	\begin{align*}
		D(\nu \| \mu) - \langle \nu, \beta X \rangle &= \left\langle \nu, \log \frac{\dif\nu}{\dif \mu} - \beta X \right\rangle \\
		&= \left \langle \nu, \log \frac{\dif \nu}{\dif \sm_{\beta,\mu}} \right\rangle - \log \langle \mu, e^{\beta X} \rangle \\
		&= D(\nu \| \sm_{\beta,\mu}) - \beta \Phi_\beta(X; \mu) \\
		& \ge - \beta \Phi_\beta(X; \mu),
	\end{align*}
	where equality holds iff $\nu = \sm_{\beta,\mu}$.
\end{proof}

\begin{lemma}[Comparison of free energies and Gibbs averages]\label{lm:equiv_soft} The functions $\beta \mapsto \langle X \rangle_{\beta,\mu}$ and $\beta \mapsto \eG_\beta(X;\mu)$ are monotonically increasing. Moreover, the following inequality holds:
	\begin{align}\label{eq:equiv_soft}
	\eF_\beta(X;\mu) \le \eG_\beta(X;\mu)
	\le 2 \eF_{2\beta} (X;\mu).
	\end{align}
\end{lemma}

\begin{proof} For $\beta > 0$,
	\begin{align*}
		\frac{\dif}{\dif \beta} \langle X \rangle_{\beta,\mu} &= \inner{\sm_{\beta,\mu}}{X^2} - (\inner{\sm_{\beta,\mu}}{X})^2 \ge 0,
	\end{align*}
This proves the monotonicity of $\beta \mapsto \langle X \rangle_{\beta,\mu}$; monotonicity of $\beta \mapsto \eG_\beta(X;\mu)$ follows upon taking expectations. Moreover,
	\begin{align*}
		\frac{\dif}{\dif\beta} \big(\beta \Phi_\beta(X;\mu)\big) &= \frac{\inner{\mu}{Xe^{\beta X}}}{\inner{\mu}{e^{\beta X}}} = \langle X \rangle_{\beta,\mu}
	\end{align*}
Using these facts and the fundamental theorem of calculus, we get
	\begin{align*}
	2\eF_{2\beta} (X; \mu)
	\ge 2\eF_{2\beta}(X;\mu)  - \eF_\beta(X;\mu)
	= \frac{1}{\beta}\int_\beta^{2\beta} \eG_\alpha(X;\mu)\dif \alpha
	\ge \eG_\beta(X;\mu). 
	\end{align*}
The nonnegativity of $\eF_\beta(X;\mu)$ follows from concavity:
\begin{align*}
	\eF_\beta(X;\mu) &= \frac{1}{\beta}\Ex \log \inner{\mu}{e^{\beta X}} \\
& \ge  \Ex \left[\inner{\mu}{X}\right] \\
&= 0,
\end{align*}
where the last step uses Fubini's theorem and the fact that $X$ is centered. This proves the second inequality in \eqref{eq:equiv_soft}. Finally, 
\begin{align*}
	0 \le \Ex D(\sm_{\beta,\mu} \| \mu) = \beta (\eG_\beta(X;\mu) - \eF_\beta(X;\mu)),
\end{align*}
which gives first inequality in \eqref{eq:equiv_soft}.
\end{proof}

Since the functional $\mu \mapsto \eF_\beta(P_X ; \mu)$ is concave and $\eP(T)$ is compact for finite $T$, there exists
a global maximizer $\mu_\beta^* \in \eP(T)$.
Moreover, any such maximizer satisfies a ``fixed-point'' relation:
\begin{lemma}
  \label{lem:optimal_mu}
For any $\mu_\beta^*$ 
that achieves the supremum of $\eF_\beta(P_X;\mu)$ over $\mu \in \eP(T)$,
\begin{align*}
\mu_\beta^* = \overline{\sm}_{\beta, \mu_\beta^*}.
\end{align*}
\end{lemma}
\begin{proof} The proof makes use of the Karush--Kuhn--Tucker (KKT) optimality conditions \citep[Sec.~5.5.3]{boyd_vanderberghe_cvx} for the finite-dimensional convex minimization problem
\begin{align*}
\text{minimize } &\phi(\mu) \deq - \Ex\,{\log \sum_{t \in T} \mu(t) e^{\beta X_t}}\\
\text{subject to }
& -\mu(t)  \le 0 \text{ for all }t \in T\\
& \sum_{t \in T} \mu(t)-1 = 0.
\end{align*}
The Lagrangian for this problem has form
\begin{align*}
	\eL(\mu,v,u) \deq \phi(\mu) - \sum_{t \in T} v(t) \mu(t) + u \Big(\sum_{t \in T}\mu(t) - 1\Big),
\end{align*}
where $v \in \Reals^T$ is the vector of Lagrange multipliers for the inequality constraints and $u \in \Reals$ is the Lagrange multiplier for the equality constraint. The KKT conditions for this problem assert that $\mu^*_\beta \in \Reals^T$ is primal optimal and $(v,u) \in \Reals^T \times \Reals$ is dual optimal if and only if the following conditions are satisfied:
\begin{itemize}
	\item[(C1)] stationarity:
\begin{align*}
 \frac{\partial \phi}{\partial \mu(t)}(\mu_\beta^*) + u - v(t) = 0 \text{ for all } t
  \iff \Ex\Bigg[\frac{e^{\beta X_t}}{\sum_s 
  \mu_{\beta}^*(s) e^{\beta X_s}}\Bigg] = u - v(t) \text{ for all } t.
 \end{align*}
 \item[(C2)] primal feasibility:
\begin{align*}
	\sum_t \mu_\beta^*(t) = 1 \quad \text{and} \quad \mu_\beta^*(t) \ge 0 \text{ for all } t.
 \end{align*}
 \item[(C3)] dual feasibility: $v(t) \ge 0$ for all $t$.
\item[(C4)] complementary slackness: $v(t)\mu^*_\beta(t) = 0$ for all $t$.
\end{itemize}
The following is a consequence of (C1)--(C4):
\begin{align*}
1 = \sum_{t \in T}\Ex\Bigg[\frac{\mu^*_\beta(t) e^{\beta X_t}}{\sum_s 
  \mu_{\beta}^*(s) e^{\beta X_s}}\Bigg] = u - \sum_{t \in T}\mu^*_\beta(t)v(t) = u.
\end{align*}
Hence, the equality
\begin{align*}
	\overline{\sm}_{\beta,\mu^*_\beta}(t) &= \Ex\Bigg[\frac{\mu^*_\beta(t) e^{\beta X_t}}{\sum_s 
  \mu_{\beta}^*(s) e^{\beta X_s}}\Bigg] = \mu^*_\beta(t)u = \mu^*_\beta(t)
\end{align*}
holds for every $t \in T$.\end{proof}

\noindent When $X$ is exchangeable, i.e., the joint law of $(X_t)_{t\in T}$ is invariant with respect to permutations of $T$, it is clear that $\mu_\beta^*=\su_{T}$, the uniform probability measure on $T$. For later use, we record some additional consequences of symmetry and concavity.

\begin{proposition}
  \label{prop:measure_sym}
Let $T$ be a finite set and let ${\mathbb G}$ be a finite group acting transitively on $T$. Let $f$ be a concave functional
on $\eP(T)$ which is ${\mathbb G}$-invariant in the sense that $f(\mu\circ g^{-1})= f(\mu)$ for any $\mu \in \eP(T)$ and any $g \in {\mathbb G}$. Then, for any $\mu\in \eP(T)$, we have
\begin{align*}
  f(\mu) \le f(\su_{T}).
\end{align*}
\end{proposition}
\begin{proof}
  Define $\mu^*$ by 
  \begin{align*}
    \mu^*(A) = \frac{1}{\abs{{\mathbb G}}}\sum_{g \in {\mathbb G}} \mu  (g^{-1}(A)) 
  \end{align*}
  for any measurable $A\subset T$ where $g^{-1}(A)= \set{g^{-1}(t): t\in A}$.
  By the concavity of $f$, for any $\mu\in \eP(T)$,
  \begin{align*}
  f(\mu) \le f (\mu^*). 
  \end{align*}
  It is readily verified that $\mu^*$ is uniform. Indeed, since ${\mathbb G}$ acts transitively on $T$, for any $s,t\in T$  we can find $h\in {\mathbb G}$ so that $h(t)=s$. Then
  \begin{align*}
    \mu^*(t)&= \frac{1}{\abs{{\mathbb G}}} \sum_{g\in {\mathbb G}} \mu(g^{-1}(t))
            =\frac{1}{\abs{{\mathbb G}}} \sum_{g\in {\mathbb G}} 
            \mu(g^{-1}\circ h^{-1}(s))
            = \frac{1}{\abs{{\mathbb G}}} \sum_{\mathfrak{f}\in {\mathbb G}} \mu(\mathfrak{f}^{-1}(s))
            = \mu^*(s). 
  \end{align*}
  Since $s$ and $t$ were arbitrary, it follows that $\mu^* = \su_T$. \end{proof}
  
\noindent Now let $X = (X_t)_{t \in T}$ be a stationary Gaussian process (recall Definition~\ref{def:stationary}). A key consequence of stationarity is that \begin{align}
    \label{eq:eq_distr}
    \set{X_t-X_s : t\in B(s,\epsilon)}\stackrel{{\rm d}}{=} \set{X_t-X_{s'}: t\in B(s',\epsilon)} \quad\text{for all } s, s'\in T, \epsilon>0
  \end{align}
  where $\stackrel{{\rm d}}{=}$ denotes equality of probability laws and where $B(s,\epsilon) \deq \{ t \in T: d(s,t) \le \epsilon \}$ is the (closed) ball of radius $\epsilon$ centered at $t$.
  To see this, consider finitely many points $t_1,\ldots,t_n\in B(s,\epsilon)$. Then for any $i,j \le n$
  \begin{align*}
  2\Cov(X_{t_i}-X_s, X_{t_j}-X_s) = 2\Ex[(X_{t_i}-X_s)(X_{t_j}-X_s)]
  = d^2(t_i,s)+ d^2(t_j,s)-d^2(t_i,t_j),
  \end{align*}
 which means that the probability law of the increments of a Gaussian process is determined by the natural metric $d$.
  Then \eqref{eq:eq_distr} follows since, by transitivity, there exists some $g \in {\mathbb G}$ such that $s' = g(s)$ and, by translation invariance, $g$ maps $B(s,\epsilon)$ isometrically onto $B(s',\epsilon)$.  Moreover, since $(X_t)$ is centered, the process $Y = (Y_t)_{t \in T}$ defined for any fixed $t_0 \in T$ by $Y_t= X_t-X_{t_0}$ satisfies
  \begin{align*}
 \eF_\beta(P_X; \mu) = \eF_\beta(P_Y,\mu).
  \end{align*}
 Therefore, using \eqref{eq:eq_distr} with $\epsilon = \diam(T)$, we see that 
  $\mu \mapsto \Phi(X ; \mu)$ is ${\mathbb G}$-invariant,
  which further implies the following:
  \begin{proposition}
    \label{prop:unif_sup}
Let a centered stationary Gaussian process $(X_t)_{t\in T}$ be given. Then, for any $\mu \in \eP(T)$,
  \begin{align*}
\eF_\beta(P_X;\mu) \le \eF_\beta(P_X; \su_T),
  \end{align*}
  where $\su_T$ is the uniform probability measure on $T$.
  \end{proposition}

An important special case of i.i.d.\ Gaussian processes is analyzed next. Let $G = (G_1,\ldots,G_n)$ be a standard Gaussian random vector in $\Reals^n$. We will view $G$ as a centered Gaussian process with the index set $T = [n] \deq \{1,\dots,n\}$. Let $\su_n$ denote the uniform probability distribution (normalized counting measure) on $[n]$, and define the following quantities:
\begin{align}\label{eq:G_quantities}
	\begin{split}
\phi_n(\beta) &\deq \eF_\beta(G; \su_n) ,\\
g_n(\beta) &\deq \eG_\beta(G; \su_n) ,\\
r_n(\beta) &\deq \Ex\lVert {\sm}_{\beta, \su_n}  \rVert_2^2 ,\\
Z_n(\beta) &\deq \sum_{i=1}^n e^{\beta G_i}
\end{split}
\end{align}
where $\lVert \cdot \rVert_2$ denotes the Euclidean $(\ell^2)$ norm on $\Reals^n$.
\begin{lemma} We have the following:
	\begin{enumerate}
		\item The functions $\beta \mapsto \phi_n(\beta)$ and $\beta \mapsto r_n(\beta)$ are monotonically increasing.
		\item The function $n \mapsto \phi_n(\beta)$ is monotonically increasing.
		\item The inequality
		\begin{align}\label{eq:phi_vs_g}
			\phi_n(\beta) \ge \frac{1}{2}g_n(\beta) =  \frac{\beta}{2}(1-r_n(\beta))
		\end{align}
		holds for all $n \ge 1$ and all $\beta \ge 0$.
	\end{enumerate}
\end{lemma}
\begin{proof} Let $\sm_\beta$ denote the random Gibbs measure $\sm_{\beta,\su_n}$, and let $\Lambda_n(\beta) \deq \log Z_n(\beta)$. The monotonicity of $\beta \mapsto \phi_n(\beta)$ is verified by direct differentiation: For $\beta > 0$,
	\begin{align}
		\phi'_n(\beta) &= \frac{1}{\beta}\big(g_n(\beta)-\phi_n(\beta)\big) = \frac{1}{\beta^2} \Ex D(\sm_\beta\|\su_n) \ge 0.
	\end{align}
For $r_n(\beta)$, we first use the definition of $\sm_\beta$ to calculate
	\begin{align*}
		\frac{\dif}{\dif \beta}\|\sm_{\beta}\|^2_2 &= 2 \sum^n_{i=1} \sm_\beta(i)\frac{\dif}{\dif \beta}\sm_\beta(i) \\
		&= 2 \sum^n_{i=1}\sm^2_\beta(i)\left(G_i  - \Lambda'_n(\beta)\right) \\
		&= 2 \|\sm_\beta\|^2_2 \left(\frac{1}{\|\sm_\beta\|^2_2}\sum^n_{i=1}  G_i \sm^2_\beta(i) - \Lambda'_n(\beta)\right) \\
		&= 2 \|\sm_\beta\|^2_2 \left(\Lambda'_n(2\beta) - \Lambda'_n(\beta)\right) \\
		&\ge 0,
	\end{align*}
	where the inequality follows from the fact that $\Lambda'_n(\beta)$ is nondecreasing, as $\Lambda_n(\beta)$ is convex. The monotonicity of $\beta \mapsto r_n(\beta)$ follows by taking expectations. To show monotonicity of $n \mapsto \phi_n(\beta)$, apply Proposition~\ref{prop:unif_sup} to $T = [n]$, $X = G$, and $\mu = \frac{1}{n-1}\sum^{n-1}_{i=1}\delta_i$.
	
Next, we obtain an expression for $g_n(\beta)$ in terms of $r_n(\beta)$.  For each $i \in [n]$, we can view $\sm(i)$ as a function of $G$:
	\begin{align*}
		\sm(i) = f_i(G) \deq \frac{e^{\beta G_i}}{\sum^n_{j=1}e^{\beta G_j}},
	\end{align*}
	with
	\begin{align*}
		\partial_jf_i(G) &= \beta\left(\frac{e^{\beta G_j}\boldsymbol{1}\{i = j\}}{\sum^n_{k=1}e^{\beta G_{k}}} - \frac{e^{\beta G_i}e^{\beta G_j}}{\left(\sum^n_{k=1} e^{\beta G_{k}}\right)^2}\right) \\
		&= \beta\big(\delta_{ij} \sm(i) - \sm(i)\sm(j)\big).
	\end{align*}
Gaussian integration by parts \citep[Lemma~2.2.4]{Adler2007} then gives
\begin{align*}
	\Ex[G_i \sm(i)] &= \sum^n_{j=1} \Ex[G_i G_j] \Ex[\partial_jf_i(G)] \\
	&= \beta\, \Ex[\sm(i)-\sm(i)^2].
\end{align*}
Summing over all $i \in [n]$, we obtain $g_n(\beta) = \beta(1-r_n(\beta))$. 

Since $g'_n(\beta) = \beta \phi_n(\beta)$, we have by the fundamental theorem of calculus
\begin{align*}
	\beta\phi_n(\beta)  &= \int^\beta_0 g_n(\alpha)\dif \alpha \\
	&= \int^\beta_0 \alpha(1-r_n(\alpha))\dif \alpha \\
	&= \frac{\beta^2}{2} - \int^\beta_0 \alpha r_n(\alpha)\dif \alpha.
\end{align*}
Integrating by parts and using the monotonicity of $r_n(\cdot)$, we obtain
\begin{align*}
	 \int^\beta_0 \alpha r_n(\alpha)\dif \alpha = \frac{\beta^2}{2}r_n(\beta) - \int^\beta_0 r'_n(\alpha) \dif \alpha \le \frac{\beta^2}{2}r_n(\beta).
\end{align*}
The inequality $\phi_n(\beta) \ge \frac{1}{2}g_n(\beta)$ follows by rearranging.
\end{proof}

\noindent Let $X = (X_i)_{1 \le i \le n}$ be a centered i.i.d.\ process. It follows by symmetry that $\overline{\sm}_{\beta,\su_n} = \su_n$, and the second identity of Lemma~\ref{lm:identities} simplifies to
\begin{align}
	\Ex D(\sm_{\beta,\su_n} \| \su_n) = I(X; \uptau_{\beta,\su_n}).
\end{align}
Sharp estimates for Gibbs averages and free energies of  Gaussian random vectors were obtained in an earlier work by the authors in terms of mutual information between $X$ and $\uptau_{\beta,\mu}$:

\begin{lemma}[\citet{chu_raginsky_alt}]\label{lem:iid_mi}
  Let $X = (X_1,\ldots,X_n)$ be a random vector of centered $\sigma$-subgaussian variables. Then, for any $\mu \in \eP([n])$ and any $\beta > 0$,
\begin{align*}
\eG_\beta(X;\mu)
  \le\sigma\sqrt{2I(X; \uptau_{\beta,\mu})}.
\end{align*}
Let $G = (G_1,\ldots,G_n)$ be a standard Gaussian random vector in $\Reals^n$. Then
\begin{align*}
 \eG_\beta(\sigma G, \su_n)
  \gtrsim \sigma \sqrt{I (\sigma G ; \uptau_{\beta, \su_n})}
\end{align*}
\begin{proof} We only prove the upper bound and refer the reader to \citet{chu_raginsky_alt} for the proof of the lower bound. Let $\sm$ and $\overline{\sm}$ denote $\sm_{\beta,\mu}$ and $\overline{\sm}_{\beta,\mu}$, respectively. By the variational representation of the relative entropy (the same result that underlies the Gibbs variational principle), for any $f : T \to \Reals$
	\begin{align*}
		- \log \langle \overline{\sm}, e^{-f} \rangle &= \inf_{\mu \in \cP(T)} \set{\langle \mu, f \rangle + D(\mu \| \overline{\sm})} \\
		&\le \langle \sm, f \rangle + D(\sm \| \overline{\sm}).
	\end{align*}
Applying this to the (random) function $f(t) = -\lambda X_t$ for some $\lambda > 0$ (to be chosen later) gives
\begin{align*}
	- \log \Bigg(\sum_{t \in T}\overline{\sm}(t)e^{\lambda X_t}\Bigg) \le -\lambda \langle \sm, X \rangle + D(\sm \| \overline{\sm}).
\end{align*}
Rearranging, taking expectations w.r.t.\ $X$, and using the fact that $\Ex D(\sm \| \overline{\sm}) = I(X; \uptau_{\beta,\mu})$, we obtain
\begin{align*}
	\eG_\beta(X;\mu) = \Ex \langle \sm, X \rangle \le \frac{1}{\lambda}\Bigg[I(X; \uptau_{\beta,\mu}) + \Ex \log \Bigg(\sum_{t \in T}\overline{\sm}(t)e^{\lambda X_t}\Bigg)\Bigg].
\end{align*}
Using Jensen's inequality and fact that the $X_t$'s are centered $\sigma$-subgaussian random variables, we can further upper-bound the second term on the right-hand side as follows:
\begin{align*}
	 \Ex \log \Bigg(\sum_{t \in T}\overline{\sm}(t)e^{\lambda X_t}\Bigg) \le \frac{\lambda^2 \sigma^2}{2}.
\end{align*}
Since $\lambda > 0$ was arbitrary, we obtain
\begin{align*}
	\eG_\beta(X; \mu) \le \inf_{\lambda > 0} \Bigg(\frac{1}{\lambda} I(X; \uptau_{\beta,\mu}) + \frac{\lambda \sigma^2}{2}\Bigg) = \sigma \sqrt{2 I(X; \uptau_{\beta,\mu})}.
\end{align*}
as claimed. \end{proof}

\end{lemma}
\begin{remark} The above bounds also hold for the quenched free energy. For the upper bound, we use the first inequality of Lemma~\ref{lm:equiv_soft}. For the lower bound, we use the inequality \eqref{eq:phi_vs_g}.
\end{remark}

\subsection{A key comparison inequality}

Recall the definition of $\eta(n,\beta)$ in \eqref{eq:eta}. The following lemma highlights the key role this quantity plays in controlling the behavior of the quenched free energy $\eF_\beta(P_X;\mu)$.

\begin{lemma}\label{lem:iid}
  Let $X = (X_1,\ldots,X_n)$ be a random vector in $\Reals^n$, such that each $X_i$ is a  centered $\sigma$-subgaussian random variable for some $\sigma>0$. Then for any $\mu \in \eP([n])$ and any $\beta > 0$,
  \begin{align*}
 \eF_\beta(X;\mu)
     \lesssim \sigma \sqrt{\eta(n,\sigma\beta)} \lesssim \eF_\beta(\sigma G; \su_n).
  \end{align*}
\end{lemma}
\begin{remark} Taking $X = \sigma G$ and $\mu = \su_n$, we immediately see that $\eF_\beta(\sigma G; \su_n) \asymp \sigma \sqrt{\eta(n,\sigma\beta)}$.
\end{remark}

\begin{proof} Since $X$ is centered, a simple scaling argument shows that
\begin{align*}
  \Phi_\beta(X ; \mu) = \sigma \Phi_{\sigma\beta}(X / \sigma, \mu).
\end{align*}
Thus, it suffices to consider the case $\sigma=1$. By Jensen's inequality and the subgaussian property, 
\begin{align*}
\eF_\beta(X;\mu) &\le \frac{1}{\beta} \log \sum^n_{i=1} \mu(i) \Ex e^{\beta X_i} \le \frac{\beta}{2}.
\end{align*}
Alternatively, we have
\begin{align*}
\eF_\beta(X;\mu) \le \Ex\Big[\max_{1 \le i \le n} X_i\Big] \le \sqrt{2\log n}.
\end{align*}
Combining the two estimates yields
\begin{align*}
	\eF_\beta(X;\mu) \le \frac{\beta}{2} \wedge \sqrt{2 \log n} \lesssim \sqrt{\eta(n,\beta)},
\end{align*}
which proves the first inequality.

We now turn to the second inequality. Recall the definitions of various quantities associated with $G$ given in \eqref{eq:G_quantities}.  We start by making a useful observation: there exists a universal constant $c>0$, such that
\begin{align}
  \label{eq:high_temp1}
  \phi_n(\beta) \ge c \beta \quad \text{for all } n \ge 2,  \beta \le 1.
\end{align}
Since $n \mapsto \phi_n(\beta)$ is an increasing function, it suffices to show that the inequality \eqref{eq:high_temp1} holds for $n =2$. From definitions, we have
\begin{align*}
\phi_2(\beta)
&= \frac{1}{\beta} \Ex\Bigg[ \log \frac{e^{\beta G_1}+ e^{\beta G_2}}{2}\Bigg]
= \frac{1}{\beta} \Ex\Bigg[\log 
\Bigg ( \frac{e^{\frac{\beta (G_1 + G_2)}{2}} \big (
e^{\frac{\beta (G_1 - G_2)}{2}} + e^{\frac{\beta (G_2 - G_1)}{2} } \big)}{2}\Bigg)\Bigg]\\
&= \frac{1}{\beta} \Ex\Bigg[\log 
\frac{e^{\frac{\beta (G_1 - G_2)}{2}} + e^{\frac{\beta (G_2 - G_1)}{2} }}{2}\Bigg]
=\frac{1}{\beta} \Ex\Bigg[ \log \cosh \Bigg(\frac{\beta G_1}{\sqrt{2}}\Bigg)\Bigg],
\end{align*}
where the last step follows from the fact that $\frac{G_1 - G_2}{2} \stackrel{{\rm d}}{=} \frac{G_1}{\sqrt{2}}$. Then, using the elementary inequality
\begin{align*}
\log \cosh x \ge \frac{x^2}{2} - \frac{x^4}{12},\quad \text{for all } x \in \bR 
\end{align*}
we have
\begin{align*}
\phi_2(\beta) 
\ge \frac{1}{\beta} \Ex\Bigg[\frac{\beta^2 G^2_1}{4} - \frac{\beta^4 G^4_1}{48}\Bigg]
= \frac{1}{\beta} \Bigg(\frac{\beta^2}{4} - \frac{\beta^4}{16}\Bigg)
\ge \frac{3 \beta}{16},
\end{align*}
 where the final inequality holds since $\beta \le 1$. This establishes \eqref{eq:high_temp1}. Next, we will show that the following ``high-temperature'' lower bound holds with some constant $c > 0$:
 \begin{align}
   \label{eq:high_temp2}
 \phi_n(\beta) \ge c \beta, \qquad \text{for all } n \ge 2 \text{ and all } 1 \le \beta \le \sqrt{\frac{1}{2}\log n}.
 \end{align}
Then, since the function $\beta \mapsto \phi_n(\beta)$ is monotonically increasing, in the ``low-temperature regime" $\beta > \sqrt{\frac{1}{2}\log n}$ we will automatically have
\begin{align*}
\phi_n(\beta) \ge \phi_n\left(\sqrt{\frac{1}{2} \log n}\right) \ge c \sqrt{\frac{1}{2} \log n}.
\end{align*}
Combining the bounds in both regimes will then give the lower bound $\phi_n(\beta) \gtrsim \sqrt{\eta(n,\beta)}$.
 
%Also, notice that as in the proof of lower bound on thermodynamic limit of REM, it suffices to show the desired lower bound in high-temperature regime.

%By monotonocity of $\beta \mapsto\phi_n(\beta)$, in the low-temperature regime $\beta > \sqrt{\frac{1}{2} \log n}$ we have

To show \eqref{eq:high_temp2}, we consider small $n$ and large $n$ cases separately.
For $2 \le n < 10^4$, we need to show
\begin{align*}
\phi_n(\beta) \ge c \beta \qquad \text{for all } 1 \le \beta \le \sqrt{\frac{1}{2} \log n} < \sqrt{2 \log 10}.
\end{align*}
This follows from the monotonicity of $\beta \mapsto \phi_n(\beta)$ and from \eqref{eq:high_temp1}:
\begin{align*}
\phi_n(\beta) \ge \phi_n(1) \ge c \ge c \frac{\beta}{ \sqrt{2 \log 10}} .
\end{align*}
Now we consider $n \ge 10^4$. It is a matter of straightforward computation to show that 
\begin{align*}
  \Ex Z_n(\beta)
&= n e^{\beta^2 / 2}\\
  \Ex Z_n^2(\beta) 
&= \Ex\Big[\sum_{i,j \le n} e^{\beta (G_i + G_j)}\Big]
=\Ex\Big[\sum^n_{i=1}e^{\beta 2 G_i}\Big] + \Ex\Big[\sum^n_{i=1} \sum_{j \neq i}e^{\beta \sqrt{2} G_i}\Big]
=n e^{2 \beta^2} + (n^2-n) e^{\beta^2}
\end{align*}
For some $\theta\in (0,1)$,
consider the event $E = \set{Z_n(\beta) > \theta \bfE Z_n(\beta)}$. 
By the Paley--Zygmund inequality, 
\begin{align*}
  \Pr[E^c] 
  &= \Pr[Z_n(\beta)  \le \theta \bfE Z_n(\beta)] 
  \le 1 - (1-\theta)^2 \frac{(\bfE Z_n(\beta))^2}{\bfE Z_n^2(\beta)} \\
  &= 1- (1-\theta)^2  \frac{1}{\frac{e^{\beta^2}}{n}+ 1-\frac{1}{n}}
  \le 1- (1-\theta)^2  \frac{1}{\frac{e^{\beta^2}}{n}+1 }
\end{align*}
Using the identity
\begin{align*}
	r_n(\beta) &= \Ex \Big[\frac{Z_n(2\beta)}{Z^2_n(\beta)}\Big]
\end{align*}
along with the easily established fact that $Z_n(2\beta)\le Z_n^2(\beta)$, we have
\begin{align*}
  r_n(\beta)
  &= \Ex\Big[\frac{Z_n(2\beta)}{Z_n^2(\beta)}{\bf 1}_E\Big] 
  +\Ex\Big[\frac{Z_n(2\beta)}{Z_n^2(\beta)}{\bf 1}_{E^c}\Big]\\
  &\le \frac{ \bfE Z_n(2\beta)}{\theta^2(\bfE Z_n(\beta))^2} + \bfP[E^c]\\
  &\le \frac{1}{\theta^2}\frac{e^{\beta^2}}{ n} + 1-(1-\theta)^2 
  \frac{1}{\frac{e^{\beta^2}}{n}+1}.
\end{align*}
For each fixed $\theta$,
$x \mapsto \frac{1}{\theta^2} x + 1- (1-\theta)^2 \frac{1}{x +1}$
is evidently an increasing function on $[0,\infty)$.
With $\theta=\frac{1}{2}$, if $\frac{e^{\beta^2}}{n} \le \frac{1}{100}$ (which is equivalent to  $\beta^2 \le \log (\frac{n}{100})$), then $r_n(\beta) \le 0.792 < \frac{4}{5}$. This, together with the inequality \eqref{eq:phi_vs_g},  implies that
\begin{align}\label{eq:low_quad}
  \phi_n(\beta) \ge \frac{\beta}{10} \quad \text{for } \beta^2 \le \log\left(\frac{n}{100}\right). 
\end{align}
The proof is complete by the fact $\log(\frac{n}{100}) \ge \frac{1}{2}\log n$ for $n \ge 10^4$. \end{proof}

\begin{corollary}
Let $X = (X_1,\ldots,X_n)$ be a random vector of centered $\sigma$-subgaussian variables for some $\sigma>0$. For any $\mu \in \eP([n])$,
\begin{align*}
 \eF_\beta(X; \mu) \lesssim \eF_\beta(\sigma G; \su_{[n]}) \quad \text{and} \quad \eG_\beta(X; \mu) \lesssim \eG_\beta(\sigma G; \su_{[n]}).
\end{align*}
\end{corollary}
\begin{remark}
 Using a more general result of \citet{vanhandel_subgaussian}, the first inequality can be strengthened to
  \begin{align*}
 \eF_\beta(X ; \mu) \lesssim \eF_\beta(\sigma G ; \mu), \qquad \text{for all } \mu \in \eP([n])
  \end{align*}
by recognizing
that $X \mapsto \Phi_\beta(X; \mu)$ is convex.
\end{remark}
\begin{proof}
  Combining the upper and lower bounds in Lemma~\ref{lem:iid} yields the first inequality.
  Since $\beta \mapsto \eG_\beta(X; \mu)$ is non-decreasing
  for any $\mu\in \eP([n])$, we can use \eqref{eq:equiv_soft} to get
  \begin{align*}
 \eG_\beta(\sigma G; \su_{[n]})
  &\gtrsim \sigma \sqrt{\eta(n,\sigma \beta)}
  \ge \frac{1}{2}\sigma \sqrt{\eta(n,2\sigma\beta)}
  \gtrsim \eG_{2\beta}(\sigma G; \su_{[n]})\\
  &\ge \eF_{2\beta}(\sigma G; \su_{[n]})
  \gtrsim \eF_{2\beta}(X; \mu)
  \gtrsim \eG_{\beta}(X;\mu)
  \end{align*}
 and complete the proof.
 \end{proof}

\section{The case of finite $T$: stationary processes}
\label{sec:finite_T_stationary}

Let $X = (X_t)_{t \in T}$ be a centered stationary Gaussian process with a finite index set $T$, on which we define the canonical metric $d(s,t) = (\Ex[(X_s-X_t)^2])^{1/2}$. This section is mainly devoted to the proof of Theorems~\ref{thm:stationary_1} and \ref{thm:stationary_2}. We also establish a related result that arrives at some of the same conclusions by an alternative route. 

The first step is to obtain a multiscale version of 
Lemma~\ref{lem:iid}.
The argument is in the spirit of the approach using labeled trees and 
packing trees \citep{guedon_supremum_2003}: we obtain one-step 
chaining estimates and then apply them recursively. Owing to the homogeneity 
of stationary Gaussians, the proofs for upper and lower bounds admit
similar structure. 

\begin{proposition}\label{prop:one_step_up} Let $X = (X_t)_{t \in T}$ be a centered  Gaussian process with a finite index set $T$. Let a subset $A \subset T$ with $\diam(A) \le a$ be given, and let $\cA = \set{ A_1,\ldots,A_n}$ be a partition of $A$. For any $\mu \in \eP(T)$ supported on $A$, let  $\mu^i \deq \mu(\cdot|A_i)$ be the 
conditional probability measure given $A_i$. Then, for any $\beta \ge 0$ we have
\begin{align*}
 \eF_\beta(X;\mu) \le \max_{i \le n} \eF_\beta(X; \mu^i)
  + C a \sqrt{\eta(n,a\beta)},
\end{align*}
where $C > 0$ is a universal constant.
\end{proposition}
\begin{proof}
For each $X_t$, consider $n$ i.i.d.\ copies $(X_t^{(i)})_{i\le n}$.
Let $t_i$ be an arbitrary point in $A_i$ for each $i$. 
Let $G = (G_i)_{i\le n}$ be a standard Gaussian random vector independent of everything else. 
Define a process $(Y_t)_{t\in A}$ by setting, for each $i$,
\begin{align*}
  Y_t \deq X_t^{(i)} - X_{t_i}^{(i)} + \frac{a}{\sqrt{2}} G_i, \qquad t \in A_i.
\end{align*}
Note that, if two points $s,t$ are in the same partition subset $A_i$,
then $\Ex{(Y_t - Y_s)^2}= \Ex{(X_t-X_s)^2}$. Otherwise if $s\in A_i$ and $t\in A_j$ for some $i\neq j$,
\begin{align*}
  \Ex{(Y_t-Y_s)^2} = \Ex{(X_t^{(i)}-X_{t_i}^{(i)})^2}
  + \Ex{(X_s^{(j)}-X_{s_j}^{(j)})^2} + \frac{a^2}{2} \Ex{(G_i-G_j)^2} \ge a^2 
  \ge \Ex{(X_t-X_s)^2}.
\end{align*}
Then a Gaussian interpolation argument due to Chatterjee (see, e.g., the proof of Theorem~2.2.5 in \citet{Adler2007}) can be used to show that
\begin{align*}
\eF_\beta(X; \mu)
&\le \eF_\beta(Y; \mu) =\frac{1}{\beta} \Ex\Bigg[\log\sum_{i\le n} \mu(A_i)\sum_{t\in A_i} \mu^i(t) 
e^{\beta(X_t^{(i)}-X_{t_i}^{(i)}+a/\sqrt{2} G_i)}\Bigg].
\end{align*}
Let $\hat{\mu} =(\hat{\mu}_1,\ldots,\hat{\mu}_n)
$ with $\hat{\mu}_i \deq \mu(A_i)$, and define $Z\deq(Z_1,\ldots,Z_n)$ by
\begin{align*}
Z_i=\frac{1}{\beta} \log \sum_{t\in A_i}\mu^i(t) e^{\beta (X_t^{(i)}-X_{t_i}^{(i)})}.
\end{align*}
By Lemma~\ref{lem:concentration}, each $Z_i$ is $\sqrt{2}a$-subgaussian.
Denoting by $\nu_i \propto \hat{\mu}_i \exp (\beta a G_i/\sqrt{2})$ 
the tilted probability measure obtained from $\hat{\mu}$,
the above inequality can be written as 
\begin{align*}
 \eF_\beta(X;\mu) \le \eF_\beta(Z;\nu) 
  + \eF_\beta\left(\frac{a}{\sqrt{2}}G ; \hat{\mu}\right).
\end{align*}
For the first term, let $\overline{\nu} \deq \bfE \nu$ and 
$\tilde{\nu}_i \propto \overline{\nu}_i \exp (\beta \bfE Z_i)$. By Jensen's inequality,
\begin{align*}
 \eF_\beta(Z;\nu)
  &\le \frac{1}{\beta} \Ex{\log \sum_i \overline{\nu}_i e^{\beta Z_i}}\\
  &= \frac{1}{\beta}\Ex\Big[\log \sum_i \tilde{\nu}_i e^{\beta (Z_i - \bfE Z_i)}\Big]
  + \frac{1}{\beta} \Ex\Big[\log \sum_i \overline{\nu}_i e^{\beta \bfE Z_i}\Big]\\
  &\le \sqrt{2}a \sqrt{\eta(n,\sqrt{2}a \beta)} + \max_i \bfE Z_i\\
  &\le 2a \sqrt{\eta(n,a \beta)} + \max_i \eF_\beta(X ;  \mu^i),
\end{align*}
where the second inequality uses Lemma~\ref{lem:iid}.
Applying Lemma~\ref{lem:iid} again to the term $\eF_\beta(aG/\sqrt{2}; \hat{\mu})$ 
gives the result.
\end{proof}
\noindent The complementary lower bound is, essentially, a super-Sudakov inequality from \citet{chu_raginsky_alt} with minor modifications:
\begin{proposition}\label{prop:one_step_low}
  Let $A \subset T$ be given, and let $S=\{s_1,\ldots,s_n\} \subseteq A$ be a $4\sigma$-packing of $A$, i.e., $d(s_i,s_j) \ge 4\sigma$ for all distinct $i,j$. Then the following inequality holds with some universal constant $c<1$:
  \begin{align*}
    \sup_{\mu \in \eP(A)} \eF_\beta(X;\mu)
    \ge c \sigma \sqrt{\eta(n,\sigma \beta)} + 
    \min_{i\le n} \sup_{\mu \in \eP(B(s_i,\sigma))} 
    \eF_\beta(X;\mu).
  \end{align*}

\end{proposition}
\begin{proof}
Let $(X_t^{(i)})_{t \in T}$ be i.i.d.\ copies of $(X_t)$, and let $G = (G_i)_{1 \le i \le n}$ be 
standard Gaussian independent 
of everything else. Let $A' \deq \bigcup_{i\le n}B(s_i, \sigma)$. For $t \in B(s_i,\sigma)$, let
\begin{align*}
  Y_t = X_t^{(i)} - X_{s_i}^{(i)} + \sigma G_i.
\end{align*}
Let $\mu^i\in \eP(B(s_i, \sigma))$ be such that
\begin{align*}
 \eF_\beta(X ; \mu^i)
  =\sup_{\mu\in \eP(B(s_i, \sigma))} \eF_\beta(X ; \mu),
\end{align*}
where $\eP(A)$, for $A \subseteq T$, denotes the set of all $\mu \in \eP(T)$ whose support is contained in $A$. Let $\nu\in \eP(A')$ be the composition of $\su_n$ and $(\mu^i)$, i.e., 
\begin{align*}
  \nu(t) = \frac{1}{n} \sum_{i\le n} 
  \mathbf{1}_{\{t\in B(s_i,\sigma)\}} \mu^i(t) \quad\text{for } t \in A'.
\end{align*}
By the fact that $A' \subset A$ and Gaussian interpolation,
\begin{align*}
    \sup_{\mu \in \eP(A)} \eF_\beta(X ;\mu)
    &\ge \eF_\beta(X ; \nu) \\
    &\ge \frac{1}{\beta}\Ex\Big[\log \sum_i \frac{1}{n} 
    \sum_{t\in B(s_i,\sigma)}\mu^i(t) e^{\beta(X_t^{(i)}-X_{s_i}^{(i)} + \sigma G_i)}\Big]\\
    &\ge \eF_\beta(\sigma G ;  \su_n) 
    +  \min_{i\le n}  \eF_\beta(X ; \mu^i)\\
    &\ge c \sigma \sqrt{\eta(n,\sigma\beta)} 
    +  \min_{i\le n}  \eF_\beta(X ; \mu^i),
\end{align*}
where the last step uses Lemma~\ref{lem:iid}.
\end{proof}

\subsection{The proof of Theorem~\ref{thm:stationary_1}}

Without loss of generality, we assume that $\Delta = 1$. We prove the upper and lower bounds for $\eF_\beta(P_X; \mu)$ first. For any $A \subset T$, define the quantity
  \begin{align*}
  \zeta_\beta(A)\deq \sup_{\mu\in \eP(A)} \eF_\beta(X ; \mu),
  \end{align*}
  which is evidently monotone with respect to set inclusion: if $A \subseteq A'$, then $\zeta_\beta(A) \le \zeta_\beta(A')$. Let $\beta_k \deq \beta r^{-k}$. 
Fix an arbitrary $t_0\in T$. Let $\{t_1,\ldots,t_{n_{k}}\}$ be a minimum ${r^{-k}}$-cover for $B(t_0, r^{-k+1} )$. Let $\{A_1,\ldots,A_{n_k}\}$ be the partition induced by these
centers so that $A_i \subset B(t_i, {r^{-k}})$.
By Proposition~\ref{prop:one_step_up}, for any $\mu\in \eP(B(t_0,r^{-k+1}))$,
\begin{align*}
\eF_\beta(X ;\mu)
&\le C r^{-k+1} \sqrt{\eta(n_k,\beta_{k-1})}  + \max_{i\le n_k} \zeta_\beta(A_i) \\
&\le C r^{-k+1} \sqrt{\eta(n_k,\beta_{k-1})}  
+ \max_{i\le n_k} \zeta_\beta(B(t_i, {r^{-k}})).
\end{align*}
Since $(X_t)$ is stationary, the value of $\zeta_\beta(B(t,r^{-k}))$ does not depend on $t$. Hence,
\begin{align*}
\zeta_\beta(B(t_0, r^{-k+1})) 
&\le C r^{-k+1} \sqrt{\eta(n_k,\beta_{k-1})}  
+  \zeta_\beta(B(t_0, {r^{-k}}))\\
&\le C r^{-k} \sqrt{\eta(n_k,\beta_{k})}  
+  \zeta_\beta(B(t_0, {r^{-k}}))
. 
\end{align*}
Applying this iteratively on $T= B(t_0,1)$ yields the upper bound.
For the lower bound, let $\set{t_1,\ldots,t_{m_k}}$ be a maximal $r^{-k}$-packing 
of $B(t_0,r^{-k+1})$.
Using duality between packing and covering numbers, we have $m_k \ge n_k$. 
Then, by Proposition~\ref{prop:one_step_low},
\begin{align*}
  \zeta_\beta(B(t_0, r^{-k+1})) 
  &\ge c r^{-k} \sqrt{\eta(m_k,\beta_k)} 
  + \zeta_\beta(B(t_0, \frac{r^{-k}}{4} ))\\
  &\ge c r^{-k} \sqrt{\eta(n_k,\beta_k)} 
  + \zeta_\beta(B(t_0, r^{-k-1} )).
\end{align*}
Since $2\zeta_\beta(T)\ge \zeta_\beta(B(t_0,1))+ \zeta_\beta(B(t_0,r^{-1}))$,
applying the above iteratively on $B(t_0,1)$ and $B(t_0,r^{-1})$ gives the
lower bound.

The same upper and lower bounds hold for $\eG_\beta(P_X;\mu)$, with possibly different multiplicative constants, by the equivalence of quenched free energy and quenched Gibbs average in Eq.~\eqref{eq:equiv_soft}.

\subsection{The proof of Theorem~\ref{thm:stationary_2}}

  Using Proposition~\ref{prop:disc_to_int} to write the integral in the form of discrete sum,
  it suffices to show 
  \begin{align}\label{eq:sum_station_global}
  \sup_{\mu\in \eP(T) }\eF_\beta(X; \mu)\asymp \sum_{k \ge 1} r^{-k} 
  \sqrt{\eta(N_k,\beta_k)}
  \end{align}
   where $N_k=\sN(T,r^{-k})$ and $\beta_k =\beta r^{-k}$.
For the upper bound, it is clear that $n_k \le N_k$.
We prove the lower bound. By definition of the covering number, for each $k$
\begin{align*}
N_k \le N_{k-1} n_k.
\end{align*}
Then 
using the inequality $\sqrt{a+b}\le \sqrt{a} + \sqrt{b}$ and the inequality $(x+y) \wedge z \le (x \wedge z) + (y \wedge z)$ that holds for all $x,y,z \ge 0$,
we have
\begin{align*}
  \rI &\deq \sum_{k \ge 1} r^{-k}\sqrt{\log N_k \wedge \beta_k^2}
      \le \sum_{k \ge 1} r^{-k}\sqrt{(\log n_k + \log N_{k-1}) \wedge \beta_k^2}\\
      &\le \sum_{k \ge 1} r^{-k} \sqrt{\eta(n_k,\beta_k)} 
      + \sum_{k \ge 1} r^{-k} \sqrt{\log N_{k-1} \wedge \beta_k^2}\\
      &\le \sum_{k \ge 1} r^{-k} \sqrt{\eta(n_k,\beta_k)} 
      + r^{-1}\rI =: {\rm II} + r^{-1}\rI
\end{align*}
which implies that $\rI \lesssim {\rm II}$.

\subsection{An alternative approach via symmetrization}

We conclude the analysis of the stationary case by providing an alternative proof 
for the upper bound in \eqref{eq:sum_station_global} without invoking 
Proposition~\ref{prop:one_step_up}.
The idea is to construct an auxiliary Gaussian process so that its 
free energy factorizes into various scales and then use the Gaussian comparison inequality.
\begin{proposition}
 We have 
\begin{align*}
\sup_{\mu \in \eP(T)} \eF_\beta(X;\mu) \lesssim \sum_{k \ge 1} r^{-k} 
\sqrt{\eta(N_k,\beta r^{-k})}.
\end{align*}
\end{proposition}
\begin{proof}
For each $k$, let $\sigma_k=r^{-k+2}$, let $S_k\subset T$ be a minimal $r^{-k}$-net of $T$ 
with cardinality $\sN(T,\sigma_k)$, and let $\pi_k: T \to S_k$ be the projection map so that 
$d(\pi_k(t),t)=\min_{s\in S_k} d(s,t)$ for any $t\in T$.
Let $(G_s^{(k)})_{s \in S_k}$ be a collection of i.i.d.\ standard Gaussians.
Define a Gaussian process $(Y_t)_{t\in T}$ by
\begin{align*}
Y_t \deq \sum_{k \ge 1} \sigma_k G_{\pi_k(t)}.
\end{align*}
Consider two distinct points $t,s \in T$ and let
$\overline{k} \deq \overline{k}(t,s) = \min \set{k \ge 1: \pi_k(t) \neq \pi_k(s)}$.
Then
\begin{align*}
\Ex{(Y_t-Y_s)^2} = 2 \sum_{k: \pi_k(t)\neq \pi_k(s)} \sigma_k^2 \ge 2 \sigma_{\overline{k}}^2
\end{align*}
and
\begin{align*}
\sqrt{\Ex{(X_t-X_s)^2}}
&\le d(t, \pi_{\overline{k}}(t)) + d(\pi_{\overline{k}(t)},\pi_{\overline{k}(s)}) 
+ d(\pi_{\overline{k}}(s),s)\\
&\le r^{-\overline{k}} + 2 r^{-\overline{k}+1} + r^{-\overline{k}}
\le \sqrt{2}r^{-\overline{k}+2}\le \sqrt{\Ex{(Y_t-Y_s)^2}}
\end{align*}
By Gaussian interpolation,
\begin{align*}
\eF_\beta(X ; \mu) \le \eF_\beta(Y ;\mu).
\end{align*}
Now we embed $T$ into $\bm{T} \deq \prod_{k = 1}^K S_k$ by means of the map $\iota : T \to \bm{T}$ given by
\begin{align*}
\iota(t) \deq (\pi_1(t),\ldots,\pi_K(t)).
\end{align*}
For each $k$, let ${\rm Sym}(S_k)$ be the symmetric group of $S_k$ (i.e., the group of bijective maps $S_k \to S_k$). Then $(Y_{\bm{t}})_{\bm{t}\in \bm{T}}$ 
is a stationary Gaussian process with respect to the group action $\prod^K_{k=1} {\rm Sym}(S_k)$.
Thus, by Proposition~\ref{prop:unif_sup} and the fact that $\iota(T) \subset \bm{T} $
\begin{align*}
\eF_\beta(Y ; \mu)
&= \eF_\beta(Y ; \iota_\#\mu)\\
&\le \eF_\beta( Y ;  \su_{\bm{T}}  )
= \frac{1}{\beta}\Ex\Big[\log \sum_{\bm{t}\in \bm{T}} 
\su_{\bm{T}}(\bm{t}) e^{\beta Y_{\bm{t}}}\Big]  \\
&= \frac{1}{\beta}
\Ex\,{\log \bigg( \frac{1}{N_1\ldots N_k} \sum_{s_1 \in S_1}\ldots
     \sum_{s_K \in S_K} \prod_{k \ge 1} e^{\sigma_k G_{s_k}}\bigg)}\\
&= \frac{1}{\beta} \sum_{k \ge 1} \Ex\Big[\log \sum_{s \in S_k} 
\frac{1}{N_k} e^{\beta \sigma_k G_{s}}\Big]
= \sum_{k \ge 1} \eF_\beta(\sigma_k G ; \su_{S_k}) \\
&\lesssim \sum_{k \ge 1} r^{-k} \sqrt{\eta(N_k,\beta r^{-k})},
\end{align*}
where the last step uses Lemma~\ref{lem:iid}.
\end{proof}

\section{The case of finite $T$: nonstationary processes}
\label{sec:finite_T_nonstationary}

We now drop the assumption of stationarity. To handle this case, we will make fundamental use of the tensorization (or lifting) technique introduced by \citet{liuSimpleSharpGeneralization2025}. In the following, we use bold symbols to denote the elements of the product space $T^N$ for $N \in \bN$: $\bm{t} = (t_1,\ldots,t_N)$. Given a Gaussian process $X = (X_t)_{t \in T}$, we define another Gaussian process $\bm{X}^N = (\bm{X}^N_{\bm{t}})_{\bm{t} \in T^N}$ by
\begin{align*}
  \bm{X}^N_{\bm{t}}  =   \sum_{i=1}^N X_{t_i}^{(i)} 
\end{align*}
where $X^{(i)} = (X_t^{(i)})_{t \in T}$ are i.i.d.\ copies of $X$. 
The natural metric associated with $\bm{X}^N$ is 
\begin{align*}
  d_N^2(\bm{t},\bm{s}) = \sum_{i} d^2(t_i, s_i).
\end{align*}
It is readily verified that, for any $\mu\in \eP(T)$, 
\begin{align*}
 \eF_\beta(X ; \mu) = \frac{1}{N} \eF_\beta(\bm{X}^N ; \mu^{\otimes N}).
\end{align*}
In fact, more can be said:
\begin{proposition}\label{prop:sup_tensor}
For any $N \in \bN$,
\begin{align*}
\sup_{\mu\in \eP(T)} \eF_\beta(X ; \mu)
= \frac{1}{N} \sup_{\bm{\mu}\in \eP(T^N)} \eF_\beta(\bm{X}^N ;  \bm{\mu}).
\end{align*}
%Equivalently,
%\begin{align*}
%\Ex{\Phi_\beta(\bm{X}^N ;  (\mu_\beta^*)^{\otimes N})}
% =\sup_{\bm{\mu}\in \cP(T^N)} \Ex{\Phi_\beta(\bm{X}^N ; \bm{\mu})}
%\end{align*}
\end{proposition}
\begin{proof}
  Let $\phi(\mu) \deq \eF_\beta(X ; \mu)$ 
  and $\phi_N (\bm{\mu})= \eF_\beta (\bm{X}^N ; \bm{\mu})$. The ``$\le$'' part follows immediately from the fact that $\phi_N(\mu^{\otimes N}) = N \phi(\mu)$. We now show the other direction. 
Let $\bm{\sm}_{\beta,\bm{\mu}}$ be the random Gibbs measure defined by $\frac{\dif\bm{\sm}_{\beta,\bm{\mu}}}{\dif\bm{\mu}} \propto e^{\beta \bm{X}^N}$. By the Gibbs variational principle (Lemma~\ref{lm:Gibbs_VP}),
\begin{align*}
	\inner{\bm{\sm}_{\beta, \bm{\mu}}}{\bm{X}}
	  -\frac{1}{\beta}D(\bm{\sm}_{\beta, \bm{\mu}}\| \bm{\mu}) = \sup_{Q\ \in \eP(T^N)}
 \left\{\inner{Q}{\bm{X}}
  -\frac{1}{\beta}D(Q\| \bm{\mu})\right\}.
\end{align*}
We can move the supremum outside of the expectation as follows:
\begin{align*}
  \phi_N(\bm{\mu})
  &= \Ex\Big[\inner{\bm{\sm}_{\beta, \bm{\mu}}}{\bm{X}}
  -\frac{1}{\beta}D(\bm{\sm}_{\beta, \bm{\mu}}\| \bm{\mu})\Big]
  = \sup_{Q_{\bm{\tau}|\bm{X}}}
  \set[\Big]{\Ex\Big[\inner{Q_{\bm{\tau}|\bm{X}}}{\bm{X}}
  -\frac{1}{\beta}D(Q_{\bm{\tau}|\bm{X}}\| \bm{\mu})\Big]},
  \end{align*}
where the supremum is now over all regular conditional probability laws (Markov kernels) $Q_{\bm{\tau}|\bm{X}}$ from $\Reals^{T^N}$ to $T^N$. For each choice of $Q_{\bm{\tau}|\bm{X}}$, we can express the expectation $\Ex D(Q_{\bm{\tau}|\bm{X}}\|\bm{\mu})$ in terms of the mutual information $I(\bm{X};\bm{\tau})$, with $Q_{\bm{\tau}}$ denoting the marginal law of $\bm{\tau} = (\tau_1,\dots,\tau_N)$. We note for later that $(X^{(1)},\dots,X^{(N)})$ and $\bm{\tau}$ are conditionally independent given $\bm{X}$ and that, for each $i$, $X^{(i)}$ and $\tau_i$ are conditionally independent given $\bm{\tau}$.

By the above invocation of the Gibbs variational principle and by the additive structure of $\bm{X}$, we have
  \begin{align*}
&\sup_{Q_{\bm{\tau}|\bm{X}}}
  \set[\Big]{\Ex\Big[\inner{Q_{\bm{\tau}|\bm{X}}}{\bm{X}}
  -\frac{1}{\beta}D(Q_{\bm{\tau}|\bm{X}}\| \bm{\mu})\Big]} \\
  &=\sup_{Q_{\bm{\tau}|\bm{X}}}
  \set[\Big]{\Ex{\inner{Q_{\bm{\tau}|\bm{X}}}{\bm{X}}}
  - \frac{I(\bm{X};\bm{\tau})+D(Q_{\bm{\tau}}\| \bm{\mu})}{\beta}}\\ 
  &\le\sup_{Q_{\bm{\tau}|\bm{X}}} 
  \set[\Big]{\Ex{\inner{Q_{\bm{\tau}|\bm{X}}}{\bm{X}}}- \frac{1}{\beta} I(\bm{X};\bm{\tau})} \\
  &=\sup_{Q_{\bm{\tau}|\bm{X}}} 
  \set[\Big]{\Ex\Big[\sum_{i=1}^N \inner{Q_{\tau_i|X^{(i)}}}{X^{(i)}}\Big]
  -\frac{1}{\beta} I(\bm{X};\bm{\tau})}\\ 
  &\le \sup_{Q_{\bm{\tau}|\bm{X}}} 
  \set[\Big]{\sum_{i=1}^N\big(\Ex{ \inner{Q_{\tau_i|X^{(i)}}}{X^{(i)}}}
  - \frac{1}{\beta}I(X^{(i)};\tau_i)\big)}\\
  &\le N \sup_{Q_{\tau|X}} \set{\Ex{\inner{Q_{\tau|X}}{X}-\frac{1}{\beta}I(X ; \tau)}}  \\
  &= N \sup_{\mu} \sup_{Q_{\tau|X}} 
  \set[\Big]{\Ex \inner{Q_{\tau|X}}{X}
  -\frac{I(X ; \tau)+D(Q_\tau \| \mu)}{\beta}} \\
  &= N \sup_{\mu} \sup_{Q_{\tau|X}} \set[\Big]{\Ex\Big[\inner{Q_{\tau|X}}{X}
  -\frac{1}{\beta}D(Q_{\tau|X }\|\mu)\Big]} \\
 & = N \sup_{\mu\in \eP(T)} \phi(\mu)
\end{align*}
where the first inequality uses nonnegativity of the relative entropy and the second inequality follows from the following chain of information-theoretic inequalities:
\begin{align*}
	I(\bm{X}; \bm{\tau}) &\ge I(X^{(1)},\dots,X^{(N)}; \bm{\tau}) \\
	&\ge \sum^N_{i=1} I(X^{(i)}; \bm{\tau}) \\
	&\ge \sum^N_{i=1} I(X^{(i)}; \tau_i).
\end{align*}
In the above, the first step is by the data processing inequality and conditional independence, the second step uses the fact that the $X^{(i)}$'s are independent (see, e.g., \cite[Theorem~6.1]{polyanskiyInformationTheoryCoding2024}), and the third step is again by the data processing inequality and conditional independence applied termwise.
\end{proof}

\subsection{Some information-theoretic preliminaries}
\label{ssec:info_th_prelims}

For $N \in \bN$, let $\cP_N \deq \{\mu\in \eP(T)
   : N \mu(t)\in \bZ_+ \text{ for all } t \in T\}$; that is,  $\mu \in \cP_N$ whenever the $\mu$-probability of each $t \in T$ is an integer multiple of $1/N$. In the information theory literature, the elements of $\cP_N$ are referred to as \textit{types} of length $N$ \citep{csiszarInformationTheoryCoding2011}. For $\mu\in \cP_N$, define the set
   \begin{align*}
	   \cC_{N,\mu} \deq \{\bm{t} = (t_1,\dots,t_N)\in T^N : \hat{P}_{\bm{t}}=\mu\}
	\end{align*}
where $\hat{P}_{\bm{t}}=\frac{1}{N}\sum_{i=1}^N \delta_{t_i}$ 
is the empirical measure of $\bm{t}\in T^N$. 
Let $\eP_{\mu}(T)\deq \set{\nu\in \eP(T): \nu\ll \mu}$ and $\cP_{N,\mu}\deq \cP_N \cap \eP_{\mu}(T)$.
Let $S_\mu \deq \set{t\in T: \mu(t) >0} $. Let $\eP_+(T) \deq \set{\nu\in \eP(T): S_\nu = T}$. 
Let $\eP_{\bQ}(T)\subset \eP(T)$ be the set of \textit{rational} probability measures, where 
we say that $\mu$ is rational if there exists $N\in \bN$ so that $\mu \in \cP_N$. For any $\rho \in \eP(T)$, we can find a rational probability measure $\hat{\rho}\in \cP_N$
so that $\lVert\rho-\hat{\rho}\rVert_{\infty} \le \frac{1}{N}$ by first flooring
each $(N\rho(t))_{t\in T}$ and then choosing arbitrary $N \abs{T} - \sum_t\lfloor N \rho(t)\rfloor$ 
coordinates to increase the weight by $\frac{1}{N}$
in order to make sure that the constructed rational measure is a probability measure. Thus,
\begin{align}
  \label{eq:round}
\lVert\rho-\hat{\rho}\rVert_{{\rm TV}}\le \frac{|T|}{N}.
\end{align}
For any $\mu\in \eP(T)$, let $\cC_{N,\mu, \epsilon}= 
\{\bm{t}\in T^N: \lVert\hat{P}_{\bm{t}} - \mu\rVert_{\rm TV}\le \epsilon\}$,
where 
\begin{align*}
	\lVert\mu-\rho\rVert_{\rm TV} \deq \frac{1}{2}\lVert\nu-\rho\rVert_1 
= \frac{1}{2}\sum_{t\in T}|\nu(t) - \rho(t)|
\end{align*} is the total variation distance between  any $\mu, \rho\in \eP(T)$. Another useful result is a relation between matching and coupling in 
\cite[lemma~3]{liuSimpleSharpGeneralization2025}:
\begin{lemma}
  \label{lem:match_couple}
  For any $\theta \in \cP_N, \bm{t}\in T^N$,
  \begin{align*}
    \lVert\hat{P}_{\bm{t}}- \theta\rVert_{\rm TV } 
    = \frac{1}{N} \min_{\bm{s}\in \cC_{N,\theta}} d_{\rm H} (\bm{t}, \bm{s}),
  \end{align*}
  where $d_{\rm H}(\bm{t},\bm{s}) \deq \sum^N_{i=1}\bm{1}_{\{t_i \neq s_i\}}$ is the Hamming distance between $\bm{t}$ and $\bm{s}$.
\end{lemma}
\noindent
The proof is identical to the one by \citet{liuSimpleSharpGeneralization2025} 
except the Hamming distance is used here instead of 
the $\ell_2$-distance for the transport cost. Finally, we will need the following result on the variation of the Shannon entropy 
\begin{align*} 
H(\mu) \deq - \sum_{t \in T} \mu(t) \log \mu(t), \qquad \mu \in \eP(T)
\end{align*}
on neighborhoods with respect to the total variation distance:

\begin{proposition}
  \label{prop:cont_H_eps}
For any $\nu\in \eP(T)$, let
\begin{align*}
	H_{\epsilon}(\nu)\deq \sup_{\rho\in B_{{\rm TV}}(\nu,\epsilon) }H(\rho)
\end{align*}
where
\begin{align}
	B_{{\rm TV}}(\nu,\epsilon) \deq \left\{ \mu \in \eP(T): \lVert \mu - \nu \rVert_{{\rm TV}} \le \epsilon \right\}
\end{align}
is the closed ball of radius $\epsilon$ with respect to total variation distance centered at $\nu$.
  For any $\epsilon \in [0,1-\frac{1}{\abs{T}}]$, let $\omega(\epsilon)
  \deq \epsilon \log (\abs{T}-1) + h(\epsilon) $
  where $h(\epsilon) \deq -\epsilon \log \epsilon - (1-\epsilon)\log (1-\epsilon)$ is the binary entropy function.
 \begin{enumerate}
   \item 
     Given any $\nu\in \eP(T)$ and $\epsilon, \delta \in [0, 1-\frac{1}{\abs{T}}]$, 
     \begin{align*}
      H_{\epsilon}(\nu) \le H_{\epsilon+\delta}(\nu)\le H_{\epsilon}(\nu) + \omega(\delta)
     \end{align*}
  \item For any $\epsilon,\delta \in [0, 1-\frac{1}{\abs{T}}]$,
    \begin{align*}
    \sup_{\rho,\nu:\, \lVert \rho-\nu \rVert_{\rm TV}\le \delta} 
    \abs{H_\epsilon(\rho)- H_{\epsilon}(\nu)} \le \omega(\delta)
    \end{align*}
 
 \end{enumerate}

\end{proposition}
\begin{proof}
  By e.g., Theorem~3 of \citet{sasonEntropyBoundsDiscrete2013},
  for $\epsilon \le 1-\frac{1}{\abs{T}}$,
  \begin{align}
    \label{eq:cont_ent}
  \sup_{\rho,\nu: \, \lVert \rho-\nu \rVert_{\rm TV} \le \epsilon} \abs{H(\nu)-H(\rho)} 
  \le \omega(\epsilon). 
  \end{align}
 For the first statement, we first show that, for any $\theta \in B_{{\rm TV}}(\nu,\epsilon+\delta)$, there exists
  $\rho_{\theta} \in B_{{\rm TV}}(\nu,\epsilon) \cap B_{{\rm TV}}(\theta,\delta)$. We distinguish two cases:
  \begin{enumerate}
    \item $\theta \in B_{{\rm TV}}(\nu,\epsilon)$. If $\theta \not\in B_{{\rm TV}}(\nu,\delta)$, then take
	\begin{align*}\rho_\theta = \theta + \frac{\delta}{\lVert \theta - \nu \rVert_{{\rm TV}}}(\nu - \theta) = \left(1 - \frac{\delta}{\lVert \theta - \nu \rVert_{{\rm TV}}} \right) \theta + \frac{\delta}{\lVert \theta - \nu \rVert_{{\rm TV}}}\nu.
	\end{align*}
Then
\begin{align*}
	\lVert \rho_\theta - \nu \rVert_{{\rm TV}} = \lVert \theta - \nu \rVert_{{\rm TV}} - \delta \le \lVert \theta - \nu \rVert_{{\rm TV}} \le \epsilon \quad\text{ and } \quad
	\lVert \rho_\theta - \theta \|_{{\rm TV}} = \delta.
\end{align*}
If $\theta \in B_{{\rm TV}}(\nu,\delta)$, then $\lVert \theta - \nu \rVert_{{\rm TV}} \le \epsilon \wedge \delta$. Let $\rho_\theta$ be any  convex combination of $\theta$ and $\nu$, e.g., $\rho_\theta = \frac{1}{2}(\theta + \nu)$. In that case,
\begin{align*}
	\lVert \rho_\theta - \nu \rVert_{{\rm TV}} = \lVert \rho_\theta - \theta \rVert_{{\rm TV}} = \frac{1}{2}\lVert \theta - \nu \rVert_{{\rm TV}} < \epsilon \wedge \delta.
\end{align*}
    \item $\theta \not\in B_{{\rm TV}}(\nu,\epsilon)$. 
      Then take
	  \begin{align*}
      \rho_{\theta} = \nu +  \frac{\epsilon}{\lVert \theta-\nu \rVert_{\rm TV}}(\theta-\nu) = \left(1 -  \frac{\epsilon}{\lVert \theta-\nu \rVert_{\rm TV}}\right) \nu + \frac{\epsilon}{\lVert \theta-\nu \rVert_{\rm TV}} \theta,
	  \end{align*}
which gives
\begin{align*}
	\lVert \rho_\theta - \nu \rVert_{{\rm TV}} = \epsilon \quad \text{ and } \quad \lVert \rho_\theta - \theta \rVert_{{\rm TV}} = \lVert \theta - \nu \rVert_{{\rm TV}} - \epsilon \le \delta.
\end{align*}
  \end{enumerate}
Thus, for any $\theta \in B_{{\rm TV}}(\nu,\epsilon + \delta)$ it follows from the above claim and from \eqref{eq:cont_ent} that
  \begin{align*}
  H(\theta) \le H(\rho_{\theta}) + \omega(\delta) \le H_{\epsilon}(\nu) + \omega(\delta).
  \end{align*}
  Taking the supremum over $B_{{\rm TV}}(\nu,\epsilon+\delta)$ yields the second inequality in the first statement of the proposition. The first inequality is obvious.
  
For the second statement, let $\rho,\nu$ with $\rho \in B_{{\rm TV}}(\nu,\delta)$ be given. Consider an arbitrary $\theta \in B_{{\rm TV}}(\rho,\epsilon)$. Then
  \begin{align*}
  \lVert \theta-\nu \rVert_{{\rm TV}} \le \lVert \theta-\rho \rVert_{{\rm TV}} 
  + \lVert \rho-\nu \rVert_{{\rm TV}} \le \epsilon + \delta.
  \end{align*}
  Thus, $\theta \in B_{{\rm TV}}(\nu,\epsilon+\delta)$ and
  \begin{align*}
  H(\theta) \le H_{\epsilon+\delta}(\nu) \le H_{\epsilon}(\nu) + \omega_1(\delta).
  \end{align*}
  Taking supremum over $\theta \in B_{{\rm TV}}(\rho,\epsilon)$ gives
  \begin{align*}
  H_{\epsilon}(\rho) - H_{\epsilon}(\nu) \le \omega(\delta).
  \end{align*}
Repeating the same argument again for $\theta \in B_{{\rm TV}}(\nu,\epsilon)$ gives the other direction.
\end{proof}

\subsection{Tensorized estimates}

The following construction lies at the heart of Liu's tensorization technique. 
Fix some $M,N \in \bN$ and $\mu \in \cP_M$ and consider the set
\begin{align*}
	\cC^{(M)}_{N,\mu} \deq \set[\Big]{\bm{t} = (t_1,\dots,t_{MN})\in T^{MN}: \hat{P}_{\bm{t}} = \mu}.
\end{align*}
In other words, $\cC^{(M)}_{N,\mu}$ is the set of tuples $\bm{t} \in T^{MN}$, where each $t \in T$ occurs exactly $MN\mu(t)$ times. It is then easy to see that the process $(\bm{X}_{\bm{t}})_{\bm{t} \in \cC^{(M)}_{N,\mu}}$ is stationary with respect to the natural action of the symmetric group $S_{MN}$ of permutations of $[MN]$ on the set $\cC^{(M)}_{N,\mu}$:
\begin{align*}
	\pi(\bm{t}) \deq (t_{\pi(1)},\dots,t_{\pi(MN)}), \qquad \text{for all }\pi \in S_{MN}, \bm{t} = (t_1,\dots,t_{MN}) \in \cC^{(M)}_{N,\mu}.
\end{align*}
This group action is transitive by virtue of the definition of $\cC^{(M)}_{N,\mu}$, and it evidently leaves the canonical metric $d_{MN}(\cdot,\cdot)$ invariant. One can then apply the methods of Section~\ref{sec:finite_T_stationary} to the sequence of such processes as $M,N \to \infty$.

In preparation for the use of Liu's tensorization in the main proofs, we first obtain estimates for 
$\eF_\beta(X;\mu)$ via asymptotic quantities defined on permutation-symmetric 
sets. The first result is a variational representation of the quenched free energy for probability 
measures via certain limiting quantities.
\begin{lemma}
  \label{lem:var_rep}
 For any $\mu \in \eP(T)$,
\begin{align*}
 \eF_\beta(X ; \mu) = \sup_{\nu \in \eP(T)} 
  \set[\Bigg]{\bm{\phi}_\beta(\nu) - \frac{1}{\beta}D(\nu \| \mu)},
\end{align*}
where, for any $\mu\in \eP(T)$, 
\begin{align*}
  \bm{\phi}_\beta(\mu)
  \deq \lim_{\epsilon\downarrow 0} \bm{\phi}_{\beta,\epsilon}(\mu)
  \deq \lim_{\epsilon\downarrow 0}\lim_{N \to \infty} 
  \frac{1}{N}\eF_\beta(\bm{X}^N ; \su_{\cC_{N,\mu,\epsilon}}),
\end{align*}
with $\su_{\cC_{N,\mu,\epsilon}}$ denoting the uniform probability measure on $\cC_{N,\mu,\epsilon}$. In particular, for any $\mu_\beta^*$ that achieves the supremum
$\sup_{\mu\in \eP(T)}\eF(X;\mu)$,
\begin{align*}
  \sup_{\mu\in \eP(T)}\eF_{\beta}(X;\mu)=\eF_\beta(X ; \mu_\beta^*) = \bm{\phi}_\beta(\mu_\beta^*)=\sup_{\mu\in \eP(T)}\bm{\phi}_{\beta}(\mu).
\end{align*}
\end{lemma}

\begin{proof}
  Without loss of generality, we can restrict the supremum to the set $\eP_{\mu}(T)\deq\set{\nu\in \eP(T): \nu \ll \mu}$,
  since for $\nu$ with $D(\nu\|\mu)=\infty$ we can never achieve the supremum value.
  The proofs of some auxiliary results used below are given in Appendix~\ref{app:proofs}. For a process $Z = (Z_s)_{s \in S}$ with a finite index set $S$, any $A \subseteq S$, and $\beta \ge 0$, we define
\begin{align*}
	\Psi_\beta(Z;A) \deq \frac{1}{\beta}\log \sum_{s \in A} e^{\beta Z_s}.
\end{align*}
First, we obtain preliminary estimates for $\eF_\beta(X ; \mu)$
  via tensorized processes:
\begin{proposition}\label{prop:bd_tensor}
For any $\mu\in \eP(T)$,
\begin{align*}
\eF_\beta(X ; \mu)
 \le  \liminf_{\epsilon\downarrow 0} \liminf_{N \to \infty}
 \sup_{\nu\in \eP_\mu(T)}\set[\Bigg]{\frac{1}{N}\Ex{\Psi_{\beta}(\bm{X}^N; \cC_{N,\nu,\epsilon} )} 
   +  \frac{\langle\nu, \log \mu\rangle}{\beta}} .
\end{align*}
For $\mu \in \eP_+(T)$,
\begin{align*}
\eF_\beta(X ; \mu)
&\ge \limsup_{\epsilon\downarrow 0}\limsup_{N \to \infty} \sup_{\nu\in \eP(T)}
\set[\Bigg]{\frac{1}{N} \Ex{\Psi_\beta(\bm{X}^N ; \cC_{N, \nu, \epsilon})}
  + \frac{\langle \nu, \log \mu \rangle}{\beta} 
+ O(\epsilon)}.
\end{align*}
\end{proposition}
\noindent Then we prove the existence of the limits of relevant asymptotic quantities (see Section~\ref{ssec:info_th_prelims} for relevant information-theoretic definitions):
\begin{proposition}
  \label{prop:limit}
  For any $\nu\in \eP(T), \epsilon >0$, the following limits exist:
\begin{align*}
  &  \lim_{N \to \infty} \frac{1}{N} 
  \Ex{\Psi_\beta(\bm{X}^N; \cC_{N,\nu, \epsilon})}\eqd\bm{\psi}_{\beta, \epsilon}(\nu)\\
  & \lim_{\epsilon \downarrow 0} \bm{\psi}_{\beta,\epsilon}(\nu)\eqd\bm{\psi}_\beta(\nu)\\ 
  & \lim_{N\to \infty} 
  \frac{1}{N} \log \abs{\cC_{N,\nu, \epsilon}}=H_{\epsilon}(\nu)\\
  &\lim_{\epsilon \downarrow 0}H_{\epsilon}(\nu)=H(\nu)  \\
  & \lim_{N \to \infty} \frac{1}{N} 
  \Ex{\Phi_\beta(\bm{X}^N;  \su_{\cC_{N,\nu, \epsilon}})}
  \eqd\bm{\phi}_{\beta,\epsilon}(\nu) 
  = \bm{\psi}_{\beta,\epsilon} (\nu) - \frac{1}{\beta} H_\epsilon(\nu)\\
  &\lim_{\epsilon \downarrow 0} \bm{\phi}_{\beta,\epsilon}(\nu)
  \eqd \bm{\phi}_\beta(\nu) = \bm{\psi}_\beta(\nu) -\frac{1}{\beta} H(\nu) 
\end{align*}
\end{proposition}
It remains to show we can exchange the limit as $\epsilon \downarrow 0$ and the supremum over $\nu \in \eP(T)$ without incurring non-vanishing error terms.
For the lower bound, Propositions~\ref{prop:bd_tensor} and \ref{prop:limit} give, for $\mu\in \eP_+(T)$,
\begin{align*}
\eF_\beta(X;\mu)
&\ge \limsup_{\epsilon\downarrow 0}\limsup_{N \to \infty} \sup_{\nu\in \eP(T)}
\set[\Bigg]{\frac{1}{N} \Ex{\Psi_\beta(\bm{X}^N ; \cC_{N, \nu, \epsilon})}
  + \frac{\langle \nu, \log \mu \rangle}{\beta} 
+ O(\epsilon)}\\
&\ge \sup_{\nu\in \eP(T)}
\set[\Bigg]{\limsup_{\epsilon \downarrow 0} \limsup_{N \to \infty}
  \Big({\frac{1}{N} \Ex{\Psi_\beta(\bm{X}^N ; \cC_{N, \nu, \epsilon})}
  + \frac{\langle \nu, \log \mu \rangle}{\beta} 
+ O(\epsilon)}\Big)}\\
&=\sup_{\nu\in \eP(T)}\set[\Bigg]{\bm{\psi}_\beta(\nu)+ \frac{\inner{\nu}{\log \mu}}{\beta}}.
\end{align*}
Since $\bm{\psi}_\beta(\nu) = \bm{\phi}_\beta(\nu) + \frac{1}{\beta}H(\nu)$, we obtain 
\begin{align}
  \label{eq:P_plus}
	\eF_\beta(X;\mu) \ge \sup_{\nu\in\eP(T)} \set[\Bigg]{\bm{\phi}_\beta(\nu) - \frac{1}{\beta}D(\nu \| \mu)}.
\end{align}
For $\mu\in \eP(T)$, consider $\mu_{\epsilon} = (1-\epsilon) \mu + \epsilon \rho $ for some $\rho \in \eP_+(T)$.
Fix an arbitrary $\nu\in \eP(T)$ such that $\nu\ll \mu$.
By \eqref{eq:P_plus},
\begin{align*}
	\eF_\beta(X;\mu_{\epsilon}) \ge  \bm{\phi}_\beta(\nu) - \frac{1}{\beta}D(\nu \| \mu_{\epsilon}).
\end{align*}
Notice that on $S_\nu$, $D(\nu\| \mu_{\epsilon})\to D(\nu\| \mu)$ as $\epsilon\downarrow 0$. Since 
$\mu\mapsto \eF_\beta(X;\mu)$ is continuous, 
taking $\epsilon \downarrow 0$, we have the lower bound
\begin{align*}
  \eF_\beta(X;\mu) \ge \sup_{\nu\in \eP_\mu(T)} \set[\Big]{\bm{\phi}_\beta(\nu)- \frac{1}{\beta}D(\nu \| \mu)} 
\end{align*}

We turn to the upper bound next. For any $\delta \in (0, \frac{1}{4})$, let $\cS_\delta$ be a
$\delta$-net for $\eP_{\mu}(T)$ with respect to $\lVert \cdot \rVert_{\rm TV}$, 
which is guaranteed to exist since $\eP(T)$ is compact.
Then for any $\nu \in \eP_{\mu}(T)$, 
$\min_{\rho \in \cS_\delta} \lVert \rho-\nu \rVert_{\rm TV} \le \delta $ .
Let
\begin{align*}
  \psi_{N}(\nu, \epsilon) \deq \frac{1}{N} \Ex{\Psi_\beta(\bm{X}^N; \cC_{N,\nu, \epsilon})}.
\end{align*}
Then
\begin{align*}
&\sup_{\nu\in \eP_{\mu}(T)} \left\{ \psi_N(\nu, \epsilon) 
+ \frac{\langle \nu, \log \mu \rangle}{\beta}\right\} \\
&= \max_{\rho\in \cS_{\delta}}\sup_{\nu\in B_{{\rm TV}}(\rho,\delta)} 
\left\{ \psi_N(\nu, \epsilon)-\psi_N(\rho, \epsilon) + \psi_N(\rho, \epsilon)
 + \frac{\langle \nu-\rho, \log\mu \rangle}{\beta} 
 + \frac{\langle \rho, \log \mu \rangle}{\beta}\right\}\\
&\le \max_{\rho \in \cS_\delta} \left\{\psi_N(\rho, \epsilon) 
+ \frac{\langle \rho, \log \mu \rangle}{\beta} \right\}
+ \max_{\rho \in \cS_{\delta}}\sup_{\nu\in B_{{\rm TV}}(\rho,\delta)} 
\left\{\psi_N(\nu, \epsilon) - \psi_N(\rho, \epsilon)\right\} \\
& \qquad \qquad + \sup_{\rho,\nu \in \eP_{\mu}(T) : \lVert \rho-\nu  \rVert_{{\rm TV}}\le \delta} 
\frac{\abs{\langle \nu-\rho, \log \mu \rangle}}{\beta}\\
&\le \max_{\rho \in \cS_\delta} \left\{\psi_N(\rho, \epsilon) 
+ \frac{\langle \rho, \log \mu \rangle}{\beta}\right\} 
+ \max_{\rho \in \cS_{\delta}} \set{\psi_N(\rho, \epsilon+\delta) - \psi_N(\rho, \epsilon)} \\
& \qquad \qquad + \sup_{\rho,\nu \in \eP_{\mu}(T) : \lVert \rho-\nu  \rVert_{{\rm TV}}\le \delta} 
\frac{\abs{\langle \nu-\rho , \log \mu\rangle}}{\beta},
\end{align*}
where in the last inequality we have used the fact that
\begin{align*}
	\sup_{\nu \in B_{{\rm TV}}(\rho,\delta)} \psi_N(\nu,\epsilon) \le \psi_N(\rho,\epsilon+\delta).
\end{align*}
This can be seen as follows: For any $\rho,\nu$ with $\lVert \rho - \nu \rVert_{{\rm TV}} \le \delta$ and any $\bm{t} \in \cC_{N,\nu,\epsilon}$,
\begin{align*}
  \lVert\hat{P}_{\bm{t}}- \rho\rVert_{{\rm TV}}
  \le \lVert\hat{P}_{\bm{t}}- \nu\rVert_{{\rm TV}} + \lVert \rho-\nu \rVert_{{\rm TV}} \le \epsilon + \delta,
\end{align*}
which implies $ \cC_{N,\nu, \epsilon} \subset \cC_{N,\rho, \epsilon+ \delta}$.
Note that, since $\epsilon \mapsto \psi_N(\nu,\epsilon+\delta)$ is increasing
as $\epsilon$ approaches $0$ from the right with $\delta > 0$ fixed, 
the right limit of $\bm{\psi}_\beta(\nu,\epsilon+\delta)$ exists:  
\begin{align*}
  \bm{\psi}_{\beta,\delta^+}(\nu) \deq \lim_{\epsilon \downarrow 0} 
  \bm{\psi}_{\beta}(\nu, \epsilon+\delta).
\end{align*}
Therefore, by Proposition~\ref{prop:limit},
\begin{align*}
&\limsup_{\epsilon\downarrow 0}\limsup_{N \to \infty} \sup_{\nu \in \eP_{\mu}(T)}
\left\{\psi_N(\nu, \epsilon) + \frac{\langle \nu, \log \mu \rangle}{\beta}\right\}\\
&\le \max_{\rho \in \cS_\delta} \left\{\bm{\psi}_\beta(\rho)
+ \frac{\langle \rho, \log \mu \rangle}{\beta}\right\} 
+ \max_{\rho \in \cS_\delta} \left\{\bm{\psi}_{\beta,\delta^+}(\rho)-\bm{\psi}_\beta(\rho)\right\}
+ \sup_{\rho,\nu \in \eP_{\mu}(T) : \lVert \rho-\nu \rVert_{{\rm TV}}\le \delta} 
  \frac{\abs{\langle \nu-\rho, \log \mu\rangle}}{\beta}\\
&\le \sup_{\nu \in \eP_{\mu}(T)} \left\{\bm{\psi}_\beta(\rho)+ \frac{\langle \nu, \log \mu \rangle}{\beta}\right\} 
+ \sup_{\nu \in \eP(T)} \left\{\bm{\psi}_{\beta,\delta^+}(\nu)-\bm{\psi}_\beta(\nu)\right\}
+ \sup_{\rho,\nu \in \eP_{\mu}(T) : \lVert \rho-\nu \rVert_{{\rm TV}}\le \delta} 
\frac{\abs{\langle \nu-\rho, \log \mu \rangle}}{\beta}
\end{align*}
where we have used the fact that $\cS_\delta$ is finite to interchange the order of limsup and max in the first inequality. 

We next show that
the second term vanishes as $\delta \to 0$.
Consider any $\nu \in \eP(T)$.
As shown in the proof of Proposition~\ref{prop:cont_H_eps},
for any $\theta \in B_{{\rm TV}}(\nu,\epsilon+\delta)$ there exists $\rho \in B_{{\rm TV}}(\nu,\epsilon)$,
such that $\lVert \rho-\theta \rVert_{\rm TV} \le \delta$.
Thus, for any $\bm{t} \in \cC_{N,\nu, \epsilon + \delta}$
there exists some $\rho \in B_{{\rm TV}}(\nu,{\epsilon})$,
so that $\lVert\hat{P}_{\bm{t}}- \rho\rVert_{\rm TV} \le \delta$.
Then we find $\bm{s}\in T^N$ so that $\lVert\hat{P}_{\bm{s}}- \rho\rVert_{\rm TV}
\le \frac{\abs{T}}{N}$, where the existence
of such an $\bm{s}$ can be shown as in \eqref{eq:round}. 
For sufficiently large $N$ such that $\abs{T}/N \le \delta \wedge \epsilon $ ,
\begin{align*}
  &\lVert\hat{P}_{\bm{s}}-\nu\rVert_{{\rm TV}} \le \lVert\hat{P}_{\bm{s}}-\rho\rVert_{\rm TV} 
  + \lVert \rho-\nu \rVert_{\rm TV}\le \epsilon + \frac{\abs{T}}{N} \le 2 \epsilon\\
  &\lVert\hat{P}_{\bm{s}}-\hat{P}_{\bm{t}}\rVert_{\rm TV}\le 
  \lVert\hat{P}_{\bm{t}}-\rho\rVert_{\rm TV} 
  + \lVert\rho-\hat{P}_{\bm{s}}\rVert_{\rm TV} \le \delta + \frac{\abs{T}}{N} \le 2\delta. 
\end{align*}
By Lemma~\ref{lem:match_couple}, we can find 
$\hat{\bm{t}}\in \cC_{N,\hat{P}_{\bm{s}}}$ such that 
\begin{align*}
  \frac{1}{N} d_{\rm H}(\hat{\bm{t}}, \bm{t}) 
  = \lVert\hat{P}_{\bm{t}}- \hat{P}_{\bm{s}}\rVert_{\rm TV} \le 2 \delta.
\end{align*}
Since $\hat{\bm{t}}\in \cC_{N,\hat{P}_{\bm{s}}}$, 
\begin{align*}
  \lVert\hat{P}_{\hat{\bm{t}}}- \nu\rVert_{\rm TV} 
  = \lVert\hat{P}_{\bm{s}}- \nu\rVert_{\rm TV}  \le 2\epsilon
  \implies \hat{\bm{t}} \in \cC_{N, \nu, 2\epsilon}.
\end{align*}
Thus, for any $\bm{t}\in \cC_{N, \nu, \epsilon+\delta}$,
we can find some $\hat{\bm{t}}\in \cC_{N, \nu, 2\epsilon}$ such that $d_{\rm H}(\bm{t}, \hat{\bm{t}}) \le 2N \delta$, implying that
\begin{align*}
  \cC_{N,\nu, \epsilon + \delta} \subset 
  \bigcup_{\hat{\bm{t}}\in \cC_{N,\nu, 2\epsilon}} B_{{\rm H}}(\hat{\bm{t}},\floor{2N\delta})
\end{align*}
where the $\floor{2N\delta}$-ball is with respect to the Hamming distance on $T^N$.
Then
\begin{align*}
\psi_N(\nu, \epsilon + \delta) 
&\le \frac{1}{N\beta} \Ex\, {\log \sum_{\bm{s}\in \cC_{N, \nu, 2 \epsilon}} 
       \sum_{\bm{t}\in B_{{\rm H}}(\bm{s},\floor{2N \delta})} 
e^{\beta (\bm{X_t}^N-\bm{X_s}^N + \bm{X_s}^N)}}\\
&\le \frac{1}{N\beta} \Ex\,{\log \max_{\bm{s}\in \cC_{N,\nu, 2\epsilon} } 
\sum_{\bm{t} \in B_{{\rm H}}(\bm{s},\floor{2N\delta})} e^{\beta (\bm{X_t}^N-\bm{X_s}^N)} }
+ \frac{1}{N\beta} \Ex\,{\log \sum_{\bm{t}\in \cC_{N,\nu, 2\epsilon}} e^{\beta \bm{X_t}^N}}\\
&\le \frac{1}{N} \Ex{\max_{\bm{s} \in \cC_{N, \nu, 2\epsilon} } 
\max_{\bm{t}\in B_{{\rm H}}(\bm{s},\floor{2N\delta})} (\bm{X_t}^N-\bm{X_s}^N)} 
+ \max_{\bm{s}\in \cC_{N, \nu, 2\epsilon}} 
\frac{\log \abs{B_{{\rm H}}(\bm{s},\floor{2N\delta})}}{N\beta} + \psi_N(\nu, 2\epsilon).
\end{align*}
Notice that, for each $(\bm{s},\bm{t})\in \bm{A}\deq 
\set{(\bm{s}, \bm{t}): \bm{s}\in 
\cC_{N,\nu, \epsilon}, \bm{t}\in B_{{\rm H}}(\bm{s},\floor{2N\delta})}$,
$(\bm{X_t}^N-\bm{X_s}^N)$ is a centered Gaussian random variable with 
\begin{align*}
  \Ex{(\bm{X_t}^N-\bm{X_s}^N)^2} 
  = \sum_{i \le n} d^2(t_i, s_i) \le \Delta^2 \floor{2N\delta} .
\end{align*}
By the maximal inequality for Gaussians, 
\begin{align*}
\psi_N(\nu, \epsilon + \delta) 
&\le 2 \Delta 
\sqrt{ \frac{\delta}{N}\log \abs{\cC_{N, \nu, 2\epsilon}}
+  \delta h(2\delta)}
+ \frac{h(2\delta)}{\beta} + \psi_N(\nu, 2\epsilon),
\end{align*}
where we have used the binomial estimates in Proposition~\ref{prop:binom_est} to upper-bound the cardinality of $\bm{A}$. Taking $N \to \infty$ and $\epsilon\downarrow 0$ 
and using Proposition~\ref{prop:limit}, we obtain for any $\nu\in \eP(T)$
\begin{align*}
 \bm{\psi}_{\beta,\delta^+}(\nu)
&\le  \bm{\psi}_\beta(\nu) + r(\delta),
\end{align*}
where $r(\delta) \to 0$ as $\delta \downarrow 0$. Therefore, by the estimate above,
\begin{align*}
&\limsup_{\epsilon \downarrow 0}\limsup_{N\to \infty} 
\sup_{\nu\in \eP_{\mu}(T)} \left\{ \psi_N(\nu, \epsilon) 
+ \frac{\langle \nu, \log \mu \rangle}{\beta}\right\} \\
& \qquad \le \sup_{\nu \in \eP_{\mu}(T)} \left\{ \bm{\psi}_\beta(\nu) 
+ \frac{\langle \nu, \log \mu \rangle}{\beta}\right\}
+r(\delta)
+ \frac{2 \delta \sup_{t\in S_{\mu}}\abs{\log \mu(t)}}{\beta}, \qquad \text{for all } \delta < \frac{1}{4}.
\end{align*}
Taking $\delta $ to $0$ 
and using the fact $\bm{\psi}_\beta(\nu) = \bm{\phi}_\beta(\nu) + 
\frac{1}{\beta} H(\nu)$ gives
\begin{align*}
	\limsup_{\epsilon \downarrow 0}\limsup_{N\to \infty} 
	\sup_{\nu\in \eP_{\mu}(T)} \set[\Bigg]{ \psi_N(\nu, \epsilon) 
	+ \frac{\langle \nu, \log \mu \rangle}{\beta}} \le \sup_{\nu \in \eP_{\mu}(T)}\set[\Bigg]{\bm{\phi}_\beta(\nu) - \frac{1}{\beta}D(\nu \| \mu)}.
\end{align*}
Combining this with the upper bound on $\eF_\beta(X;\mu)$ from Proposition~\ref{prop:bd_tensor} leads to
\begin{align*}
	\eF_\beta(X; \mu) \le  \sup_{\nu \in \eP_{\mu}(T)}\set[\Bigg]{\bm{\phi}_\beta(\nu) - \frac{1}{\beta}D(\nu \| \mu)}.
\end{align*}

It remains to prove the equality for $\mu^*_\beta$.
Let $\phi_\beta(\mu) \deq \eF_\beta(X;\mu)$.
First observe that
\begin{align*}
  \sup_{\mu\in \eP(T)}\phi_\beta(\mu) = \phi_\beta(\mu^*_\beta)
  = \sup_{\nu\in \eP(T)}\set[\Bigg]{ \bm{\phi}_\beta(\nu)- \frac{1}{\beta}D(\nu\|\mu^*_\beta)} 
  = \sup_{\nu \in \eP(T)} \bm{\phi}_\beta(\nu).
\end{align*}
Indeed, 
\begin{align*}
  \sup_{\nu} \bm{\phi}_\beta(\nu)
  \ge \sup_{\nu} \set[\Bigg]{ \bm{\phi}_\beta(\nu)- \frac{1}{\beta}D(\nu\|\mu^*_\beta)}
  = \phi_\beta(\mu^*_\beta)
  = \sup_{\nu} \phi_\beta(\nu) \ge \sup_{\nu}\bm{\phi}_\beta(\nu)
\end{align*}
where the last inequality uses the fact that
\begin{align*}
	\phi_\beta(\nu) = \sup_{\nu'} \set[\Bigg]{\bm{\phi}_\beta(\nu') - \frac{1}{\beta} D(\nu' \| \nu)} \ge \bm{\phi}_\beta(\nu).
\end{align*}
Now assume for the sake of contradiction that, for an arbitrary small $\epsilon>0$,
there exists $\nu^*\in \eP(T)$ 
so that $D(\nu^*\|\mu^*_\beta)> \epsilon\beta$ and
\begin{align*}
  \bm{\phi}_\beta(\nu^*) - \frac{1}{\beta} D(\nu^*\|\mu^*_\beta) > \sup_{\nu}\set[\Bigg]{\bm{\phi}_\beta(\nu)- \frac{1}{\beta}D(\nu\|\mu^*_\beta)} - \epsilon.
 \end{align*}
Since $\sup_{\nu}\set{\bm{\phi}_\beta(\nu) - \beta^{-1} D(\nu \| \mu^*_\beta)} = \sup_\nu \bm{\phi}_\beta(\nu)$, this would imply that
 \begin{align*}
 \bm{\phi}_\beta(\nu^*) - \frac{1}{\beta} D(\nu^*\|\mu^*_\beta) &> \sup_{\nu} \bm{\phi}_\beta(\nu) - \epsilon  \ge \bm{\phi}_\beta(\nu^*) - \epsilon,
\end{align*}
contradicting the assumption that $D(\nu^*\|\mu^*_\beta) > \epsilon\beta$. Thus the claim follows.
\end{proof}

\noindent
Notice that $\bm{\phi}_\beta$ is defined as a right limit of 
$\bm{\phi}_{\beta, \epsilon}$. For $\epsilon>0$, no matter how small,
$\cC_{N,\mu,\epsilon}$ is not an exact $\mu$-type class, and thus 
$(\bm{X_t})_{\bm{t}\in \cC_{N, \mu, \epsilon}}$ is not stationary.
It remains to establish a continuity property of $\bm{\phi}_\beta$
in order to reduce the analysis to the stationary Gaussian case. For each $\mu\in \eP_{\bQ}(T)$, let $\cI_\mu$ be a subsequence of $\bN$ 
so that the probability masses of $N\mu$ are integers for each $N\in \cI_\mu$,
and let $\displaystyle\lim_{\cI_\mu\ni N \to \infty} (\dots)$ denote the subsequential limit along $\cI_\mu$. Recalling our discussion in the beginning of this section, we see that the process $(\bm{X}_{\bm{t}})_{\bm{t} \in \cC_{N,\mu}}$ is stationary for each $N \in \cI_\mu$.

\begin{lemma}
  \label{lem:limit_via_rational}
For $\mu\in \eP_{\bQ}(T)$, the limit 
\begin{align*}
    \lim_{\cI_\mu\ni N \to \infty} 
  \frac{1}{N}\eF_\beta(\bm{X}^N, \su_{\cC_{N,\mu}}) \eqd\bm{\varphi}_\beta(\mu)
\end{align*}
exists and $\bm{\varphi}_\beta$ is locally uniformly continuous.
Thus, $\bm{\varphi}_\beta$ extends uniquely to a continuous function 
$\tilde{\bm{\varphi}}_\beta$ on $\eP(T)$.
  Moreover, for any $\mu \in \eP(T)$ and small $\epsilon>0$,
  \begin{align*}
    \bm{\phi}_{\beta, \epsilon}(\mu) 
    = \sup_{\nu\in \eP_{\bQ}(T)\cap B_{{\rm TV}}(\mu, \epsilon)} \bm{\varphi}_\beta(\nu)
  \end{align*}
  and $\bm{\phi}_\beta = \tilde{\bm{\varphi}}_\beta $.
\end{lemma}
\begin{proof}
The existence of the limit can be shown as in Proposition~\ref{prop:limit}.
It suffices to show the continuity of 
$\mu \mapsto \lim_{\cI_\mu\ni N \to \infty}
\frac{1}{N} \Ex \Psi_\beta(\bm{X}; \cC_\mu) $ 
due to \eqref{eq:cont_ent}.
For any $\mu, \nu\in \eP_{\bQ}(T)$ such that $\norm{\mu-\nu}_{\rm TV}\le \epsilon$,
choose $N\in \cI_{\mu}\cap \cI_{\nu}$ so that $\epsilon N \le \frac{1}{2}$.
Then $\mu, \nu\in \cP_N$. 
By Lemma~\ref{lem:match_couple},
for any $\bm{t}\in \cC_{N,\mu}$ there exists some $\bm{s} \in \cC_{N, \nu} $,
such that $d_{\rm H}(\bm{t}, \bm{s}) \le N \epsilon$. In other words,
\begin{align*}
  \cC_{N,\mu}\subset \bigcup_{\bm{s} \in \cC_{N,\nu}}B_{{\rm H}}(\bm{s},\floor{\epsilon N} ).
\end{align*}
Then
\begin{align*}
 & \frac{1}{N} \Ex \Psi_\beta(\bm{X}^N ;\cC_{N,\mu}) \\
  &\le \frac{1}{N\beta}  \Ex \Big[ \log \sum_{\bm{s} \in \cC_{N,\nu} } 
  \sum_{t\in B_{{\rm H}}(\bm{s}, \floor{N \epsilon})} 
    e^{\beta (\bm{X_t}-\bm{X_s}+\bm{X_s})}\Big]\\
  &\le \frac{1}{N} \Ex \Psi_\beta(\bm{X}^N ; \cC_{N,\nu}) + \frac{1}{N\beta}\Ex \Big[ \max_{\bm{s}\in \cC_{N,\nu}} \log 
  \sum_{\bm{t}\in B_{{\rm H}}(\bm{s},\floor{N \epsilon})}e^{\beta (\bm{X_t}-\bm{X_s})}\Big]\\
  &\le \frac{1}{N} \Ex \Psi_\beta(\bm{X}^N ; \cC_{N,\nu}) + \frac{1}{N} \Ex \Big[ \max_{\bm{s}\in \cC_{N,\nu}} 
  \max_{\bm{t}\in B_{{\rm H}}(\bm{s},\floor{N \epsilon})} (\bm{X_t}-\bm{X_s})\Big] 
  + \max_{\bm{s}\in \cC_{N,\nu}}\frac{\log \abs{B_{{\rm H}}(\bm{s},\floor{N\epsilon})}}{N\beta}\\
  &\le \frac{1}{N} \Ex \Psi_\beta(\bm{X}^N ; \cC_{N,\nu}) +  2 \Delta 
  \sqrt{\frac{\epsilon\log \abs{\cC_{N,\nu}}}{N}+ \epsilon h(\epsilon)} 
  + \frac{h(\epsilon)}{\beta},
\end{align*}
where the last inequality uses the maximal inequality for Gaussians 
and the volume estimates of the Hamming ball in Proposition~\ref{prop:binom_est}. 
Using the estimate on the $\Psi$ functional, we deduce 
\begin{align}\label{eq:Phi_var}
  \frac{1}{N} \eF_\beta(\bm{X}^N, \su_{\cC_{N,\mu}})
  \le 
  \frac{1}{N} \eF_\beta(\bm{X}^N,  \su_{\cC_{N,\nu}}) 
+  \omega_N(\epsilon),
\end{align}
where
\begin{align*}
	\omega_N(\epsilon) \deq \frac{\log \abs{\cC_{N,\nu}}-\log \abs{\cC_{N,\mu}}}{N \beta} +2 \Delta 
	  \sqrt{\frac{\epsilon\log \abs{\cC_{N,\nu}}}{N}+ \epsilon h(\epsilon)} 
	  + \frac{h(\epsilon)}{\beta}.
\end{align*}
Taking the limit along 
$\cI_{\mu}\cap \cI_{\nu}$ yields 
\begin{align*}
  \bm{\varphi}_\beta(\mu) \le \bm{\varphi}_\beta(\nu) + \overline{\omega}(\epsilon),
\end{align*}
where 
\begin{align*}
	\overline{\omega}(\epsilon) \deq   \frac{\omega(\epsilon)}{\beta}+ 2\Delta 
	  \sqrt{\epsilon \log \abs{T}+ \epsilon h(\epsilon)} 
	  + \frac{h(\epsilon)}{\beta}
\end{align*}
is independent of $\mu,\nu$ and converges to $0$ as $\epsilon \downarrow 0$.
By similar arguments, we have
\begin{align}\label{eq:varphi_var}
  \sup_{\eP_{\bQ}(T)\ni\mu, \nu:\,\norm{\mu-\nu}_{\rm TV} \le \epsilon} 
  \abs{\bm{\varphi}_\beta(\mu)-\bm{\varphi}_\beta(\nu)}
  \le \overline{\omega}(\epsilon).
\end{align}
Therefore, $\bm{\varphi}_\beta$ 
is locally uniformly continuous on $\eP_{\bQ}(T)$ 
and its continuous extension $\tilde{\bm{\varphi}}_\beta$ is locally uniformly continuous 
on $\eP(T)$.

Next we show the relation between $\bm{\phi}_\beta$ and $\bm{\varphi}_\beta$.
  Fix a small $\epsilon>0$. 
  For a rational measure $\nu\in B_{{\rm TV}}(\mu, \epsilon)$, we consider $N\in \cI_{\nu}$.
  By set inclusion
  $\cC_{N,\nu}\subset \cC_{N,\mu, \epsilon}$, we have 
  \begin{align*}
  \frac{1}{N}\eF_\beta(\bm{X}^N; \su_{\cC_{N,\mu,\epsilon}})  
  \ge  \frac{1}{N}  \eF_\beta(\bm{X}^N; \su_{\cC_{N,\nu}}).
  \end{align*}
  Taking $N \to \infty$ along $\cI_{\nu}$ gives the lower bound.
  For the reverse inequality, we fix a large $M \in \bN$ and consider $N\in M \bN$.
  we have
  \begin{align*}
  \frac{1}{N}\eF_\beta(\bm{X}^N; \su_{\cC_{N, \mu, \epsilon}})  
  &= \frac{1}{N\beta} \Ex \Big[\log \sum_{\bm{t}\in \cC_{N,\mu,\epsilon}} 
   \frac{1}{\abs{\cC_{N,\mu,\epsilon}}}e^{\beta \bm{X_t}}\Big]\\
  &= \frac{1}{N\beta} \Ex \Big[\log \sum_{\nu\in \cP_N \cap B_{{\rm TV}}(\mu, \epsilon)} 
    \sum_{\bm{t}\in \cC_{N,\nu}} 
   e^{\beta \bm{X_t}}\Big]
 -\frac{\log \abs{\cC_{N,\mu,\epsilon}}}{N\beta}\\
  &\le \frac{1}{N\beta} \Ex \Big[ \max_{\nu\in \cP_N \cap B_{{\rm TV}}(\mu, \epsilon)} 
  \Phi_\beta(\bm{X}^N;\su_{\cC_{N,\nu}})\Big]+ \frac{\log\abs{\cP_N}}{N\beta}\\
  &\le \max_{\nu\in\cP_N \cap B(\mu, \epsilon)} 
  \frac{1}{N}\eF_\beta(\bm{X}^N; \su_{\cC_{N,\nu}}) + \frac{2 \Delta}{\beta} 
  \sqrt{\frac{\log \abs{\cP_N}}{N}} + \frac{\log\abs{\cP_N}}{N\beta}
  \end{align*}
  where the last step uses the subgaussian maximal inequality.
  Similar to \eqref{eq:round}, for any $\nu\in \cP_N \cap B_{{\rm TV}}(\mu, \epsilon)$,
  we find $\hat{\nu}\deq \hat{\nu}(\nu)\in \cP_M\subset \cP_N$ so that 
  $\norm{\hat{\nu}-\nu}_{\rm TV}\le \frac{\abs{T}}{M}\eqd \delta_M$
  and $\norm{\hat{\nu}-\mu}_{\rm TV}\le \epsilon+\delta_M$.
  In other words, $\set{\hat{\nu}(\nu): \nu\in \cP_N \cap B_{{\rm TV}}(\mu, \epsilon)}
  \subset \cP_M \cap B_{{\rm TV}}(\mu, \epsilon+\delta_M)$. 
  Applying the estimate \eqref{eq:Phi_var} for each pair $(\nu,\hat{\nu})$ and taking 
 the maximum, we have
  \begin{align*}
   \max_{\nu\in\cP_N \cap B_{{\rm TV}}(\mu, \epsilon)} 
  \frac{1}{N}\eF_\beta(\bm{X}^N; \su_{\cC_{N,\nu}}) 
  \le
   \max_{\rho\in\cP_M \cap B_{{\rm TV}}(\mu, \epsilon+\delta_M)} 
  \frac{1}{N}\eF_\beta(\bm{X}^N; \su_{\cC_{N,\rho}})  + \omega_N(\delta_M).
  \end{align*}
  Since $M\bN$ is a subsequence of $\cI_\rho$ for any 
  $\rho\in \cP_M \cap B_{{\rm TV}}(\mu, \epsilon+\delta_M)$, taking $N \to \infty$ along 
  $M\bN$ yields
  \begin{align*}
    \bm{\phi}_{\beta, \epsilon}(\mu) 
    \le \max_{\nu \in \cP_M \cap B_{{\rm TV}}(\mu, \epsilon+ \delta_M)}
    \bm{\varphi}_\beta(\nu) + \overline{\omega}(\delta_M).
  \end{align*}
  As in the proof of Proposition~\ref{prop:cont_H_eps}, for sufficiently large 
  $M$ (e.g. $\delta_M \le \epsilon/2$), we can find a 
  $\rho_{\nu}\in B_{{\rm TV}}(\mu, \epsilon-\delta_M)$ for each 
  $\nu\in B_{{\rm TV}}(\mu, \epsilon+ \delta_M)$, such that 
  $\norm{\rho_\nu-\nu}_{\rm TV}\le 2\delta_M$.
  Each $\rho_\nu$ can be rounded to type $\hat{\rho}_\nu\in \cP_M$ so that 
  $\norm{\hat{\rho}_\nu-\rho_\nu}_{\rm TV}\le \delta_M$,
  $\norm{\hat{\rho}_\nu-\mu}_{\rm TV}\le \epsilon$,
  and $\norm{\hat{\rho}_\nu-\nu}_{\rm TV}\le 3 \delta_M$. 
  It follows that 
  \begin{align*}
    \cP_M \cap B_{{\rm TV}}(\mu, \epsilon+\delta_M) 
    \subset \bigcup_{\rho \in B_{{\rm TV}}(\mu,\epsilon)\cap \cP_M}
       (B_{{\rm TV}}(\rho, 3\delta_M)\cap \cP_M)
  \end{align*}
  and
  \begin{align*}
  \bm{\phi}_{\beta, \epsilon}(\mu) 
  \le \sup_{\rho \in B_{{\rm TV}}(\mu, \epsilon)\cap \eP_{\bQ}(T)}
  \sup_{\nu\in B_{{\rm TV}}(\rho,3\delta_M)\cap\eP_{\bQ}(T)} \bm{\varphi}_{\beta}(\nu) + \overline{\omega}(\delta_M).
  \end{align*}
  By \eqref{eq:varphi_var}, the variation of $\bm{\varphi}_\beta$ on rational measures
  within $3\delta_M$ is controlled by $\overline{\omega}(3\delta_M)$, implying that 
  \begin{align*}
  \bm{\phi}_{\beta, \epsilon}(\mu) 
  \le \sup_{\rho \in B_{{\rm TV}}(\mu, \epsilon)\cap \eP_{\bQ}(T)}
   \bm{\varphi}_{\beta}(\rho) + \tilde{\omega}(\delta_M),
  \end{align*}
  where $\tilde{\omega}(\delta) \to 0$ as $\delta \downarrow 0$.  Since $\delta_M$ can be arbitrarily small, we prove the reverse inequality.

Finally, we prove the statement that $\bm{\phi}_\beta = \tilde{\bm{\varphi}}_\beta$. It is clear that, by definition, 
  $\bm{\phi}_{\beta, \epsilon}(\mu) \ge \tilde{\bm{\varphi}}_\beta(\mu)$ 
  for any $\mu\in \eP(T)$. Taking $\epsilon\downarrow 0$ gives the lower bound. 
  For the upper bound, we consider 
  a sequence of rational measures $(\mu_k)_{k}\subset\eP_{\bQ}(T)$ that converges to 
  $\mu\in \eP(T)$.  For any $\nu \in \eP_{\bQ}(T) \cap B_{\rm TV}(\mu,\epsilon)$, the estimate
  \begin{align*}
	  \bm{\varphi}_\beta(\nu) \le \bm{\varphi}_\beta(\mu_k) + \overline{\omega}(\epsilon + \lVert \mu_k - \mu \rVert_{\rm TV})
\end{align*}
holds for all sufficiently large $k$, by \eqref{eq:varphi_var}. Taking the supremum of both sides over all such $\nu$, we obtain
	  \begin{align*}
		  \bm{\varphi}_{\beta,\epsilon}(\mu) \le \bm{\varphi}_\beta(\mu_k) + \overline{\omega}(\epsilon + \lVert \mu_k - \mu \rVert_{\rm TV})
	\end{align*}
Taking $k\to \infty$ and then letting 
  $\epsilon\downarrow 0$ yields the desired lower bound and completes the proof.
\end{proof}

\subsection{The proof of Theorem~\ref{thm:nonstarionary_1}}

We are now in a position to prove our first main result on nonstationary Gaussian processes. By the same reasoning as in the proof of Theorem~\ref{thm:stationary_1}, it suffices to show the upper and lower bounds
  for the quenched free energy. We first consider rational measures.
For any $\mu\in \eP_{\bQ}(T)$ and any $N \in \cI_\mu$, applying Theorem~\ref{thm:stationary_2} to the stationary process $\bm{X}^N = (\bm{X_t}^N)_{\bm{t} \in \cC_{N,\mu}}$ and scaling,
we have
\begin{align*}
  \frac{1}{N}\eF_\beta(\bm{X}^N;\su_{\cC_{N,\mu}})
  \asymp \int_0^\Delta 
    \Bigg(\sqrt{\frac{\log \sN(\cC_{N,\mu}, \sqrt{N} \epsilon)}{N}}
         \wedge \beta \epsilon\Bigg) \dif \epsilon.
\end{align*} 
Since $\sqrt{\log \sN(\cC_{N,\mu}, \sqrt{N}\epsilon)/N}\le \sqrt{\log \abs{T}}$, 
we have absolute integrability. Upon taking the limit along $\cI_\mu$, 
we obtain 
\begin{align*}
  \bm{\varphi}_\beta(\mu) \asymp \int_0^\Delta 
  \Big(\sqrt{R_{\mu}(\epsilon^2)}\wedge \beta \epsilon\Big)\dif \epsilon,
\end{align*}
where we use dominated convergence theorem for the upper bound,
Fatou's lemma for the lower bound, and the following estimate, which immediately follows from Lemma~6 and Eq.~(26) in
\citet{liuSimpleSharpGeneralization2025}:
\begin{lemma}
Let $\mu$ be rational. For any $\epsilon>0$, we have 
\begin{align*}
&\limsup_{\cI_{\mu}\ni N\to \infty} 
\frac{1}{N} \log \sN(\cC_{N,\mu}, \sqrt{N}\epsilon) 
\le R_{\mu}\left(\frac{\epsilon^2}{16}\right)\\
&\liminf_{\cI_{\mu}\ni N\to \infty} 
\frac{1}{N} \log \sN(\cC_{N,\mu}, \sqrt{N}\epsilon)
\ge R_\mu(\epsilon^2).
\end{align*}
\end{lemma}
\noindent
Using the continuity of $\mu \mapsto {\bm{\phi}}_\beta(\mu)$ (Lemma~\ref{lem:limit_via_rational}) and $\mu \mapsto R_\mu(\eps^2)$ \citep{palaiyanur_sahai_2008}
and Lemma~\ref{lem:var_rep},
we obtain
\begin{align*}
   \int_0^\Delta 
   \Big(\sqrt{R_{\mu_{\beta}^*}(\epsilon^2)}\wedge \beta \epsilon\Big)\dif \epsilon\asymp \eF_{\beta}(X;\mu_{\beta}^*) 
   & =\sup_{\mu\in \eP(T)}\eF_\beta(X;\mu)\\
  &= \sup_{\mu\in \eP(T)}\bm{\phi}_\beta(\mu) \asymp \sup_{\mu\in \eP(T)} \int_0^\Delta 
  \Big(\sqrt{R_{\mu}(\epsilon^2)}\wedge \beta \epsilon\Big)\dif \epsilon.
\end{align*}

\subsection{The soft Fernique functional}

Let $X$ be a random element of $\Reals^T$ with probability law $P_X$ and let $\mu \in \eP(T)$ be given. The following functional of $P_X$ and $\mu$ was introduced by \citet{Fernique_1975,Fernique_1978} in the context of his analysis of expected suprema of Gaussian processes:
\begin{align*}
	F(P_X,\mu) \deq \sup_{(X,\tau) \atop X \sim P_X, \tau \sim \mu} \Ex X_\tau.
\end{align*}
Here, the supremum is taken over all couplings of $P_X$ and $\mu$. When we consider a \textit{particular} coupling $(X,\tau)$ of $P_X$ and $\mu$, we will refer to the corresponding expected value $\Ex X_\tau$ as a \textit{soft} Fernique functional. For subgaussian processes, Liu obtains an upper bound
on the soft Fernique functional via a mutual-information 
truncated rate-distortion integral:

\begin{theorem}[\citet{liuSimpleSharpGeneralization2025}]\label{thm:lift_upper}
Let $X = (X_t)_{t\in T}$ be a centered subgaussian process
indexed by a finite set $T$. Let $\tau$ be a random element of
 $T$ jointly distributed with $X$, such that $P_\tau=\nu$.
Then 
\begin{align*}
  \Ex X_\tau \lesssim \int_0^\Delta 
  \sqrt{R_{\nu}(\epsilon^2)\wedge I(X;\tau)} \dif \epsilon.
\end{align*}
\end{theorem}
\noindent
We show that, as an immediate consequence of Theorem~\ref{thm:nonstarionary_1},
the bound above is tight and can be reversed when 
$X = (X_t)_{t\in T}$ is a Gaussian process and  when
$\tau=\uptau_{\beta,\mu}$ with 
reference measure $\mu$.
(Thus,
$P_{\uptau_{\beta,\mu}}=\overline{\sm}_{\beta,\mu}$.)
\begin{theorem}\label{thm:lift_lower}
Let $X = (X_t)_{t\in T}$ be a centered Gaussian process with a finite index set $T$.
Then we have
\begin{align*}
  \Ex X_{\uptau_{\beta,\mu}} \gtrsim \int_0^\Delta 
  \sqrt{R_{\overline{\sm}_{\beta,\mu}}(\epsilon^2)\wedge I(X;\uptau_{\beta,\mu})} \dif \epsilon .
\end{align*}
\end{theorem}
\begin{proof}
  Denote $\overline{u}\deq \sqrt{I(X;\uptau_{\beta,\mu})}$. 
  Let ${\rm Leb}$ be the Lebesgue measure.
  By Lemma~\ref{lem:var_rep}, the proof of Theorem~\ref{thm:nonstarionary_1}, and layer-cake representation, 
  there exists some
  universal constant $c>0$ so that
  \begin{align*}
    \Ex X_{\uptau_{\beta,\mu}} 
    &\ge \eF_\beta(X;\mu)\\
  &\ge \bm{\phi}_\beta(\overline{\sm}_{\beta,\mu})-\frac{1}{\beta}D(\overline{\sm}_{\beta,\mu}\|\mu)\\
  &\ge c \int_0^\Delta 
  \Big(\sqrt{R_{\overline{\sm}_{\beta,\mu}}(\epsilon^2)}\wedge \beta \epsilon\Big)\dif \epsilon -\frac{1}{\beta}D(\overline{\sm}_{\beta,\mu}\|\mu)\\
 & \ge c\int_0^\Delta 
  \Big(\sqrt{R_{\overline{\sm}_{\beta,\mu}}(\epsilon^2)}
    \wedge \beta \epsilon \wedge \overline{u} \Big)\dif \epsilon-\frac{1}{\beta}D(\overline{\sm}_{\beta,\mu}\|\mu)\\
  &= c\int_0^\infty 
  {\mathrm{Leb}}\big[\set{\epsilon>0:\sqrt{R_{\overline{\sm}_{\beta,\mu}}(\epsilon^2)}
    \wedge \beta \epsilon \wedge \overline{u}> u}  \big] \dif u -\frac{1}{\beta}D(\overline{\sm}_{\beta,\mu}\|\mu) \\
  &= c\int_0^{\overline{u}}
  {\mathrm{Leb}}\big[\set{\epsilon>0:\sqrt{R_{\overline{\sm}_{\beta,\mu}}(\epsilon^2)}
    \wedge \beta \epsilon> u}  \big] \dif u -\frac{1}{\beta}D(\overline{\sm}_{\beta,\mu}\|\mu)\\
  &\ge c\int_0^{\overline{u}}
    {\mathrm{Leb}}\big[\set{\epsilon:\sqrt{R_{\overline{\sm}_{\beta,\mu}}(\epsilon^2)}> u}  \big] \dif u 
  -c\int_0^{\overline{u}}\frac{u}{\beta} \dif u -\frac{1}{\beta}D(\overline{\sm}_{\beta,\mu}\|\mu)\\
  &=c \int_0^\Delta \sqrt{R_{\overline{\sm}_{\beta,\mu}}(\epsilon^2)\wedge I(X ; \uptau_{\beta,\mu})} 
  \dif \epsilon - c\frac{I(X;\uptau_{\beta,\mu})}{2\beta}-\frac{1}{\beta}D(\overline{\sm}_{\beta,\mu}\|\mu)
  \end{align*}
  By \eqref{eq:identities}, it follows that
\begin{align*}
  \Big(1+ \frac{c}{2} \vee 1\Big) \Ex X_{\uptau_{\beta,\mu}} 
  &\ge \Ex X_{\uptau_{\beta,\mu}} 
  + \Big(\frac{c}{2}\vee 1\Big)\frac{I(X;\uptau_{\beta,\mu})+D(\overline{\sm}_{\beta,\mu}\| \mu)}{\beta} \\
  &\gtrsim \int_0^\Delta \sqrt{R_{\overline{\sm}_{\beta,\mu}}(\epsilon^2)\wedge I(X ; \uptau_{\beta,\mu})} 
\dif \epsilon.
\end{align*}
This completes the proof.
\end{proof}

\subsection{An alternative approach via code-measure correspondence}

In the preceding sections, we used the symmetry
of type classes to reduce the analysis of the general nonstationary case to the stationary case after tensorization.
In this section, we give an alternative combinatorial argument for the rate-distortion integral bound. 
It can be observed that the rate-distortion function is a natural quantity after tensorization due to code-measure correspondence \citep{kontoyiannis_arbitrary_2002}. This is a concept from  information theory that, for our purposes, states that $\log \frac{1}{\mu(B(t,\eps))}$ is (roughly) proportional to the number of bits one would use to localize $t \in T$ to a ball of radius $\eps$ relative to a `prior' probability measure $\mu$ on $T$.
However, we can only recover the lower bound for $\Ex X_{\uptau_{\beta}^*}$, where $\uptau_{\beta}^* \deq \uptau_{\beta,\mu_\beta^*}$,
mainly because applying Talagrand's partition scheme
requires the objective functional to be monotone with respect to set inclusion.

We will prove a one-shot version and an asymptotic version of upper and lower bounds.
For a random element $\tau$ of $T$ with distribution $\nu$ and for $\epsilon>0$, let
\begin{align*}
\sM_\nu(\mu,\epsilon) 
&\deq \Ex_{\tau\sim \nu} \sqrt{\log \frac{1}{\mu(B(\tau,\epsilon))}}\\ 
\sM_\nu^*(\epsilon) &\deq \inf_{\mu\in \eP(T)} \sM_\nu(\mu, \epsilon)
\end{align*}
\begin{theorem}[one-shot]
  \label{thm:one-shot}
Let $X = (X_t)_{t\in T}$ be a centered subgaussian process
with a finite index set $T$. Let $\tau$ be a random element of $T$ jointly distributed with $X$, such that $P_\tau=\nu$.
Then 
\begin{align*}
\Ex{X_\tau} \lesssim \Delta + \int_0^\Delta 
\Big(\sM_{\nu}^*(\epsilon) \wedge \sqrt{I(X ; \tau) }\Big) \dif \epsilon,
\end{align*}
where, as before, $\Delta$ denotes the diameter of $(T,d)$. If $X$ is a Gaussian process, 
then there exists a probability measure $\mu\in \eP(T)$, such that 
\begin{align*}
\eF_\beta(X;\mu_\beta^*) + \Delta
&\gtrsim \sup_{t\in T}\int_0^\Delta 
 \Big(\sM_{\delta_t}(\mu,\epsilon)\wedge \beta \epsilon\Big) \dif \epsilon
\end{align*}
In particular,
\begin{align*}
\Ex X_{\uptau_\beta^*}+ \Delta \gtrsim 
\sup_{\nu\in\eP(T)}\int_0^\Delta 
\Big(\sM_\nu^*(\epsilon)\wedge
\sqrt{I(X ; \uptau_\beta^*) }\Big) \dif \epsilon.
\end{align*}
\end{theorem}

\begin{proof} Without loss of generality, assume $\Delta =1$.
Fix some $r\ge 2$. For each $t\in T$ and $k \ge 1$, let $B_k(t) \deq B(t,r^{-k})$.
Given $(X,\tau)$ with joint law $P_X \otimes P_{\tau|X}$,
we consider a tuple of random elements $(X,\tau,\tau_k,\tau_{k-1})$
and specify their joint law by 
\begin{align*}
P_{X \tau \tau_k \tau_{k-1}} = P_X \otimes P_{\tau|X} \otimes P_{\tau_k|\tau} \otimes P_{\tau_{k-1}|\tau},
\end{align*}
where
\begin{align*}
P_{\tau_{k}|\tau} (\cdot) = \frac{\mu_k (\cdot \cap B(\tau, r^{-k}))}{\mu_k(B(\tau, r^{-k}))}
\end{align*}
for some arbitrary $\mu_k \in \eP(T)$.
It is clear that $P_{\tau_0|\tau}=\mu_0$.
Since $T$ is finite, there exists some $K>0$, such that $P_{\tau_K|\tau}=\delta_\tau$. For each $k$,
\begin{align*}
\Ex[X_{\tau_k}-X_{\tau_{k-1}}\mid X]
&= \Ex\big[\Ex[X_{\tau_k}-X_{\tau_{k-1}}\mid \tau, X]\mid X\big].
\end{align*}
Applying Proposition~\ref{prop:decorr} to $\cX = T \times T$, $\mu = P_{\tau_k|\tau} \otimes P_{\tau_{k-1}|\tau}$, $\nu = \mu_k \otimes \mu_{k-1}$, $f(u,v) = d(u,v)$, and $g(u,v) = \frac{X_u - X_v}{d(u,v)}$, we obtain
\begin{align*}
&\Ex[X_{\tau_k}-X_{\tau_{k-1}}\mid \tau, X]
=\sum_{u,v\in T} (X_u-X_v)P_{\tau_k|\tau}( u)P_{\tau_{k-1}|\tau}(v)\\
&\le r^{-k+1} \sqrt{2 \log \Big(1+\frac{1}{\mu_k(B_k(\tau))\mu_{k-1}(B_{k-1}(\tau))}\Big)}
+ r^{-k+1}\sum_{u,v\in T} \Big(e^{\frac{\abs{X_u-X_v}^2}{d^2(u,v)}}-1\Big) 
\mu_k( u) \mu_{k-1}( v).
\end{align*}
Thus, using the increment condition for a subgaussian process, we have
\begin{align*}
  \Ex[X_{\tau_k}-X_{\tau_{k-1}}] \lesssim r^{-k+1} 
  \Ex{\sqrt{\log \frac{1}{\mu_k(B_k(\tau))}}
  +\Ex\sqrt{\log \frac{1}{\mu_{k-1}(B_{k-1}(\tau))}} +1}.
\end{align*}
Alternatively, by the mutual information bound (see, e.g., \citet{russo_controlling_2016,xu_raginsky_2017}), 
\begin{align*}
\Ex[X_{\tau_k}-X_{\tau_{k-1}}] \le r^{-k+1} \sqrt{2 I(X ; (\tau_k, \tau_{k-1}))} 
\le r^{-k+1}\sqrt{2 I(X ; \tau)}.
\end{align*}
Combining both estimates, we obtain
\begin{align*}
\Ex[X_{\tau_k}-X_{\tau_{k-1}}] 
\lesssim r^{-k+1} \left(\Ex{\sqrt{\log \frac{1}{\mu_k(B_\tau(\tau))}\wedge I(X ; \tau)}}
+\Ex{\sqrt{\log \frac{1}{\mu_{k-1}(B_{k-1}(\tau))}\wedge I(X ; \tau)}}+ 1\right).
\end{align*}
Summing over $k$ yields
\begin{align*}
\Ex{X_\tau} \lesssim
\Delta+ \sum_{k \ge 1} r^{-k+1} 
\Ex{\sqrt{\log \frac{1}{\mu_k(B_k(\tau))}\wedge I(X ; \tau)}}. 
\end{align*}
Since each $\mu_k$ can be chosen independently of others,
\begin{align*}
\Ex{X_\tau} 
&\lesssim
\Delta+ \sum_{k \ge 1} r^{-k+1} \inf_{\mu} 
\Ex{\sqrt{\log \frac{1}{\mu(B_k(\tau))}\wedge I(X ; \tau)}} \\
&\le \Delta + \int_0^\Delta \inf_\mu 
\Ex{\sqrt{\log \frac{1}{\mu(B(\tau, \epsilon))}\wedge I(X ; \tau)}} \dif \epsilon,
\end{align*}
where the last inequality is due to the fact that $\epsilon \mapsto \inf_\mu \Ex{\sqrt{\log \frac{1}{\mu(B(\tau, \epsilon))}}}$ is a decreasing function,
so that the discrete sum can converted to integral form as in \eqref{eq:sum_int}.

Now we prove the lower bound.
The proof follows closely the greedy construction of nested partitions by \citet{Talagrand_1992} .
Let $r \ge 4$.
Let $\cA_0 = \set{T}$. 
Let $\theta_0 \deq \theta^{\cA_0}\in \eP(\cA_0)$ be such that $\theta_0(T)=1$. 
Assume $\cA_0,\ldots,\cA_{k-1}$ is constructed and $\theta_0,\ldots,\theta_{k-1}$ 
be constructed. 
Choose an arbitrary $A\in \cA_{k-1}$.
We need to construct a partition for $A$ and a measure for $\cA_{k-1}(A)$.
Let 
\begin{align*}
\zeta_\beta (A) \deq \sup_{\mu\in \eP(A)}\eF_\beta(X; \mu).
\end{align*}
Let $B_0=A$. Choose $t_1\in B_0$ so that
\begin{align*}
\zeta_\beta (B_0 \cap B(t_1, r^{-k-1}/2))= \max_{t\in B_0}\zeta_\beta (B_0 \cap B(t, r^{-k-1}/2)).
\end{align*}
Let $A_1= B_0 \cap B(t, r^{-k}/2)$ and $B_1= B_0 \setminus A_1$. Now iterate this:
If, for $i \ge 1$,$B_{i-1} \neq \emptyset$, choose $t_i\in B_{i-1}$ so that
\begin{align*}
  \zeta_\beta(B(t_i, r^{-k-1}/2)\cap B_{i-1}) = \max_{t\in B_{i-1}}\zeta_\beta(B(t, r^{-k-1}/2)\cap B_{i-1}).
\end{align*}
Let $A_i = B(t_i, r^{-k/2})\cap B_{i-1}$ and $B_i = B_{i-1}\setminus A_i$ and continue. This procedure is guaranteed
to terminate. 
Repeat the construction for the other sets in $\cA_{k-1}$.
Thus, for $A\in \cA_{k-1}$, we have a partition $A_1,\ldots,A_m$ for $A$,
and packing points $t_1,\ldots,t_m$ since for $i\neq j$, $d(t_i, t_j) \ge r^{-k}/2$. 
Let $H_i \deq B(t_i, r^{-k-1}/2)\cap B_{i-1}$ and $\beta_k=\beta r^{-k}$.
Using Proposition~\ref{prop:one_step_low} with $\sigma=r^{-k}/8$, we have
\begin{align*}
\zeta_\beta(A) 
&\ge c r^{-k} \sqrt{ \eta_i (\beta_{k})} + \zeta_\beta(H_i)\\
&= c r^{-k} \sqrt{\log (i  \wedge \beta_k^2)} + \zeta_\beta(H_i) .
\end{align*}
Let $i_k(A_k(t))$ be the index inherited from the construction steps. The inequality above
can be written as:
\begin{align*}
\zeta_\beta(A_{k-1}(t))
\ge c r^{-k} \sqrt{\log i_k(A_k(t))\wedge \beta_{k}^2} + \zeta_\beta (A_{k+1}(t)).
\end{align*}
Using $1+a \wedge b \ge 1 \wedge b + a \wedge b \ge (a+1)\wedge b$, we obtain
\begin{align*}
\zeta_\beta(A_{k-1}(t)) + r^{-k} \ge c r^{-k} \sqrt{ \log (i_k(A_k(t))+1)^2 \wedge \beta_k^2} + \zeta_\beta(A_{k+1}(t)).
\end{align*}
Since $\sum_{i \ge 1} \frac{1}{(i+1)^2} <1$, there exists some probability measure $\theta^{\cA_k(A)}\in \eP(\cA_{k}(A))$
so that $\log (i+1)^2 \gtrsim \log \frac{1}{\theta^{\cA_k(A)}(A_i)}$ for $A_i\in \cA_k(A)$.
Then we can define $\theta_k\in \eP(\cA_k)$ so that
\begin{align*}
  \theta_k(B)= \theta_{k-1}(A) \theta^{\cA_k(A)}(B)\quad\text{for any } B \in \cA_k(A), A\in \cA_{k-1} .
\end{align*}
Notice that $\theta_k(A)=\theta_{k-1}(A)$, thus $\theta_k$ can be considered as a refinement for the previous level.
Replacing the index with conditional measure in the above inequality yields
\begin{align*}
\zeta_\beta(A_{k-1}(t)) + r^{-k} \ge c r^{-k} \sqrt{ \log \frac{\theta_k(A_{k-1}(t))}{\theta_k(A_k(t))} 
\wedge \beta_k^2} + \zeta_\beta(A_{k+1}(t)).
\end{align*}
Applying the above inequality iteratively on $\zeta_\beta(T)$ and $\zeta_\beta(A_1(t))$ 
and repeating the steps in the proof of Theorem~\ref{thm:stationary_2}, we conclude
that there exists a measure $\theta\in \eP(T)$ so that 
\begin{align*}
\sup_{\mu\in\eP(T)}\eF_\beta(X ; \mu) + \Delta
&\gtrsim \sup_{t\in T} \int_0^\Delta 
\sqrt{\log \frac{1}{\theta(B(t, \epsilon))}\wedge \beta^2 \epsilon^2}  \dif \epsilon.
\end{align*}
Thus, for any random element $\tau$ of $T$, 
\begin{align*}
 \sup_{\mu\in\eP(T)}\eF_\beta(X; \mu) + \Delta
&\gtrsim  \int_0^\Delta \inf_{\mu\in\eP(T)} 
\Ex{\sqrt{\log \frac{1}{\mu(B(\tau, \epsilon))}
    \wedge \beta^2 \epsilon^2}}  \dif \epsilon \\
&\ge  \int_0^\Delta \inf_{\mu\in\eP(T)} 
\Ex{\sqrt{\log \frac{1}{\mu(B(\tau, \epsilon))}
    \wedge \beta^2 \epsilon^2 \wedge I(X ; \uptau_\beta^*)}} 
\dif \epsilon.
\end{align*}
The proof is complete by applying a calculation similar to that in the proof of 
Theorem~\ref{thm:lift_lower}.
\end{proof}

The bounds of Theorem~\ref{thm:one-shot} give a sharp characterization for 
$\Ex{X_{\uptau_\beta^*}}+\Delta$ for Gaussian processes.
In the supremum case, the term $\Delta$ can be absorbed by $\Ex [\sup_{t\in T}X_t]$ 
as a multiplicative universal constant, thus recovering the majorizing measure theorem. However, for general $X_\tau$,
they are not comparable. 
One useful observation is that $\Delta$ vanishes in the $N \to \infty$ asymptotics;
thus, we can apply bounds on the tensorized processes 
together with the code-measure correspondence and some results form information theory on the fundamental limits of lossy compression. 
We adapt the converse and achievability bounds from lossy source coding theory as follows. 
\begin{lemma}
  \label{lem:measure_RD}
For any $\nu\in \eP(T)$ and $\bm{\mu}\in \eP(T^N)$,
\begin{align*}
  \liminf_{N\to \infty} \frac{1}{\sqrt{N}} \sM_{\nu^{\otimes N}}(\bm{\mu},\sqrt{N}\epsilon)
  \ge \sqrt{R_{\nu}(\epsilon^2)}.
\end{align*}
For any $\delta>0$ and $\nu\in \eP(T)$, we have 
\begin{align*}
\lim_{N\to \infty}\frac{1}{\sqrt{N}}\sM_{\nu^{\otimes N}}^*(\sqrt{N}\epsilon) 
\le \sqrt{R_{\nu}(\epsilon^2)}. 
\end{align*}
\end{lemma}
\begin{proof}
  The lower bound follows from \citet[Theorem~1]{kontoyiannis_arbitrary_2002}
  and \citet{kiefferSampleConversesSource1991}.
  The upper bound follows from Jensen's inequality 
  and \citet[Theorem~5]{kontoyiannis_arbitrary_2002}. \end{proof}
\noindent 
With the estimates above, we can recover the fixed-rate truncated 
rate-distortion integral bounds in the previous subsection. 
\begin{theorem}[asymptotic]
  \label{thm:asymptotic}
  Using the same notation as in Theorem~\ref{thm:one-shot},
  for any subgaussian $(X_t)$, we have
  \begin{align*}
  \Ex{X_\tau} \lesssim \int_0^\Delta 
  \sqrt{R_\nu(\epsilon^2)\wedge I(X ; \tau)} \dif \epsilon.
  \end{align*}
  If $(X_t)$ is a Gaussian process,
  \begin{align*}
    \Ex{X_{\uptau_\beta^*}} 
    \gtrsim \sup_{\nu\in \eP(T)} 
    \int_0^\Delta \sqrt{R_{\nu}(\epsilon^2)\wedge I(X ; \uptau_\beta^*)} 
    \dif \epsilon.
  \end{align*}

\end{theorem}
\begin{proof}
Now we apply the bounds on tensorized processes.
Let $(\bm{X},\bm{\tau})\sim P_{X\tau}^{\otimes N} $. Then 
 \begin{align*}
	 \Ex{\bm{X}_{\bm{\tau}}}=\Ex{\sum_{i=1}^N \bm{X}^{(i)}_{\bm{\tau}_i}}
  = N \Ex X_\tau.
 \end{align*}
  After applying the upper bounds of Theorem~\ref{thm:one-shot} 
  on $\frac{1}{N}\Ex{\bm{X}_{\bm{\tau}}}$ and letting $N \to \infty$,
  the upper bound follows from the dominated convergence theorem and Lemma~\ref{lem:measure_RD}.
For the lower bound, we apply 
the first inequality in the lower bound of one-shot estimates to
$\eF_\beta (\bm{X}^N, \bm{\mu}_{\beta}^*)=
\sup_{\bm{\mu}\in \eP(T^N)}\eF_\beta(\bm{X}^N;\bm{\mu})$ to conclude that 
there exists a probability measure $\bm{\theta} \in \eP(T^N)$, such that, for any $\nu\in \eP(T)$,
\begin{align*}
 \eF_\beta(\bm{X}^N, \bm{\mu}_{\beta}^*)
  + \sqrt{N}\Delta 
  &\gtrsim 
  \sup_{\bm{t}\in T^N} \int_0^{\sqrt{N}\Delta} 
  \big(\sM_{\delta_{\bm{t}}}
  \big(\bm{\theta},\epsilon)\wedge \beta \epsilon \big)
  \dif \epsilon\\
  &\ge \int_0^\Delta \sqrt{N} 
  \big(\sM_{\nu^{\otimes N}}(\bm{\theta},\sqrt{N}\epsilon\big) 
  \wedge \beta \sqrt{N}\epsilon) \dif \epsilon. 
\end{align*}
Since $\bm{\mu}_\beta^*=(\mu_\beta^*)^{\otimes N}$ due to 
Proposition~\ref{prop:sup_tensor}, 
after multipling both sides with $1 / N$ and letting $N \to \infty$,
we deduce from Fatou's lemma and Lemma~\ref{lem:measure_RD} that,
for any $\nu\in \eP(T)$,
\begin{align*}
\eF_\beta(X; \mu_\beta^*) 
&\gtrsim \liminf_{N \to \infty} 
\int_0^\Delta \big(\sM_{\nu^{\otimes N}}(\bm{\theta}, \sqrt{N}\epsilon)/\sqrt{N}
\wedge \beta \epsilon\big)  \dif \epsilon\\
&\ge \int_0^\Delta \big(\sqrt{R_{\nu}(\epsilon^2)}\wedge \beta \epsilon \big)\dif \epsilon
\ge \int_0^\Delta \big(\sqrt{R_{\nu}(\epsilon^2) \wedge I(X;\uptau_\beta^*)}\wedge \beta \epsilon \big)\dif \epsilon.
\end{align*}
The desired lower bound follows by a calculation similar to the one in the proof of Theorem~\ref{thm:lift_lower}.
\end{proof}

\section{The case of countable $T$}
\label{sec:countable_T}

In this section, we extend Theorem~\ref{thm:nonstarionary_1}, \ref{thm:lift_upper}, and
\ref{thm:lift_lower} to countable index sets. Notice that 
when $T$ is countably infinite, the supremum of $\eF_\beta(X; \mu)$ over $\mu \in \eP(T)$ may not 
be attained. Therefore, we state the bounds for 
$\sup_{\mu\in \eP(T)} \eF_\beta(X;\mu)$ 
and $\sup_{\mu\in \eP(T)} \eG_\beta(X;\mu)$.

\subsection{The proof of Theorem~\ref{thm:countable_softmmt}}

  When $\beta = \infty$, the statement reduces to the characterization for 
  the expected supremum $\Ex [\sup_{t\in T} X_t]$ via the rate-distortion integral,
  which can be recovered from \citet{liuSimpleSharpGeneralization2025}.
  We consider $\beta< \infty$.
  Similar to the finite $T$ case, it suffices to establish the characterization for quenched the free energy 
  due to \eqref{eq:equiv_soft}. 
  Note that by Theorem~\ref{thm:nonstarionary_1}:
  for any finite $S \subset T$,
  \begin{align*}
  \sup_{\mu\in \eP(S)}\eF_\beta(X;\mu) 
  \asymp \sup_{\mu\in \eP(S)} \int_0^\infty \Big(\sqrt{R_{\mu}(\epsilon^2)}\wedge \beta \epsilon\Big) \dif \epsilon.
  \end{align*}
  We would like to show
  \begin{align*}
  \sup_{S \subset T} \sup_{\mu\in \eP(S)}\eF_\beta(X;\mu) 
    &= \sup_{\mu\in \eP(T)} \eF_\beta(X;\mu)\\
  \sup_{S \subset T} \sup_{\mu\in \eP(S)} 
    \int_0^\infty \Big(\sqrt{R_\mu(\epsilon^2)}\wedge \beta \epsilon\Big)\dif \epsilon 
    &=  \sup_{\mu\in \eP(T)}
    \int_0^\infty \Big(\sqrt{R_\mu(\epsilon^2)}\wedge \beta \epsilon\Big)\dif \epsilon.
  \end{align*}
  We use a truncation argument similar to the one by \citet{liuSimpleSharpGeneralization2025}.
  Since $T$ is countable, we enumerate it as $\set{t_0,t_1,\ldots}$ and let 
  $T_k= \set{t_0,\ldots,t_k} $.
  Let $Z$ be a random element of $T$ such that $\eL(Z)=\mu$.
  Define
  \begin{align*}
    Z_k \deq Z{\bf 1}_{\set{Z\in T_k}} + t_0 {\bf 1}_{\set{Z\notin T_k}}
  \end{align*}
  and denote $\mu_k \deq \eL(Z_k)$.
  It is clear that $2\norm{\mu_k-\mu}_{\rm TV}= \mu(T_k^c)\to 0$. 
  Let
  \begin{align*}
	  f(\mu_k)\deq \eF_\beta(X;\mu_k) \quad \text{and} \quad g(\mu_k)\deq \int_0^\infty (R_\mu(\epsilon^2)^{1 / 2} \wedge \beta \epsilon) \dif \epsilon.
\end{align*}
  It suffices to show $f(\mu_k)\to f(\mu)$ and $g(\mu_k)\to g(\mu)$ for any $\mu\in \eP(T)$.
  Indeed, if this is the case then, for each $k$,
  \begin{align*}
	  f(\mu_k) \le \sup_{S \subset T} \sup_{\nu\in \eP(S)} f(\nu);
	\end{align*}
	taking the limit as $k \to \infty$
  followed by the supremum  over $\mu \in \eP(T)$  yields
  \begin{align*}\sup_{\mu\in \eP(T)} f(\mu) \le \sup_{S \subset T} \sup_{\nu\in \eP(S)} f(\nu).
	 \end{align*}
  Then the equality follows since the reverse inequality is clear by definitions.
  Identical arguments apply to the function $g$.

  We show $\lim_{k\to \infty}f(\mu_k)=f(\mu)$. Without loss of generality, let $t_0$ be such that 
  $\mu(t_0) > 0$. Also, since 
  \begin{align*}
	  \Ex\Bigg[\log \sum_{t\in T} \mu(t) \exp(\beta X_t)\Bigg]=\Ex\Bigg[\log \sum_{t\in T} \mu(t) \exp(\beta (X_t-X_{t_0}))\Bigg],
	 \end{align*}
  we can assume $X_{t_0}=0$ a.s.. Let
  \begin{align*}
	  Y \deq \sum_{t\in T} \mu(t) \exp(\beta X_t)
	\end{align*}
  and
  \begin{align*}
	  Y_k \deq \sum_{t\in T_k} \mu(t)\exp(\beta X_t) + \exp(\beta X_{t_0})\mu(T_k^{\rm c}).
	\end{align*}
  Then $f(\mu_k) =\Ex \log Y_k$ and $f(\mu) =\Ex \log Y$. 
  Note that $\Ex Y \le \exp(\beta^2 \Delta^2 / 2)<\infty$, which implies that $Y<\infty$ a.s.\ since $Y > 0$. 
  Therefore, the tail sum
  \begin{align*}
	  \sum_{t\in T_k^{\rm c}}\mu(t) \exp(\beta X_t) \to 0 \text{ a.s.}
	 \end{align*}
	 and
  \begin{align*}
	  \abs{Y_k-Y} \le \sum_{t\in T_k^{\rm c}} \mu(t) \exp(\beta X_{t}) + \mu(T_k^{\rm c}) \to 0 \text{ a.s.}
	 \end{align*}
  Then $\log Y_k \to \log Y$ a.s.. 
  Notice that 
  $\log Y_k \le \log (Y_k + 1) \le Y_k \le Y + 1$
  and 
  $\log Y_k \ge \log \mu(t_0)$. Thus, 
  $\abs{\log Y_k} \le (Y+1) \vee \abs{\log \mu(t_0)} $.
  Then the dominated convergence theorem yields the desired limit.

  To show $\lim_{k\to \infty}g(\mu_k) = g(\mu)$, we establish
  \begin{align*}
    R_{\mu}(\epsilon^2 + \Delta^2 \mu(T_k^{\rm c}))
   \le R_{\mu_k}(\epsilon^2) \le R_{\mu}(\epsilon^2) + 2 h(\mu(T_k^{\rm c})).
  \end{align*}
  The upper bound can be obtained from Eq.~(137) in \citet{liuSimpleSharpGeneralization2025}.
  We prove the lower bound. Let $Q_{\hat{Z}|Z_k}$ be optimal channel (regular conditional probability law of $\hat{Z}$ given $Z_k$) that achieves 
  $R_{\mu_k}(\epsilon^2)$. Then 
  \begin{align*}
    \Ex d^2(Z, \hat{Z})
    &= \Ex [d^2(Z,\hat{Z}){\bf 1}_{\set{Z \in T_k}}]
    + \Ex [d^2(Z,\hat{Z}){\bf 1}_{\set{Z \not\in T_k}}]\\
    &\le \epsilon^2 + \Delta^2 \mu(T_k^{\rm c})
  \end{align*}
  By the data processing inequality, $I(\hat{Z};Z)\le I(\hat{Z}; Z_k)$, which, together with 
  the distortion estimate, gives the lower bound.
  Therefore, 
  \begin{align*}
  \limsup_{k\to \infty} g(\mu_k) \le g(\mu)
  \end{align*}
   by continuity of the binary entropy function 
   and 
   \begin{align*}
   \liminf_{k \to \infty} g(\mu_k)
   \ge \int_0^\Delta
   \Big(\sqrt{\liminf_{k\to \infty} R_\mu(\epsilon^2+ \Delta^2 \mu(T_k^{\rm c}))}
      \wedge \beta \epsilon\Big)\dif \epsilon
      = g(\mu)
   \end{align*}
   by Fatou's lemma and continuity of rate-distortion function except at most countably
   many points \citep{polyanskiyInformationTheoryCoding2024}. 
   
   \subsection{The soft Fernique functional in the countable case}

\begin{theorem}\label{thm:countable_mi_trunc}
Assume that $(X_t)_{t\in T}$ be a centered subgaussian process
indexed with a countable set $T$. Let $\tau$ be a random variable 
on $T$ with arbitrary correlation with $X$ such that $P_\tau=\nu$.
Then 
\begin{align*}
  \Ex X_\tau \lesssim \int_0^\Delta 
  \sqrt{R_{\nu}(\epsilon^2)\wedge I(X;\tau)} \dif \epsilon.
\end{align*}
If $(X_t)_{t\in T}$ is a Gaussian process then we have, for any $\mu\in \eP(T)$,
\begin{align*}
 \Ex X_{\uptau_{\beta,\mu}} 
\gtrsim 
\int_0^\Delta \sqrt{R_{\overline{\sm}_{\beta,\mu}}(\epsilon^2)\wedge I(X; \uptau_{\beta,\mu})} \dif \epsilon.
\end{align*}
\end{theorem}
\begin{proof}
  We adapt the notation from the proof of Theorem~\ref{thm:countable_softmmt}. 
  For the upper bound, 
  let $X^k=(X_{t_i})_{i \le k}$ and
  let $\tau_k = \tau {\bf 1}_{\set{\tau\in T_k}} + t_0{\bf 1}_{\set{\tau \not\in T_k}}$ 
  and denote $\nu_k = P_{\tau_k}$.
  We can assume $\int_0^\infty \Big(\sqrt{R_\nu(\epsilon^2)}\wedge I(X;\tau)\Big)\dif \epsilon<\infty$,
  otherwise there is nothing to prove. Following \citet{liuSimpleSharpGeneralization2025},
  it can be shown that 
  \begin{align*}
  \lim_{k \to \infty} \Ex X_{\tau_k} = \Ex X_\tau.
  \end{align*}
  With the same argument as in the proof of Theorem~\ref{thm:countable_softmmt}, we can obtain 
  \begin{align*}
  \lim_{k\to \infty} \int_0^\infty \sqrt{R_{\nu_k}(\epsilon^2)} \dif \epsilon 
  =\int_0^\infty \sqrt{R_{\nu}(\epsilon^2)} \dif \epsilon.
  \end{align*}
  Then applying Theorem~\ref{thm:lift_upper} on $\Ex X_{\tau_k}$ 
  using $I(X^k;\tau_k) \le I(X; \tau)$, and taking the limit on both sides yields the desired upper bound.
  Next, we prove the lower bound.
  For brevity, let $\sm \deq \sm_{\beta,\mu}$, $\sm_{k} \deq \sm_{\beta, \mu_k}$, and $\uptau \deq \uptau_{\beta,\mu}$.
  Note that 
  \begin{align*}
  \Ex X_{\uptau} \ge \eF_{\beta}(X;\mu) = \lim_{k\to \infty} \eF_{\beta}(X;\mu_k)
  \end{align*}
  where the equality is established in the proof of Theorem~\ref{thm:countable_softmmt}.
  Applying the arguments in Theorem~\ref{thm:lift_lower} 
  with the truncation rate $\overline{u}=\sqrt{I(X;\uptau)}$, we have
  \begin{align*}
     \eF_{\beta}(X;\mu_k) 
    \ge c \int_0^\infty \sqrt{R_{\overline{\sm}_k}(\epsilon^2)\wedge I(X;\uptau_{\beta,\mu})}  
    - c \frac{I(X;\uptau)}{2\beta} - \frac{1}{\beta}D(\overline{\sm}_k\|\mu_k).
  \end{align*}
  If we can show $\norm{\sm_k-\sm}_{\rm TV}\to 0$ a.s.
  and $D(\overline{\sm}_k\| \mu_k)\to D(\overline{\sm}\|\mu)$ as $k\to \infty$, then taking limit on both sides of the estimate 
  above and using \eqref{eq:identities} will complete the proof.
  As in the proof of Theorem~\ref{thm:countable_softmmt}, we assume $\mu(t_0)>0$ and $X_{t_0}=0$ a.s..
  Since
  \begin{align*}
  \sum_{t\in T}\abs{\sm_{k}(t)-\sm(t)}
  &= \abs[\Big]{\frac{(\mu(t_0)+\mu(T_k^{\rm c}))\exp(\beta X_{t_0})}{Y_k}-\frac{\mu(t_0)\exp(\beta X_{t_0})}{Y}}\\
  &\quad + \sum_{t\in T_k \setminus \set{t_0}}\mu(t)\exp(\beta X_t) \abs[\Big]{\frac{1}{Y_k}-\frac{1}{Y}} 
  + \frac{1}{Y}\sum_{t\in T_k^{\rm c}}\mu(t)\exp(\beta X_t)\\
  &\le \frac{\mu(T_k^{\rm c})}{Y_k} + \sum_{t\in T_k }\mu(t)\exp(\beta X_t) \abs[\Big]{\frac{1}{Y_k}-\frac{1}{Y}} 
  + \frac{1}{Y}\sum_{t\in T_k^{\rm c}}\mu(t)\exp(\beta X_t)\\
  &\le\frac{\mu(T_k^{\rm c})}{Y_k} + \frac{\abs{Y_k-Y}}{Y_k} 
  + \frac{1}{Y}\sum_{t\in T_k^{\rm c}}\mu(t)\exp(\beta X_t)
  \end{align*}
  and we have shown that $\mu(T_k^{\rm c})\to 0$, $\abs{Y_k-Y}\to 0$ a.s. 
  and $\sum_{t\in T_k^{\rm c}}\mu(t)\exp(\beta X_t)\to 0$ a.s. as $k\to \infty$,
  we have
  \begin{align*}
    \lim_{k\to \infty}\norm{\sm_{k}-\sm}_{\rm TV}= 0\quad \text{a.s.}.
  \end{align*}
It remains to prove the convergence of the relative entropy term.
  Observe that $\overline{\sm}_k(t)= \Ex \frac{\mu_k(t)\exp(\beta X_t)}{\sum_s \mu_k(s) \exp(\beta X_s)}\le \Ex \frac{\mu_k(t)\exp(\beta X_t)}{\mu(t_0)}$,
  thus $\frac{\overline{\sm}_k(t)}{\mu_k(t)}\le \frac{\Ex \exp(\beta X_t)}{\mu(t_0)}  \le \frac{\exp (\beta^2 \Delta^2 /2)}{ \mu(t_0)} $ for all $t\in T$.
  Since $x\mapsto x \log x$ is bounded on $[0,\frac{\exp (\beta^2 \Delta^2 /2)}{ \mu(t_0)} ]$, we can find a constant $L>0$ so that 
  \begin{align*}
    \sum_{t\in T}\mu_k(t) \abs[\bigg]{\frac{\overline{\sm}_k(t)}{\mu_k(t)} \log \frac{\overline{\sm}_k(t)}{\mu_k(t)}}
  \le L \sum_{t\in T} (\mu(t)+\delta_{t_0}(t)) < \infty.
  \end{align*}
  By the dominated convergence theorem, 
  \begin{align*}
  \lim_{k\to \infty} D(\overline{\sm}_k\|\mu_k) = D(\overline{\sm}\|\mu).
  \end{align*}
(Generally, the relative entropy is only weakly lower-semicontinuous \citep{polyanskiyInformationTheoryCoding2024}. However, in this particular case we can establish convergence by a direct argument making use of convergence in total variation and of uniform boundedness, see the article of \citet{dobrushin_1960} for more details.)
\end{proof}

\section{Example: The Sherrington--Kirkpatrick model}
\label{sec:SK}

As an illustrative example, we show how our results can be applied to the quenched free energy in the Sherrington--Kirkpatrick (SK) model \citep{panchenko_sherrington-kirkpatrick_2013}. For $N \in \{2,3,\dots\}$, let $(G_{ij})_{1 \le i,j \le N}$ be a collection of i.i.d.\ standard Gaussian random variables. Consider the Gaussian process $X  = (X_t)_{t\in T}$ indexed by 
$t = (t_1,\ldots,t_N) \in T = \set{-1, 1}^N$ and defined by
\begin{align*}
X_t \deq \frac{1}{\sqrt{N}} \sum_{1\le i, j \le N} G_{ij}t_i t_j.
\end{align*}
This is the (random) Hamiltonian for the SK model without an external field. For $s,t\in T$, denote the \textit{overlap} of $s,t$ as 
\begin{align*}
R(s,t) \deq \frac{1}{N} \sum_{i\le N} t_i s_i
\end{align*}
Then the covariance of $X$ is $\Ex [X_t X_s]= N R^2(s,t)$ and
the canonical distance induced by $(X_t)$ is 
$d^2(t,s) =2N(1-R^2(s,t)) $. It is a pseudometric since $d(t,-t)=0$. We will write $X^{(N)}$ whenever we need to indicate the dependence on $N$ explicitly.
Let $\bG = \set{-1,1}^N$ be a group under coordinate multiplication. We define an action of $\bG$ on $T$:
for $g=(g_1,\ldots,g_N)\in \bG$ and $t=(t_1,\ldots,t_N)\in T$,
\begin{align*}
g(t) \deq (g_1 t_1,\ldots,g_N t_N).
\end{align*}
It is clear that $(X_t)_{t\in T}$ is a stationary Gaussian process with respect to $\bG$.

Let $\eD$ be the set of cumulative distribution functions on $[0,1]$, i.e., monotone nondecreasing, right-continuous functions $\alpha : [0,1] \to [0,1]$ with $\lim_{u \downarrow 0} \alpha(u) = 0$ and $\lim_{u \uparrow 1}\alpha(u) = 1$. The elements of $\eD$ play the role of (functional) \textit{order parameters} for the SK model. For each $\alpha\in \eD$ and $\beta \ge 0$, let 
 $f_{\alpha,\beta}$ be the solution to the \textit{Parisi PDE},
\begin{align*}
\p_u f_{\alpha,\beta}(u,x) = - \beta^2(\p_{x x}f_{\alpha,\beta}(u,x)
+ \alpha(u)(f_{\alpha,\beta}(u,x))^2) 
\end{align*}
for $(u,x)\in [0,1]\times \bR$ with the terminal condition $f_{\alpha,\beta}(1,x)=\log \cosh x$. 
Define the \textit{Parisi functional} by
\begin{align*}
\sP(\alpha,\beta) \deq f_{\alpha, \beta}(0,0) - \beta^2 \int_0^1 \alpha(u)u \dif u.
\end{align*}
The minimum
\begin{align*}
	\sP(\beta) \deq \min_{\alpha \in \eD} \sP(\alpha,\beta)
\end{align*}
exists and is achieved at a unique $\alpha \in \eD$ \citep{auffinger_parisi_2015}. Let $\su^{(N)}$ denote the uniform probability distribution on $\{-1,+1\}^N$. The Parisi formula states that 
\begin{align*}
\lim_{N\to \infty}\frac{\beta}{N} \eF_\beta(X^{(N)} ; \su^{(N)}) 
= \sP(\beta).
\end{align*}
The upper bound in the Parisi formula
follows from Guerra's replica-symmetry-breaking interpolation, which
relies on the standard Gaussian comparison inequality and Gaussian integration-by-parts 
identities \citep{guerraBrokenReplicaSymmetry2003}.
The lower bound is much more delicate and was first proved by \citet{talagrand_parisi_2006}
in a general setting of mixed even $p$-spin models 
and later extended to mixed $p$-spin models that include odd $p$-spin interactions
by \citet{panchenko_parisi_2014}.
We state a Parisi formula in the finite setting. 
\begin{theorem}\label{thm:parisi}
  For any finite $N \ge 2$, we have
  \begin{align*}
  \frac{\beta}{N}\eF_\beta(X^{(N)}, \su^{(N)}) \asymp \sP(\beta).
  \end{align*}
\end{theorem}
\noindent
Since the replica-symmetry-breaking upper bound holds for any $N$ \citep{guerraBrokenReplicaSymmetry2003},
it suffices to prove a matching lower bound up to some universal constant.
The idea is to utilize the soft entropy integral bound for stationary Gaussian  processes
to construct an explicit order parameter $\alpha_{D} \in \eD$ and show that 
\begin{align*}
\frac{\beta}{N} \eF_\beta(X^{(N)};\su^{(N)}) \gtrsim \sP(\alpha_D, \beta).
\end{align*}
Below we will show that the following choice will suffice:
\begin{align*}
  \alpha_D (u) \deq 
  \sqrt{\frac{(d(\frac{1+u}{2}\|\frac{1}{2})-\frac{\log 2}{2})_+}{\beta^2(1-u^2)}} 
  \wedge 1 
  =\sqrt{\frac{(\frac{\log 2}{2} - h(\frac{1-u}{2}))_+}{\beta^2(1-u^2)}} 
  \wedge 1,
\end{align*}
where, for $p,q \in (0,1)$, $d(p \| q) \deq p \log \frac{p}{q} + (1-p)\log \frac{1-p}{1-q}$ is the relative entropy between ${\rm Bernoulli}(p)$ and ${\rm Bernoulli}(q)$ distributions on $\{0,1\}$.

\begin{proof} Fix an $N \ge 2$.
Notice that 
\begin{align*}
  NR(s,t)=\sum_{i=1}^N t_i s_i 
  = \sum_{\set{i: t_i=s_i}} t_i s_i
    +\sum_{\set{i: t_i\neq s_i}} t_i s_i = N-2  d_{\rm H}(t,s),
\end{align*}
where $d_{\rm H}(s,t) = \sum_{1\le i\le N}  \mathbf{1}_{\set{t_i \neq s_i}}$ is the Hamming distance on $T^N \deq \{-1,+1\}^N$.
Then 
\begin{align}\label{eq:hamming_cano}
N d^2(s,t) = 8 d_{\rm H}(s,t)(N - d_{\rm H}(s,t))
\end{align}
and $\Delta=\diam(T,d) = \sqrt{2N}$. By Theorem~\ref{thm:stationary_2},
\begin{align*}
\frac{1}{N}\eF_\beta(X^{(N)};\su^{(N)})
\gtrsim \frac{1}{N} \int_0^{\sqrt{2N}} 
\big(\sqrt{\log \sN(T, \epsilon, d)} \wedge \beta \epsilon\big) \dif \epsilon
\end{align*}
From Eq.~\eqref{eq:hamming_cano},
\begin{align*}
d(s,t) \le \epsilon \iff 
\Bigg[ d_{\rm H}(s,t) \le\frac{N}{2}-\frac{N}{2}\sqrt{1-\frac{\epsilon^2}{2N}} \quad\text{or}\quad 
d_{\rm H}(s,t) \ge \frac{N}{2}+\frac{N}{2}\sqrt{1-\frac{\epsilon^2}{2N}} \Bigg].
\end{align*}
In other words, 
denoting $\delta\deq\floor[\big]{\frac{N}{2}-\frac{N}{2}\sqrt{1-\frac{\epsilon^2}{2N}} }$, 
we have for any $t\in T $
\begin{align*}
B_d(t,\epsilon) = B_{\rm H}(t,\delta)
\cup B_{\rm H}(-t,\delta).
\end{align*}
Let $S$ be a minimal $\epsilon$-net of $T^N$ with respect to $d$, so that $\abs{S}=\sN(T, \epsilon,d)$.
Then 
\begin{align*}
T \subset \bigcup_{t\in S} B_d(t, \epsilon)
= \bigcup_{t\in S\cup (-S)} 
B_{\rm H}(t,\delta),
\end{align*}
which implies $\sN(T, \delta, d_{\rm H}) 
\le  2\sN(T, \epsilon,d)$. By covering and volume estimates,
\begin{align*}
{V_k}\sN(T, k, d_{\rm H})\ge 2^N
\end{align*}
for $0 \le k \le N$, where $V_k=\sum_{i=1}^k \binom{N}{i}$ is the volume of the Hamming ball of radius $k$.
We can deduce from Proposition~\ref{prop:binom_est} that 
\begin{align*}
\log \sN(T, \delta, d_{\rm H})-\log 2
&\ge N \log 2 - \log V_{\delta}
\ge N \log 2 - N h(\delta / N)\\
&\ge N \log 2 - N h \left(\frac{1}{2}- \frac{1}{2}\sqrt{1-\frac{\epsilon^2}{2N}}\right),
\end{align*}
where the last inequality is due to the fact that the binary entropy function $h$ is 
increasing on $[0,1 / 2]$.
Therefore, combining the estimates above and the fact that $N \ge 2$, we get
\begin{align*}
\frac{\beta}{N}\eF_\beta(X^{(N)}; \su^{(N)})
&\gtrsim \frac{\beta}{N}\int_0^{\sqrt{2N}} 
\Bigg(\sqrt{N\bigg(\frac{\log 2}{2}
      -  h\bigg(\frac{1}{2}- \frac{1}{2} \sqrt{1-\frac{\epsilon^2}{2N}}\bigg)
  \bigg)_+}
\wedge \beta \epsilon\Bigg) \dif \epsilon\\
&= \sqrt{2}\beta\int_0^1
\Bigg(\sqrt{\frac{(\frac{\log 2}{2}- h(\frac{1-u}{2}))_+}{1-u^2}} 
\wedge \sqrt{2} \beta\Bigg) u \dif u\\
&\gtrsim \beta^2 \int_0^1 \alpha_D(u) u \dif u,
\end{align*}
where the equality uses the change of variables $\epsilon=\sqrt{2N(1-u^2)}$.
(It is worth mentioning that this is exactly the relation between 
the canonical distance $d$ and the overlap function $R$.)

It remains to show that
\begin{align*}
\sP(\alpha_D, \beta) \lesssim \beta^2 \int_0^1 \alpha_D(u) u \dif u. 
\end{align*}
We first obtain an upper bound on $\sP(\alpha, \beta)$ for any choice of the order parameter $\alpha\in \eD$. To that end, we use a variational representation of $f_{\alpha,\beta}$ 
by \citet[Theorem~3]{auffinger_parisi_2015}:
\begin{align*}
f_{\alpha,\beta}(0,0)
=\max_{Y} \Ex 
 \Bigg[\log \cosh\left(\beta^2 \int_0^1 2\alpha(u)  Y_u \dif u 
 + \sqrt{2}\beta \int_0^1 \dif W_u\right)- \beta^2\int_0^1 \alpha(u) Y^2_u \dif u\Bigg]
\end{align*}
where the maximum is taken over all progressively measurable processes $Y = (Y_u)_{u \in [0,1]}$ with respect to 
the filtration generated by standard Brownian motion $W = (W_u)_{u \in [0,1]}$.
Using Lipschitz continuity of $\log \cosh$ and the inequality $2 \abs{Y} \le Y^2 + 1$,
we have
\begin{align*}
f_{\alpha,\beta}(0,0)
&\le \max_{Y} \Ex  
\Big[\log \cosh(\sqrt{2}\beta G) + \beta^2\int_0^1 2\alpha(u) \abs{Y_u} \dif u
- \beta^2 \int_0^1 \alpha(u) Y^2_u \dif u\Big]\\
&\le \Ex [\log \cosh(\sqrt{2}\beta G)] + \beta^2 \int_0^1 \alpha(u) \dif u,
\end{align*}
where $G$ is a standard Gaussian random variable. Since
\begin{align*}
\Ex[\log \cosh (\sqrt{2}\beta G)]=\beta\eF_\beta(2G ; \su_{[2]})
\lesssim \beta^2 \wedge \beta
\end{align*}
and
\begin{align*}
\beta^2 \int_0^1 \alpha(u)\dif u
\le 2\beta^2 \int_{\frac{1}{2}}^1 \alpha(u) \dif u 
\le 2\beta^2 \int_{\frac{1}{2}}^1 2\alpha(u) u \dif u 
\le 4 \beta^2 \int_0^1 \alpha(u) u \dif u,
\end{align*}
it suffices to show 
\begin{align*}
\beta^2 \int_0^1 \alpha_D(u) u \dif u \gtrsim \beta^2 \wedge \beta.
\end{align*}
Let $g(u) \deq \sqrt{\frac{1}{1-u^2}(d(\frac{1+u}{2}\|\frac{1}{2})-\frac{1}{2}\log 2)_{+}}$.
Then $\alpha_D(u)= \frac{g(u)}{\beta} \wedge 1 $. It is easy to calculate that $g(4 / 5) >0$.
Notice that
\begin{itemize}
  \item If $\beta < g(4 / 5)$,
    \begin{align*}
    \beta^2 \int_0^1 \alpha_D(u) u \dif u 
    \ge \beta^2 \int_{4 /5}^1 \left(\frac{g(u)}{\beta} \wedge 1\right)  u \dif u 
    \ge \beta^2 \int_{4 / 5}^1 u \dif u \gtrsim \beta^2 .
    \end{align*}
  \item If $\beta \ge g (4 / 5)$,
    \begin{align*}
      \beta^2 \int_0^1 \alpha_D(u) u \dif u
      \ge \beta^2 \int_{4 / 5}^1 \left(\frac{g(u)}{\beta} \wedge 1\right) u \dif u
      \ge \beta g(4 / 5) \int_{4 / 5}^1  u \dif u \gtrsim \beta.
    \end{align*}
\end{itemize}
Therefore, the proof can be concluded as follows:
\begin{align*}
\sP(\beta)
&\le \sP(\alpha_D, \beta) 
= f_{\alpha_D,\beta}(0,0) - \beta^2 \int_0^1 \alpha_D(u) u \dif u
\lesssim \beta^2 \wedge \beta +\beta^2 \int_0^1 \alpha_D(u) u \dif u \\
&\lesssim \beta^2 \int_0^1 \alpha_D(u) u \dif u 
\lesssim \frac{\beta}{N} \eF_\beta(X^{(N)}; \su^{(N)}),
\end{align*}
as claimed.
\end{proof}

\appendix

\section{Omitted proofs}
\label{app:proofs}

\begin{proof}[Proof of Proposition~\ref{prop:bd_tensor}]

Consider some $\mu\in \eP(T)$. Then
\begin{align*}
  \frac{1}{N}\eF_\beta(\bm{X}^N, \mu^{\otimes N})
  &= \frac{1}{N \beta} \Ex\,{\log \sum_{\bm{t}\in T^N} 
  \mu^{\otimes N}(\bm{t} )e^{\beta \bm{X_t}}}\\
  &= \frac{1}{N \beta} \Ex\,{\log \sum_{\bm{t}\in S_{\mu}^N} 
  \mu^{\otimes N}(\bm{t} )e^{\beta \bm{X_t}}}\\
  &= \frac{1}{N \beta} \Ex\, { \log \sum_{\nu\in \cP_{N,\mu}} \sum_{\bm{t} \in \cC_{N,\nu}} 
  \prod_{t\in S_{\mu}} \mu(t)^{N \nu(t)} e^{\beta \bm{X_t}}}\\
  &= \frac{1}{N \beta} \Ex\, { \log \sum_{\nu\in \cP_{N,\mu}} \sum_{\bm{t} \in \cC_{N,\nu}} 
  e^{\beta \bm{X_t} + N \langle\nu, \log \mu\rangle}}\\
\end{align*}
Let
\begin{align*}
  \xi_{\mu}(\nu) \deq \frac{1}{\beta}  \log \sum_{\bm{t} \in \cC_{N,\nu}} 
  e^{\beta \bm{X_t} + \langle\nu, \log \mu\rangle }
\end{align*}
By Lemma~\ref{lem:concentration}, $\xi_{\mu}(\nu)$ is a $\sqrt{2N}\Delta$-subgaussian.
Then, using the maximal inequality for subgaussian random variables,
\begin{align*}
  \frac{1}{N} \max_{\nu\in \cP_{N,\mu}}\Ex\xi_\mu(\nu)
&\le\frac{1}{N } \Ex\big[\max_{\nu\in \cP_{N,\mu}}\xi_{\mu}(\nu)\big]\\
&\le\frac{1}{N \beta} \Ex\Big[\log\sum_{\nu \in \cP_{N,\mu}} e^{\beta\xi_{\mu}(\nu)}\Big]\\
&\le \frac{1}{N}\Ex\big[\max_{\nu\in \cP_{N,\mu}} \xi_{\mu}(\nu)\big] + \frac{\log |\cP_{N,\mu}|}{N \beta}\\
&\le \frac{1}{N}\max_{\nu \in \cP_{N,\mu}} \Ex{\xi_{\mu}(\nu)}
+ \frac{2\Delta \sqrt{N\log |\cP_{N,\mu}|} }{N } + \frac{\log |\cP_{N,\mu}|}{N \beta}.
\end{align*}
Therefore, with the estimate $|\cP_{N,\mu}|\le |\cP_N | \le (N+1)^{|T|}$ (see e.g. \citet{csiszarInformationTheoryCoding2011}), we have 
\begin{align*}
  \frac{1}{N}\eF_\beta(\bm{X}^N, \mu^{\otimes N})
  = \frac{1}{N} \max_{\nu\in \cP_{N,\mu}} \Ex{\xi_{\mu}(\nu)} + o_N(1).
\end{align*}
Substituting the expression of $\xi_\mu(\nu)$,
\begin{align*}
 \frac{1}{N}\eF_\beta(\bm{X}^N, \mu^{\otimes N})
  &\le \frac{1}{N}\max_{\nu\in \cP_{N,\mu}} \Ex{\xi_{\mu}(\nu)} + o_N(1)\\
  &= \max_{\nu\in\cP_{N,\mu}} \set[\Big]{\frac{1}{N\beta}\Ex{\Psi_\beta(\bm{X}^N; \cC_{N,\nu})} 
  +  \frac{\langle\nu, \log \mu\rangle}{\beta} } + o_N(1)\\
  &\le \max_{\nu\in\cP_{N,\mu}} \set[\Big]{\frac{1}{N\beta}
\Ex{\Psi_\beta(\bm{X}^N; \cC_{N,\nu,\epsilon})} 
+  \frac{\langle\nu, \log \mu\rangle}{\beta} } + o_N(1)\\
  &\le \sup_{\nu\in \eP_{\mu}(T)}\set[\Big]{\frac{1}{N}
  \Ex{\Psi_{\beta}(\bm{X}^N; \cC_{N,\nu,\epsilon} )} 
   +  \frac{\langle\nu, \log \mu\rangle}{\beta}} + o_N(1).
\end{align*}
Thus,
\begin{align*}
 \eF_\beta(X;\mu)
 \le \liminf_{\epsilon\downarrow 0} \liminf_{N \to \infty}
   \sup_{\nu\in \eP_{\mu}(T)}\set[\Bigg]{\frac{1}{N}\Ex{\Psi_{\beta}(\bm{X}^N; \cC_{N,\nu,\epsilon} )} 
   +  \frac{\langle\nu, \log \mu\rangle}{\beta}} .
\end{align*}
The next step is to obtain a lower bound for $\mu\in \eP_+(T)$. We begin by writing
\begin{align*}
\eF_\beta(X; \mu)
&= \frac{1}{N} \eF_\beta(\bm{X}^N, \mu^{\otimes N})\\
&= \frac{1}{N \beta} \Ex\Big[ \log \sum_{\bm{t}\in T^N}
  \mu^{\otimes N}(\bm{t}) e^{\beta \bm{X_t}}\Big]\\
&\ge \frac{1}{N \beta} \Ex\Big[ \log \sum_{\bm{t}\in \cC_{N,\nu,\epsilon}} 
\mu^{\otimes N}(\bm{t})e^{\beta \bm{X_t}}\Big].
\end{align*}
Since $\mu \in \eP_+(T)$, $\lVert \log \mu \rVert_{\infty} < \infty$.
Using this and the fact that 
$\lVert\hat{P}_{\bm{t}}-\nu\rVert_{\infty} \le 2 \epsilon$ for any $\bm{t}\in \cC_{N,\nu,\epsilon}$, we have
\begin{align*}
\eF_\beta(X;\mu) 
&\ge \frac{1}{N\beta} \Ex\Big[\log \sum_{\bm{t}\in \cC_{N, \nu,\epsilon}}
\prod_{t\in T} \mu(t)^{N(\nu(t)-2\epsilon)} e^{\beta \bm{X_t}}\Big]\\
&\ge \frac{1}{N \beta} \Ex\Big[\log \sum_{\bm{t}\in \cC_{N, \nu, \epsilon}} 
  e^{N \langle \nu, \log \mu \rangle-2 N \epsilon \lVert \log \mu \rVert_{\infty}
\abs{T}} e^{\beta \bm{X_t}}\Big]\\
&= \frac{1}{N} \Ex{\Psi_\beta (\bm{X}^N; \cC_{N, \nu, \epsilon}) }
+ \frac{N \langle \nu, \log \mu \rangle- 2N \epsilon 
\lVert \log \mu \rVert_{\infty} \abs{T}}{N \beta} 
\end{align*}
Therefore,
\begin{align*}
\eF_\beta(X;\mu)
&\ge \limsup_{\epsilon\downarrow 0}\limsup_{N \to \infty} \sup_{\nu\in \eP(T)}
\set[\Bigg]{\frac{1}{N} \Ex{\Psi_\beta(\bm{X}^N ; \cC_{N, \nu, \epsilon})}
  + \frac{\langle \nu, \log \mu \rangle}{\beta} 
+ O(\epsilon)}.
\end{align*}
\end{proof}

\begin{proof}[Proof of Proposition~\ref{prop:limit}]
Fix some $\nu \in \eP(T)$. We first show that the limits as $N \to \infty$ exist for each $\epsilon \in (0, 1-\frac{1}{\abs{T}})$ 
pointwise. We will use the following shorthand notation:
\begin{align*}
  \cC_N &\deq \cC_{N, \nu, \epsilon}\\
  Z_N &\deq \sum_{\bm{t}^N\in \cC_N} e^{\beta \bm{X}^N_{\bm{t}^N}} \\
  a_N &\deq \frac{1}{N} \Ex \log Z_N.
\end{align*}
We first show that $Na_N$ is superadditive. For $N, M \ge 1$,
let $\bm{X}^N, \bm{X}^M, \bm{X}^{N+M}$ be independent.
By definition of tensorized processes, we can concatenate $\bm{X}_{\bm{t}^N}^N$ 
and $\bm{X}_{\bm{t}^M}^M$.
For $\bm{t}^N \in \cC_N$ and $\bm{t}^{M}\in \cC_{M}$,
\begin{align*}
  &\lVert\hat{P}_{\bm{t}^N}-\nu\rVert_{{\rm TV}}\le \epsilon
  \quad\text{and}\quad\lVert{\hat{P}_{\bm{t}^M}-\nu\rVert_{{\rm TV}}}\le \epsilon
\end{align*}
Then by concatenation,
\begin{align*}
  \hat{P}_{(\bm{t}^N, \bm{t}^M)} = \frac{N}{N+M} 
  \hat{P}_{\bm{t}^N} + \frac{M}{N+M} \hat{P}_{\bm{t}^M}
  \implies \lVert\hat{P}_{(\bm{t}^N, \bm{t}^M)}-\nu\rVert \le \epsilon.
\end{align*}
Thus,
\begin{align*}
  (M+N)a_{M+N}
  &=\Ex\,{\log \sum_{\bm{t}^{N+M}\in \cC_{N+M} } e^{\beta \bm{X}_{\bm{t}^{M+N}}^{M+N}}}
\ge \Ex\,{\log \sum_{\bm{t}^N\in \cC_N}\sum_{\bm{t}^M\in \cC_M} 
e^{\beta \bm{X}_{\bm{t}^M}^M + \beta \bm{X}_{\bm{t}^N}^N}}\\
&= \Ex\,{\log Z_N } + \Ex\,{\log Z_M} = Na_N + Ma_M
\end{align*}
Therefore, the limit $\lim_{N \to \infty} a_N = \lim_{N \to \infty}\frac{1}{N} \Ex \log Z_N$ exists by Fekete's lemma
(see, e.g., \citet[Theorem~1.1]{panchenko_sherrington-kirkpatrick_2013}).

Now we show the existence of the limit
 $\lim_{N \to \infty} \frac{1}{N} \log \abs{\cC_{N, \nu, \epsilon}}$.  Recall that $\cP_N \subset \eP(T)$ is set of types of length $N$. 
%We suppress $N$ for brevity. 
 The following continuity property of $H_\epsilon$ (see Proposition~\ref{prop:cont_H_eps} for the definition) will be useful:
\begin{align*}
\max_{\rho \in B_{{\rm TV}}(\nu,\epsilon)\cap \cP_N} \log \abs{\cC_{N,\rho}}
&\le \log \abs{\cC_{N, \nu, \epsilon}} \le {\abs{T}}\log (N+1) 
+ \max_{\rho \in B_{{\rm TV}}(\nu,\epsilon)\cap \cP_N } \log \abs{\cC_{N,\rho}}\\
&\le \sup_{\rho \in B_{{\rm TV}}(\nu,\epsilon)} N H(\rho) + O(\log N) 
= N H_{\epsilon}(\nu) + O(\log N)
\end{align*}
where the last inequality uses $\abs{N H(\rho) - \log \abs{\cC_{N,\rho}}}\le O(\log N)$ for $\rho\in \cP_N$ 
(see, e.g., \citet{csiszarInformationTheoryCoding2011}).
Thus,
\begin{align*}
\limsup_{N \to \infty}\frac{1}{N} \log \abs{C_{N, \nu, \epsilon}} \le H_\epsilon(\nu).
\end{align*}

Fix some $0< \delta<\epsilon$. By the definition of $H_\epsilon$, we can find some $\rho \in B_{\rm TV}(\nu,\epsilon-\delta)$  so that
\begin{align*}
H(\rho ) \ge H_{\epsilon-\delta}(\nu)- \delta .
\end{align*}
By \eqref{eq:round}, there exists some $\hat{\rho} \in \cP_N$, so that $\lVert\rho- \hat{\rho}\rVert_{{\rm TV}}   \le \frac{\abs{T}}{N}$.
Then 
\begin{align*}
  \lVert\hat{\rho}-\nu\rVert_{{\rm TV}} \le \lVert\hat{\rho} - \rho\rVert_{{\rm TV}} + \lVert \rho-\nu \rVert_{{\rm TV}}
    \le \epsilon - \delta + \frac{\abs{T}}{N} 
    \le \epsilon\quad \text{for } N \ge \frac{\abs{T}}{\delta}.
\end{align*}
Therefore, for $N \ge \abs{T} / \delta$ and $\abs{T} / N \le 1 - 1 / \abs{T}$,
\begin{align*}
 \frac{1}{N} \log \abs{\cC_{N, \nu, \epsilon}} 
 &\ge \frac{1}{N} \log \abs{ \cC_{N, \hat{\rho}}} \\
&\ge  H(\hat{\rho}) - O(N^{-1}\log N )\\
&\ge H(\rho) - o_N(1)\\
&\ge H_{\epsilon-\delta}(\nu) - \delta -  o_N(1).
\end{align*}
Taking $N \to \infty$ and using Proposition~\ref{prop:cont_H_eps} yields
\begin{align*}
\liminf_{N \to \infty} \frac{1}{N} \log \abs{\cC_{ \nu, \epsilon}} \ge H_{\epsilon}(\nu). 
\end{align*}
Therefore, for any $\nu \in \eP(T)$ and for any sufficiently small $\epsilon>0$,
\begin{align*}
&\lim_{N \to \infty} \frac{1}{N} \log \abs{\cC_{N, \nu, \epsilon}} = H_{\epsilon}(\nu),\\
&\lim_{N\to \infty} \frac{1}{N}\eF_\beta(\bm{X}^N, \su_{\cC_{N,\nu,\epsilon}})
=\bm{\psi}_{\beta,\epsilon}(\nu)- \frac{1}{\beta} H_{\epsilon}(\nu).
\end{align*}
Since $\epsilon \mapsto \bm{\psi}_{\beta, \epsilon}(\nu)$ is nonnegative and monotone increasing,
$\lim_{\epsilon \downarrow 0}\bm{\psi}_{\beta,\epsilon}(\nu)$  
exists.
Also, by continuity of entropy in \eqref{eq:cont_ent} 
\begin{align*}
\lim_{\epsilon \downarrow 0} H_{\epsilon}(\nu)= H(\nu) .
\end{align*}
This completes the proof. \end{proof}

\section{Miscellaneous technical results}
\label{app:misc}

\setcounter{proposition}{0}
\renewcommand{\theproposition}{\Alph{section}.\arabic{proposition}}

The following result is useful for converting between certain integrals and geometrically weighted discrete sums:

\begin{proposition}\label{prop:disc_to_int}
  If $f:[0,1] \to [0,\infty)$ is non-increasing
  and $g:[0,1]\to [0,\infty)$ is affine and non-decreasing, for $r \ge 2$,
  \begin{align}
  \label{eq:sum_int}
  \frac{1}{r}\sum_{k\ge 1} r^{-k} (f(r^{-k}) \wedge g(r^{-k})) 
  \le \int_0^1 (f(\epsilon)\wedge g(\epsilon)) \dif \epsilon 
  \le r^2 \sum_{k \ge 1} r^{-k} (f(r^{-k}) \wedge g(r^{-k})) 
  \end{align}

\end{proposition}

\begin{proof}
Since $f$ is nonincreasing and $g$ is affine, nondecreasing, and nonnegative (and thus has the form $g(r) = ar + b$ with $a,b \ge 0$), for any $r \ge 2$ we have
\begin{align*}
  &\quad\sum_{k \ge 1} r^{-k} (f(r^{-k})\wedge g(r^{-k})) 
  \le \sum_{k \ge 1} r^{-k} (r-1) r \cdot 
  \frac{1}{r} (f(r^{-k})\wedge g(r^{-k})) \\
  &\le \sum_{k \ge 1} r^{-k} (r-1) r  (f(r^{-k})\wedge g(r^{-k-1})) 
  \le r\sum_{k \ge 1} \int_{r^{-k-1}}^{r^{-k}} 
  (f(\epsilon)\wedge g(\epsilon))  \dif \epsilon\\
  &\le r\int_{0}^{1} 
  (f(\epsilon)\wedge g(\epsilon))  \dif \epsilon
  = r \sum_{k \ge 1} \int_{r^{-k}}^{r^{-k+1}} 
  (f(\epsilon)\wedge g(\epsilon))  \dif \epsilon\\
  &\le r \sum_{k \ge 1} (r^{-k+1}-r^{-k}) (f(r^{-k})\wedge g(r^{-k+1})) 
  \le r \sum_{k \ge 1} (r^{-k+1}-r^{-k}) (f(r^{-k})\wedge r g(r^{-k})) \\
  &\le r^{3} \sum_{k \ge 1} (f(r^{-k})\wedge g(r^{-k})) .
\end{align*}
Dividing through by $r$ gives the desired inequalities. 
\end{proof}

\noindent The next result gives upper and lower bounds on truncated binomial sums in terms of Shannon entropies; it is useful for estimating volumes of Hamming balls. (See, e.g., Lemma~8 on p.~310 of \cite{macwilliams_1977} for a more refined estimate.)

\begin{proposition}\label{prop:binom_est}
  For some $m,n\in \bN$, if $p \deq m/n \le \frac{1}{2}$, then 
\begin{align*}
 \frac{1}{n+1}e^{n h(p)} \le
\sum_{i=0}^m \binom{n}{i}
\le e^{n h(p)}
\end{align*}
where $h(p) \deq -p \log p -(1-p)\log (1-p)$ is the binary Shannon entropy in nats.
\begin{proof}
  Let $V_m \deq \sum_{i=0}^{m} \binom{n}{i}$. Observe that for $i \le n$ 
\begin{align*}
\frac{(1-p)^{n-i}p^i}{(1-p)^{n-m}p^m}
= (1-p)^{m-i} p^{i-m}=\Big(\frac{1}{p}-1\Big)^{m-i}
\end{align*}
is decreasing as $i$ increases. Therefore,
\begin{align*}
  1= (1-p+p)^{n} \ge \sum_{i=1}^m (1-p)^{n-i}p^{i}\binom{n}{i}
   \ge (1-p)^{n-m}p^{m} V_m = e^{-n h(p)}V_m 
\end{align*}
yields the upper bound. For lower bound,
since for binomial distribution, its mode is $\floor{(n+1)p}=m$, i.e., 
for any $i \le n$,
\begin{align*}
  (1-p)^{n-i}p^{i}\binom{n}{i} \le (1-p)^{n-m}p^{m}\binom{n}{m} 
  = e^{-n h(p)}\binom{n}{m}
\end{align*}
Summing over $i$ gives the lower bound.
\end{proof}
\end{proposition}

\noindent The  result below is a special case of the decorrelation lemma of \citet{chu_raginsky_unified_2023}:

\begin{proposition}\label{prop:decorr} Let $\mu$ and $\nu$ be two probability measures on a Borel space $\eX$ such that $\mu \ll \nu$, and let $f,g : \eX \to \Reals_+$ be two nonnegative measurable functions. Then the following inequality holds:
	\begin{align*}
		\langle \mu, fg \rangle \le \Bigg\langle \mu, f\sqrt{2\log\Big(\frac{\dif\mu}{\dif\nu}+1\Big)} \Bigg\rangle + \langle \nu, f\cdot (\exp(g^2)-1)\rangle.
	\end{align*}
\end{proposition}

\section*{Acknowledgments} This
work was supported in part by the NSF under the award CCF--2348624 (``Towards a control framework for neural generative modeling").

\bibliography{softmax_gaussian.bbl}

\begin{thebibliography}{32}
\providecommand{\natexlab}[1]{#1}
\providecommand{\url}[1]{\texttt{#1}}
\expandafter\ifx\csname urlstyle\endcsname\relax
  \providecommand{\doi}[1]{doi: #1}\else
  \providecommand{\doi}{doi: \begingroup \urlstyle{rm}\Url}\fi

\bibitem[Adler and Taylor(2007)]{Adler2007}
Robert~J. Adler and Jonathan~E. Taylor.
\newblock \emph{Random Fields and Geometry}.
\newblock Springer, 2007.

\bibitem[Auffinger and Chen(2015)]{auffinger_parisi_2015}
Antonio Auffinger and Wei-Kuo Chen.
\newblock The {{Parisi}} formula has a unique minimizer.
\newblock \emph{Communications in Mathematical Physics}, 335\penalty0
  (3):\penalty0 1429--1444, 2015.

\bibitem[Ben~Arous et~al.(2005)Ben~Arous, Bogachev, and
  Molchanov]{ben_arous_2005}
G{\'e}rard Ben~Arous, Leonid~V. Bogachev, and Stanislav~A. Molchanov.
\newblock Limit theorems for sums of random exponentials.
\newblock \emph{Probability Theory and Related Fields}, 132\penalty0
  (4):\penalty0 579--612, 2005.

\bibitem[Bovier(2006)]{bovier_statistical_2006}
Anton Bovier.
\newblock \emph{Statistical {{Mechanics}} of {{Disordered Systems}}: {{A
  Mathematical Perspective}}}.
\newblock Cambridge University Press, 2006.

\bibitem[Boyd and Vandenberghe(2004)]{boyd_vanderberghe_cvx}
Stephen Boyd and Lieven Vandenberghe.
\newblock \emph{Convex Optimization}.
\newblock Cambridge University Press, 2004.

\bibitem[Chatterjee(2014)]{Chatterjee_2014}
Sourav Chatterjee.
\newblock \emph{Superconcentration and Related Topics}.
\newblock Springer, 2014.

\bibitem[Chu and Raginsky(2023)]{chu_raginsky_unified_2023}
Yifeng Chu and Maxim Raginsky.
\newblock A unified framework for information-theoretic generalization bounds.
\newblock In A.~Oh, T.~Naumann, A.~Globerson, K.~Saenko, M.~Hardt, and
  S.~Levine, editors, \emph{Advances in Neural Information Processing Systems},
  volume~36, pages 79260--79278. Curran Associates, Inc., 2023.

\bibitem[Chu and Raginsky(2026)]{chu_raginsky_alt}
Yifeng Chu and Maxim Raginsky.
\newblock {Talagrand} meets {Talagrand}: Upper and lower bounds on expected
  soft maxima of {Gaussian} processes with finite index sets.
\newblock In \emph{Conference on Algorithmic Learning Theory}, 2026.
\newblock URL \url{https://arxiv.org/abs/2502.06709}.

\bibitem[Cover and Thomas(2006)]{cover_thomas}
Thomas~M. Cover and Joy~A. Thomas.
\newblock \emph{Elements of Information Theory}.
\newblock Wiley, 2nd edition, 2006.

\bibitem[Csisz{\'a}r and K{\"o}rner(2011)]{csiszarInformationTheoryCoding2011}
Imre Csisz{\'a}r and J{\'a}nos K{\"o}rner.
\newblock \emph{Information {{Theory}}: {{Coding Theorems}} for {{Discrete
  Memoryless Systems}}}.
\newblock Cambridge University Press, 2011.

\bibitem[Dobrushin(1960)]{dobrushin_1960}
Roland~L. Dobrushin.
\newblock Passage to the limit under the information and entropy signs.
\newblock \emph{Theory of Probability and Its Applications}, 5\penalty0
  (1):\penalty0 25--32, 1960.

\bibitem[Fernique(1975)]{Fernique_1975}
Xavier Fernique.
\newblock Regularit\'e des trajectoires des fonctions al\'eatoires gaussiennes.
\newblock In \emph{Ecole d'Et{\'{e}} de Probabilit{\'{e}}s de Saint-Flour
  {IV}{\textemdash}1974}, pages 1--96. Springer, 1975.

\bibitem[Fernique(1978)]{Fernique_1978}
Xavier Fernique.
\newblock Caract\'erisation de processus \`a trajectoires major\'ees ou
  continues.
\newblock \emph{S\'eminaire de probabilit\'es de Strasbourg}, 12:\penalty0
  691--706, 1978.

\bibitem[Gu\'edon and Zvavitch(2003)]{guedon_supremum_2003}
Olivier Gu\'edon and Artem Zvavitch.
\newblock Supremum of a {{Process}} in {{Terms}} of {{Trees}}.
\newblock In \emph{{{Geometric}} Aspects of Functional Analysis}, pages
  136--147. {Springer}, 2003.

\bibitem[Guerra(2003)]{guerraBrokenReplicaSymmetry2003}
Francesco Guerra.
\newblock Broken {{Replica Symmetry Bounds}} in the {{Mean Field Spin Glass
  Model}}.
\newblock \emph{Communications in Mathematical Physics}, 233\penalty0
  (1):\penalty0 1--12, 2003.

\bibitem[Kieffer(1991)]{kiefferSampleConversesSource1991}
John~C. Kieffer.
\newblock Sample converses in source coding theory.
\newblock \emph{IEEE Transactions on Information Theory}, 37\penalty0
  (2):\penalty0 263--268, 1991.

\bibitem[Kontoyiannis and Zhang(2002)]{kontoyiannis_arbitrary_2002}
Ioannis Kontoyiannis and Junshan Zhang.
\newblock Arbitrary source models and {{Bayesian}} codebooks in rate-distortion
  theory.
\newblock \emph{IEEE Transactions on Information Theory}, 48\penalty0
  (8):\penalty0 2276--2290, 2002.

\bibitem[Liu(2025)]{liuSimpleSharpGeneralization2025}
Jingbo Liu.
\newblock Simple and {{Sharp Generalization Bounds}} via {{Lifting}}, September
  2025.
\newblock URL \url{https://arxiv.org/abs/2508.18682}.

\bibitem[{MacWilliams} and Sloane(1977)]{macwilliams_1977}
Florence~J. {MacWilliams} and Neil~J.A. Sloane.
\newblock \emph{The Theory of Error-Correcting Codes}.
\newblock North-Holland, Amsterdam, 1977.

\bibitem[Palaiyanur and Sahai(2008)]{palaiyanur_sahai_2008}
Hari Palaiyanur and Anant Sahai.
\newblock On the uniform continuity of the rate-distortion function.
\newblock In \emph{2008 IEEE International Symposium on Information Theory},
  pages 857--861, 2008.

\bibitem[Panchenko(2013)]{panchenko_sherrington-kirkpatrick_2013}
Dmitry Panchenko.
\newblock \emph{The {{Sherrington-Kirkpatrick Model}}}.
\newblock Springer {{Monographs}} in {{Mathematics}}. {Springer}, 2013.

\bibitem[Panchenko(2014)]{panchenko_parisi_2014}
Dmitry Panchenko.
\newblock The {{Parisi}} formula for mixed $p$-spin models.
\newblock \emph{The Annals of Probability}, 42\penalty0 (3), 2014.

\bibitem[Polyanskiy and Wu(2024)]{polyanskiyInformationTheoryCoding2024}
Yury Polyanskiy and Yihong Wu.
\newblock \emph{Information {{Theory}}: {{From Coding}} to {{Learning}}}.
\newblock Cambridge University Press, 2024.

\bibitem[Russo and Zou(2016)]{russo_controlling_2016}
Daniel Russo and James Zou.
\newblock Controlling bias in adaptive data analysis using information theory.
\newblock In Arthur Gretton and Christian~C. Robert, editors, \emph{Proceedings
  of the 19th International Conference on Artificial Intelligence and
  Statistics}, volume~51 of \emph{Proceedings of Machine Learning Research},
  pages 1232--1240, Cadiz, Spain, 2016. PMLR.

\bibitem[Sason(2013)]{sasonEntropyBoundsDiscrete2013}
Igal Sason.
\newblock Entropy {{Bounds}} for {{Discrete Random Variables}} via {{Maximal
  Coupling}}.
\newblock \emph{IEEE Transactions on Information Theory}, 59\penalty0
  (11):\penalty0 7118--7131, 2013.

\bibitem[Talagrand(1987)]{Talagrand_1987}
Michel Talagrand.
\newblock Regularity of gaussian processes.
\newblock \emph{Acta Mathematica}, 159:\penalty0 99--149, 1987.

\bibitem[Talagrand(1992)]{Talagrand_1992}
Michel Talagrand.
\newblock A simple proof of the majorizing measure theorem.
\newblock \emph{Geometric and Functional Analysis}, 2\penalty0 (1):\penalty0
  118--125, March 1992.

\bibitem[Talagrand(2006)]{talagrand_parisi_2006}
Michel Talagrand.
\newblock The {{Parisi}} formula.
\newblock \emph{Annals of Mathematics}, 163\penalty0 (1):\penalty0 221--263,
  2006.

\bibitem[Talagrand(2011)]{Talagrand_MF}
Michel Talagrand.
\newblock \emph{Mean Field Models for Spin Glasses, Volume I: Basic Examples}.
\newblock Springer, 2011.

\bibitem[Talagrand(2014)]{talagrand_2014}
Michel Talagrand.
\newblock \emph{Upper and Lower Bounds for Stochastic Processes: Modern Methods
  and Classical Problems}.
\newblock Springer, 2014.

\bibitem[van Handel(2025)]{vanhandel_subgaussian}
Ramon van Handel.
\newblock On the subgaussian comparison theorem, 2025.
\newblock URL \url{https://arxiv.org/abs/2512.18588}.

\bibitem[Xu and Raginsky(2017)]{xu_raginsky_2017}
Aolin Xu and Maxim Raginsky.
\newblock Information-theoretic analysis of generalization capability of
  learning algorithms.
\newblock In I.~Guyon, U.~Von Luxburg, S.~Bengio, H.~Wallach, R.~Fergus,
  S.~Vishwanathan, and R.~Garnett, editors, \emph{Advances in Neural
  Information Processing Systems}, volume~30. Curran Associates, Inc., 2017.

\end{thebibliography}
\end{document}